\documentclass[12pt,reqno]{amsart}  
\usepackage[margin=1.1in]{geometry} 
\usepackage{amssymb}                
\usepackage{mathrsfs}               
\usepackage[T1]{fontenc}            
\usepackage{tgtermes}               
\usepackage{graphicx}               

\newcommand{\RR}{\mathbb{R}}
\newcommand{\wRR}[1]{\mathop{\widehat{\mathbb{R}}^{#1}}\nolimits}
\newcommand{\Leb}{\mathscr{L}}
\newcommand{\Haus}{\mathscr{H}}

\newcommand{\tH}{\widetilde{H}}
\newcommand{\loc}{\mathrm{loc}}

\newcommand{\betac}{\beta^{\mathrm{ctr}}}
\newcommand{\thetahd}{\theta^{\mathrm{HD}}}
\newcommand{\oB}{B_\circ}

\newcommand{\res}{\hbox{ {\vrule height .22cm}{\leaders\hrule\hskip.2cm} } }

\newcommand{\Xint}[1]{\mathchoice
    {\XXint\displaystyle\textstyle{#1}}%
    {\XXint\textstyle\scriptstyle{#1}}%
    {\XXint\scriptstyle\scriptscriptstyle{#1}}%
    {\XXint\scriptscriptstyle\scriptscriptstyle{#1}}%
    \!\int}
\newcommand{\XXint}[3]{\setbox0=\hbox{$#1{#2#3}{\int}$}
    \vcenter{\hbox{$#2#3$}}\kern-.5\wd0}
\newcommand{\avint}{\Xint-}

\newcommand{\co}{\mathop\mathrm{co}\nolimits}
\newcommand{\diam}{\mathop\mathrm{diam}\nolimits}
\newcommand{\dist}{\mathop\mathrm{dist}\nolimits}
\newcommand{\esssup}{\mathop\mathrm{ess\,sup}}
\newcommand{\HD}{\mathop\mathrm{HD}\nolimits}

\newcommand{\Amin}[2]{\min\{\|A'_{#1}\|,\|A'_{#2}\|\}}
\newcommand{\Amax}[2]{\max\{\|A'_{#1}\|,\|A'_{#2}\|\}}
\newcommand{\Adiff}[2]{\|A'_{#1}-A'_{#2}\|}
\newcommand{\Anorm}[1]{\|A'_{#1}\|}

\newtheorem{theorem}{Theorem}[section]
\newtheorem{corollary}[theorem]{Corollary}
\newtheorem{lemma}[theorem]{Lemma}
\newtheorem{step}{Step}

\theoremstyle{definition}
\newtheorem{definition}[theorem]{Definition}
\newtheorem{hypothesis}[theorem]{Hypothesis}

\theoremstyle{remark}
\newtheorem{remark}[theorem]{Remark}

\newtheorem{case}{Case}

\numberwithin{equation}{section}
\numberwithin{figure}{section}

\begin{document}

\date{March 12, 2014}
\title[Quasiconformal planes and almost affine maps]{Quasiconformal planes with bi-Lipschitz pieces\\ and extensions of almost affine maps}
\thanks{The authors were partially supported: J.~ Azzam by NSF DMS RTG 08-38212, M.~ Badger by an NSF postdoctoral fellowship DMS 12-03497, and T.~ Toro by NSF DMS 08-56687 and a grant from the Simons Foundation \#228118. A portion of this research was completed while the authors visited the Institute for Pure and Applied Mathematics for the long program on Interactions Between Analysis and Geometry in the spring of 2013.}
\author{Jonas Azzam}
\address{Department of Mathematics\\ University of Washington\\ Box 354350\\ Seattle, WA 98195-4350}
\email{jonasazzam@math.washington.edu}
\author{Matthew Badger}
\address{Department of Mathematics\\ Stony Brook University\\ Stony Brook, NY 11794-3651}
\email{badger@math.sunysb.edu}
\author{Tatiana Toro}
\address{Department of Mathematics\\ University of Washington\\ Box 354350\\ Seattle, WA 98195-4350}
\email{toro@math.washington.edu}
\subjclass[2010]{Primary 30C65. Secondary 28A75, 54C20.}
\keywords{quasiconformal maps, quasisymmetric maps, almost affine maps, extension theorems, quasiplanes, rectifiable sets, big pieces of bi-Lipschitz images, Reifenberg flat sets, Jones beta numbers}

\begin{abstract} A quasiplane $f(V)$ is the image of an $n$-dimensional Euclidean subspace $V$ of $\RR^N$ ($1\leq n\leq N-1$) under a quasiconformal map $f:\RR^N\rightarrow\RR^N$ . We give sufficient conditions in terms of the weak quasisymmetry constant of the underlying map for a quasiplane to be a bi-Lipschitz $n$-manifold and for a quasiplane to have big pieces of bi-Lipschitz images of $\RR^n$. One main novelty of these results is that we analyze quasiplanes in arbitrary codimension $N-n$. To establish the big pieces criterion, we prove new extension theorems for ``almost affine'' maps, which are of independent interest. This work is related to investigations by Tukia and V\"ais\"al\"a on extensions of quasisymmetric maps with small distortion.
\end{abstract}

\maketitle

\setcounter{tocdepth}{1}
\tableofcontents


\section{Introduction}
\label{s:intro}

The quasiconformal maps of Euclidean space (whose precise definition is deferred until \S2) are a class of homeomorphisms $f:\RR^N\rightarrow\RR^N$ ($N\geq 2$) with several nice properties: \begin{itemize}
\item $f$ maps balls onto regions with uniformly bounded eccentricity ($f$ is quasisymmetric);
\item $f$ is differentiable at Lebesgue almost every $x\in\RR^N$; and
\item $f$ maps sets of Lebesgue measure zero onto sets of Lebesgue measure zero.
\end{itemize} Nevertheless, quasiconformal maps may distort geometric characteristics of lower dimensional sets in $\RR^N$ such as Hausdorff dimension, Hausdorff measure, and rectifiability. For example, there exist quasiconformal maps of the plane that map the unit circle onto the Koch snowflake. It is natural to ask, therefore, under which circumstances---and to what extent---can one control the distortion of geometry by quasiconformal maps. This question has been studied from a number of viewpoints by several authors,
 see e.g.~
\cite{Astala-area},
\cite{Heinonen-theorem-of-Semmes},
\cite{Semmes-question-of-Heinonen},
\cite{B99-dim},
\cite{DT-snow},
\cite{Rohde-qc},
\cite{quasiplanes1},
\cite{quasiplanes2},
\cite{Prause},
\cite{KO09},
\cite{LSU2010},
\cite{Meyer-snowballs},
\cite{Smirnov},
\cite{PTU12},
\cite{ACTUV},
\cite{BTW13},
\cite{BGRT},
\cite{VW-Jordan},
\cite{A-qdiff},
\cite{BH-freq},
and the references therein.

In this paper, we find conditions that ensure that a quasiplane is rectifiable, or that at least, ensure that a quasiplane contains nontrivial rectifiable subsets. A \emph{quasiplane} is the image $f(V)$ of an $n$-dimensional Euclidean subspace $V\subset\RR^N$ ($1\leq n\leq N-1$) under a quasiconformal map $f:\RR^N\rightarrow\RR^N$. When $n=1$, a quasiplane $f(V)$ is also called a \emph{quasiline}. When $n=N-1$, a quasiplane $f(V)$ is the unbounded variant of a \emph{quasisphere} $g(S^{N-1})$, which is the image of the unit sphere $S^{N-1}$ under a quasiconformal map $g:\RR^N\rightarrow\RR^N$. A set $X\subset\RR^N$ is \emph{$n$-rectifiable} (in the sense of geometric measure theory, e.g.~ see \cite{Mattila}) if there exist countably many Lipschitz maps $f_i:[0,1]^n\rightarrow\RR^N$ whose images cover $\Haus^n$-almost all of $X$, that is, \begin{equation*}\Haus^n\left(X\setminus{\bigcup_i\nolimits}\, f_i([0,1]^n)\right)=0,\end{equation*} where $\Haus^n$ denotes $n$-dimensional Hausdorff measure on $\RR^N$. This notion of rectifiability can be strengthened or weakened in a variety of ways, a few of which will enter the discussion below. In particular, a set $X\subset\RR^N$ is \emph{locally $L$-bi-Lipschitz equivalent to subsets of} $\RR^n$ if for all $x_0\in X$ there exist $r>0$, a map $h:X\cap B^N(x_0,r)\rightarrow\RR^n$, and a constant $c>0$ such that \begin{equation}\label{e:i1} c|x-y| \leq |h(x)-h(y)| \leq Lc |x-y|\quad\text{for all }x,y\in X\cap B^N(x_0,r).\end{equation} We also say that $X$ is \emph{locally bi-Lipschitz equivalent to subsets of $\RR^n$} if the bi-Lipschitz constant $L$ in (\ref{e:i1}) is allowed to depend on $x_0$.

In \cite{BGRT}, the second and third named authors, jointly with James T.~ Gill and Steffen Rohde, gave sufficient conditions for a quasisphere $f(S^{N-1})$ to be locally bi-Lipschitz equivalent to subsets of $\RR^{N-1}$.  The conditions were given in terms of the maximal dilatation of $f$ \cite[Theorem 1.1]{BGRT} and in terms of the weak quasisymmetry constant of $f$ \cite[Theorem 1.2]{BGRT}. The latter condition can be reformulated for quasiplanes, as follows. For all $X\subset\RR^N$ and maps $f:X\rightarrow\RR^N$, the \emph{weak quasisymmetry constant} $H_f(X)\in[1,\infty]$ of $f$ in $X$ is the least constant such that for all $x,y,a\in X$, \begin{equation*} |x-a|\leq |y-a| \Longrightarrow |f(x)-f(a)|\leq H_f(X) |f(y)-f(a)|.\end{equation*} In order to simplify several expressions below, we assign \begin{equation*}\tH_f(X):= H_f(X)-1.\end{equation*} For all $1\leq n\leq N$, we identify the Euclidean space $\RR^n$ with the subspace $\RR^{n}\times\{0\}^{N-n}$ of $\RR^N$. We let $B^n(x,r)$ and $\oB^n(x,r)$ denote, respectively, the \emph{closed} and \emph{open} ball in $\RR^n$ with center $x\in\RR^n$ and radius $r>0$. In addition, we let $\Leb^n$ denote Lebesgue measure on $\RR^n$ and we normalize $n$-dimensional Hausdorff measure $\Haus^n$ on $\RR^N$ so that $\Haus^n(B^n(0,1))=\Leb^n(B^n(0,1))$.

\begin{theorem}(\cite{BGRT}) \label{t:bgrt}
Suppose $1\leq n=N-1$. If $f:\RR^{N}\rightarrow \RR^{N}$ is quasiconformal and
\begin{equation} \label{e:bgrt-dini}
\int_{0}^{1}\sup_{x\in B^n(x_0,1)} \tH_{f}(B^N(x,r))^{2} \frac{dr}{r} <\infty \quad \text{for all $x_0\in\RR^n$},
\end{equation}
then the quasiplane $f(\RR^{n})$ is locally $(1+\delta)$-bi-Lipschitz equivalent to subsets of $\RR^n$ for all $\delta>0$. Thus, $f(\RR^n)$ is $n$-rectifiable and $\Haus^{n}\res f(\RR^n)$ (the restriction of $\Haus^n$ to $f(\RR^n)$) is locally finite.
\end{theorem}

The conclusion in Theorem \ref{t:bgrt} that $f(\RR^n)$ is locally $(1+\delta)$-bi-Lipschitz equivalent to subsets of $\RR^n$ for all $\delta>0$ is strictly weaker than $f(\RR^n)$ being locally $C^1$. However, if $1\leq n\leq N-1$ and the square Dini condition (\ref{e:bgrt-dini}) is replaced with a linear Dini condition, then the quasiplane $f(\RR^{n})$ is a $C^1$ embedded submanifold of $\RR^N$; see \cite[Chapter 7, \S4]{Reshetnyak}.

The first main result of this paper is to extend Theorem \ref{t:bgrt} to arbitrary codimension.

\begin{theorem}\label{t:abt} Suppose $1\leq n\leq N-1$. If $f:\RR^{N}\rightarrow \RR^{N}$ is quasiconformal and (\ref{e:bgrt-dini}) holds,
then the quasiplane $f(\RR^{n})$ is locally $(1+\delta)$-bi-Lipschitz equivalent to subsets of $\RR^n$ for all $\delta>0$. Thus, $f(\RR^n)$ is $n$-rectifiable and $\Haus^{n}\res f(\RR^n)$ is locally finite.\end{theorem}

Secondly, we show how to relax the hypothesis of Theorem \ref{t:abt} and obtain the conclusion that a  quasiplane is locally bi-Lipschitz equivalent to subsets of $\RR^n$.

\begin{theorem}\label{t:abt2} Suppose $1\leq n\leq N-1$. If $f:\RR^{N}\rightarrow \RR^{N}$ is quasiconformal and
\begin{equation} \label{e:abt-dini}
\sup_{z\in B^n(x_0,1)}\int_{0}^{1} \tH_{f}(B^N(x,r))^{2} \frac{dr}{r} <\infty \quad \text{for all $x_0\in\RR^n$},
\end{equation}
then the quasiplane $f(\RR^{n})$ is locally bi-Lipschitz equivalent to subsets of $\RR^n$ near $f(x_0)$ for each $x_0\in\RR^n$ with local bi-Lipschitz constant depending only on $n$, $N$, and the quantity in (\ref{e:abt-dini}). Thus, $f(\RR^n)$ is $n$-rectifiable and $\Haus^{n}\res f(\RR^n)$ is locally finite.\end{theorem}

The exponent 2 appearing in Theorems \ref{t:abt} and \ref{t:abt2} is the best possible; that is, 2 cannot be replaced with $2+\varepsilon$ for any $\varepsilon>0$. For example, the construction in David and Toro \cite{DT-snow} (with the parameters $Z=\RR^{n}$ and $\varepsilon_j=1/j$) can be used to produce a quasiconformal map $f:\RR^{N}\rightarrow\RR^{N}$ ($N=n+1$) such that \begin{equation*}
\int_{0}^{1}\sup_{x\in B^n(x_0,1)} \tH_{f}(B^N(x,r))^{2+\varepsilon} \frac{dr}{r} <\infty \quad \text{for all $x_0\in\RR^n$ and }\varepsilon>0,
\end{equation*} but for which the associated quasiplane $f(\RR^n)$ is not $n$-rectifiable and has locally infinite $\Haus^n$ measure; in fact, $f(\RR^n)$ does not contain any curves with positive and finite $\Haus^1$ measure.

The third main result of the paper is that one can replace the locally uniform condition \eqref{e:abt-dini} with a Carleson measure condition and still detect some rectifiable structure in the image $f(\RR^n)$. To make this precise, we introduce some additional terminology. A set $X\subset\RR^{N}$ contains \emph{big pieces of bi-Lipschitz images of $\RR^n$} if there exist constants $L\geq 1$ and $\alpha>0$ such that for all $x\in X$ and $0<r<\diam X$ there exist $S_{x,r}\subset X\cap B^N(x,r)$ and $h_{x,r}:S_{x,r}\rightarrow \RR^n$ such that $\Haus^n(S_{x,r})\geq \alpha r^n$ and $h_{x,r}$ is $L$-bi-Lipschitz. The constants $L$ and  $\alpha$ are collectively called \emph{BPBI constants} of $X$; to differentiate between them, we call $L$ a \emph{BPBI bi-Lipschitz constant} of $X$ and we call $\alpha$ a \emph{BPBI big pieces constant} of $X$.

\begin{theorem}\label{t:main}  Suppose $2\leq n\leq N-1$. If $f:\RR^N\rightarrow \RR^N$ is quasiconformal and there exists $C_f>0$ such that for all $x_{0}\in \RR^{n}$ and $r_{0}>0$,
\begin{equation} \label{e:carleson}
\int_{B^{n}(x_{0},r_{0})} \int_{0}^{r_{0}}\tH_{f}(B^N(x,r))^{2}\,\frac{dr}{r}\,d\Leb^{n}(x) \leq C_f\, \Leb^{n}(B^{n}(x_{0},r_{0})),
\end{equation}
then  the quasiplane $f(\RR^{n})$ contains big pieces of bi-Lipschitz images of $\RR^{n}$ with BPBI constants depending on at most $n$, $N$, $H_f(\RR^n)$, and $C_f$. Furthermore,  the BPBI bi-Lipschitz constant $L=L(n,N,C_f)\rightarrow 1$ as $C_f\rightarrow 0$ with $n$ and $N$ held fixed.\end{theorem}

In the theory of uniform rectifiability \cite{DS91,DS93}, it is usually assumed that a set $X\subset\RR^N$ with big pieces of bi-Lipschitz images of $\RR^n$ is closed and Ahlfors $n$-regular, in the sense that $c_1r^{n}\leq \Haus^n(X\cap B^N(x,r))\leq c_2r^n$ for all $x\in X$ and $0<r< \diam X$. However, we wish to emphasize that in this paper we do not impose these regularity assumptions in the definition of big pieces of bi-Lipschitz images of $\RR^n$. As a consequence, the quasiplanes in Theorem \ref{t:main} are not necessarily $n$-rectifiable, but at least contain uniformly large rectifiable sets at each location and scale in the image. The restriction to $n\geq 2$ in Theorem \ref{t:main} enters our proof of the theorem when we invoke Gehring's theorem on distortion of Lebesgue measure by quasiconformal maps in $\RR^n$ (see Corollary \ref{c:gehring}). We do not currently know whether or not Theorem \ref{t:main} holds for quasilines. In this context, let us mention that in recent work the first author gave necessary and sufficient conditions in terms of linear approximation properties of $f$ for the image $f(\RR^n)$ of a quasisymmetric map $f:\RR^n\rightarrow\RR^N$ to have big pieces of bi-Lipschitz images of $\RR^n$ when $n\geq 2$, but demonstrated that analogous characterizations fail when $n=1$; see \cite{A-qdiff} for details.

At the core of each of Theorems \ref{t:abt}, \ref{t:abt2}, and \ref{t:main},  is a crucial observation of Prause \cite{Prause} that the image $f(\RR^n)$ of an embedding $f:\RR^N\rightarrow\RR^N$ with small weak quasisymmetry constant $\tH_f(B^N(x,r))$ along $x\in \RR^n$ can be locally approximated by $n$-dimensional planes in $\RR^N$ with correspondingly small error. See \S\ref{s:prelim2} for precise formulations of approximation of a set by planes and related criterion for bi-Lipschitz parameterization by subsets of $\RR^n$. To prove Theorem \ref{t:bgrt}, the authors of \cite{BGRT} gave a refinement of Prause's estimate in the special case $n=N-1$ and used it check the hypothesis of a bi-Lipschitz parameterization theorem from \cite{Toro95}. This approach had two limitations, which we show how to sidestep below. First and foremost the bi-Lipschitz parameterization theorem of \cite{Toro95} requires strong bilateral affine approximation estimates for $f(\RR^n)$, which we (still) do not know how to verify in the case of higher codimension ($1\leq n\leq N-2$). In its place, we now use a more flexible parameterization theorem from \cite{DT12}, which only requires strong \emph{unilateral} affine approximation estimates \emph{and} weak bilateral affine approximation estimates (see Theorem \ref{t:bilip} below). Checking the hypothesis of the new parameterizations theorem for quasiplanes in arbitrary codimension is non-trivial and requires several new estimates, but is within reach. See \S\ref{s:outline} for a detailed outline of our approach.

The second limitation from \cite{BGRT} that we address is how to relax the strong uniformity in condition (\ref{e:bgrt-dini}). In particular, to prove Theorem \ref{t:main}, we develop a tool for extending quasisymmetric mappings that are locally ``almost affine''. This extension result (Theorem \ref{t:extend1}) is of independent interest. For the definition of an almost affine map and the statement of the extension theorem, see \S\ref{ss:almost} and \S\ref{s:extend}, respectively. Roughly speaking, we show that if a map $f:E\rightarrow\RR^N$ defined on a closed set $E\subset\RR^n$ is approximately affine at all scales and locations in a suitable sense, then it extends to a global map $F:\RR^n\rightarrow\RR^N$ that is still almost affine and is smooth away from $E$. Moreover, if the affine approximations to $f$ are uniformly quasisymmetric, then the map $F$ is  quasisymmetric. This is related to investigations by Tukia and V\"ais\"al\"a (see \cite{TV84} and  \cite{Vaisala86}) on sets $E\subset \RR^n$ with the \emph{quasisymmetric extension property}, i.e.~ sets on which every embedding $f:E\rightarrow \RR^N$ with small quasisymmetric distortion can be extended to a quasisymmetric map on $\RR^n$. As shown by the first author (see \cite {A-qdiff}), understanding the approximation properties of a quasisymmetric map by affine maps is critical to decoding the geometry of its image.

The remainder of the paper is organized as follows. To start, we give two preliminary sections, which contain the necessary background on quasisymmetric and quasiconformal maps (\S\ref{s:prelim}) and affine approximation and bi-Lipschitz parameterization of sets (\S\ref{s:prelim2}). Next, we outline the new ingredients appearing in the proofs of the main theorems in \S\S\ref{ss:qsflat}--\ref{ss:almost}; and, we record the proofs of the main theorems in \S\ref{ss:proofs}. In the second half of the paper, \S\S \ref{s:flatness}--\ref{s:extend2}, we verify the new claims in \S\ref{s:outline}. The contents of these latter sections are described in the outline in \S\ref{s:outline}.

Throughout the sequel, we write $a\lesssim b$ (or $b\gtrsim a$) to denote that $a\leq Cb$ for some absolute constant $0<C<\infty$ and write $a\sim b$ if $a\lesssim b$ and $b\lesssim a$. Likewise we write $a\lesssim_{t} b$ (or $b\gtrsim_{t} a$) to denote that $a\leq C b$ for some constant $0<C<\infty$ that may depend on a list of parameters $t$ and write $a\sim_t b$ if $a\lesssim_t b$ and $b\lesssim_t a$.


\section{Preliminaries I: quasisymmetric and quasiconformal maps}
\label{s:prelim}

This section is intended to be a quick overview of the definitions of quasisymmetric, weakly quasisymmetric, and quasiconformal maps; the relationships between them; and a smattering of their essential properties. For additional background,  we refer the reader to V\"ais\"al\"a \cite{Vaisala}, and  Heinonen \cite{Heinonen}. Lemma \ref{l:wqs-crit}, Corollary \ref{c:1qs}, as well as the derivation of Corollary \ref{c:gehring} from Theorem \ref{t:gehring} are standard exercises, whose proofs are included for the convenience  of the reader.

A \emph{topological embedding} $f:X\rightarrow Y$ from a metric space $(X,d_X)$ into a metric space $(Y,d_Y)$ is a map that is a homeomorphism onto its image $f(X)$. A \emph{quasisymmetric map} $f:X\rightarrow Y$ is a topological embedding that ``preserves relative distances'' in the sense that $$d_X(a,x)\leq t\, d_X(b,x)\Longrightarrow d_Y(f(a),f(x))\leq \eta(t)\, d_Y(f(b),f(x))$$ for all $a,b,x\in X$ and $t>0$, for some increasing homeomorphism $\eta:(0,\infty)\rightarrow(0,\infty)$ called a  \emph{control function} for $f$. A map $f:X\rightarrow Y$ is called \emph{$\eta$-quasisymmetric} if $f$ is quasisymmetric and $\eta$ is a control function for $f$.

Quasisymmetric maps behave well under three basic map operations. First, the restriction $f|_A$ of an $\eta$-quasisymmetric map $f:X\rightarrow Y$ to a subset $A\subset X$ is again $\eta$-quasisymmetric. Second, the inverse $f^{-1}:f(X)\rightarrow X$ of $f$ is $\eta'$-quasisymmetric, where $\eta'(t)=1/\eta^{-1}(1/t)$ for all $t>0$. Third, the composition $g\circ f:X\rightarrow Z$ of $f$ with a $\zeta$-quasisymmetric map $g:Y\rightarrow Z$ is $(\zeta\circ \eta)$-quasisymmetric.

Quasisymmetric embeddings map bounded spaces onto bounded spaces, quantitatively.

\begin{lemma}[{\cite[Proposition 10.8]{Heinonen}}] If $f:X\rightarrow Y$ \label{l:QS-comp} is $\eta$-quasisymmetric and $A\subset B\subset X$ are such that $0<\diam A\leq \diam B<\infty$, then
\begin{equation*}
\frac{1}{2\eta\left(\frac{\diam B}{\diam A}\right)} \leq \frac{\diam f(A)}{\diam f(B)}\leq \eta\left(\frac{2\diam A}{\diam B}\right).
\end{equation*}
\end{lemma}

A \emph{weakly quasisymmetric map} $f:X\rightarrow Y$ is a topological embedding such that \begin{equation*}\begin{split}H_f(X):=\inf\{&H\geq 1: d_X(a,x)\leq d_X(b,x)\Longrightarrow\\ &d_Y(f(a),f(x))\leq H\, d_Y(f(b),f(x))\text{ for all } a,b,x\in X\}<\infty.\end{split}\end{equation*} The quantity $H_f(X)$ is called the \emph{weak quasisymmetry constant} of the map $f$ on $X$. A map $f:X\rightarrow Y$ is \emph{weakly $H$-quasisymmetric} if $f$ is weakly quasisymmetric and $H_f(X)\leq H<\infty$.

Every quasisymmetric map is weakly quasisymmetric. To wit, if $f$ is an $\eta$-quasisymmetric map on $X$, then $H_f(X)\leq \eta(1)$. In fact, for every quasisymmetric map $f$ on $X$ there exist (many) control functions $\eta_f$ such that $H_f(X)=\eta_f(1)$. Less obvious, however, is the fact that for certain metric spaces every weakly quasisymmetric map is quasisymmetric. A metric space $X$ is called \emph{doubling} if there is a positive integer $D=D(X)$ so that every set of diameter $d$ in the space can be covered by at most $D$ sets of diameter at most $d/2$.

\begin{theorem}[{\cite[Theorem 10.19]{Heinonen}}] Let $X$ and $Y$ be doubling metric spaces. If $X$ is connected and $f:X\rightarrow Y$ is weakly quasisymmetric, then $f$ is $\eta$-quasisymmetric for some control function $\eta$ depending only on doubling character of $X$ and $Y$, and on $H_f(X)$.\end{theorem}

In particular, weakly quasisymmetric maps between Euclidean spaces are quasisymmetric.

\begin{corollary}[{\cite[Corollary 10.22]{Heinonen}}] \label{c:wqs2qs} Let $X\subset\RR^n$ be a connected set and let $f:X\rightarrow\RR^N$. If $f$ is weakly quasisymmetric, then $f$ is $\eta$-quasisymmetric for some control function depending only on $n$, $N$, and $H_f(X)$. \end{corollary}

\begin{theorem}[{\cite[Theorem 10.30]{Heinonen}}] \label{t:wqs-cpt} Let $X\subset\RR^n$ be a connected set containing $x_1\neq x_2$. For all $H\geq 1$, the family of  weakly $H$-quasisymmetric maps $f:X\rightarrow\RR^N$ such that $f(x_i)=x_i$ for $i=1,2$ is sequentially compact in the topology of uniform convergence on compact sets.\end{theorem}

Here is a useful criterion for checking that a map from one Euclidean space into another is weakly quasisymmetric.

\begin{lemma}\label{l:wqs-crit} If $f:\RR^n\rightarrow\RR^N$ is continuous, nonconstant, and $H_f(\RR^n)<\infty$, then $f$ is weakly quasisymmetric.\end{lemma}

\begin{proof} Suppose $f:\RR^n\rightarrow\RR^N$ is continuous, nonconstant, and $H_f(\RR^n)<\infty$. To show that $f$ is weakly quasisymmetric we must prove $f$ is injective and $f^{-1}:f(\RR^n)\rightarrow\RR^n$ is continuous.

Assume to reach a contradiction that $f(x_0)=f(z_0)$ for some $x_0\neq z_0$, and let $r>0$ denote the distance between $x_0$ and $z_0$. Then \begin{equation*} |f(x_0)-f(y)|\leq H_f(\RR^n)|f(x_0)-f(z_0)|=0 \quad\text{for all }|x_0-y|\leq r,\end{equation*} since $H_f(\RR^n)<\infty$. That is, $f$ is constant on $B^n(x_0,r)$. Let $x_1\in\partial B^n(x_0,r)$ denote the unique point such that $|x_0-x_1|=r$ and $|z_0-x_1|=2r$. Then $f(x_1)=f(x_0)=f(z_0)$ and \begin{equation*} |f(z_0)-f(y)| \leq H_f(\RR^n)|f(z_0)-f(x_1)|=0\quad\text{for all }|z_0-y|\leq 2r,\end{equation*} since $H_f(\RR^n)<\infty$. That is, $f$ is constant on $B^n(z_0,2r)$. Let $z_1\in \partial B^n(z_0,2r)$ denote the unique point such that $|z_0-z_1|=2r$ and $|x_1-z_1|=4r$. Proceeding inductively, we see that $f$ is constant on a sequence of balls, \begin{equation*}B^n(x_0,r)\subset B^n(z_0,2r)\subset B^n(x_1,4r)\subset B^n(z_1,8r)\subset\cdots,\end{equation*} exhausting $\RR^n$. This contradicts the hypothesis that $f$ is nonconstant. Therefore, $f$ is injective.

Let $\wRR{n}=\RR^n\cup\{\infty\}$ and $\wRR{N}=\RR^N\cup\{\infty\}$ denote the one-point compactifications of $\RR^n$ and $\RR^N$, respectively. Extend $f$ to an injective map $F:\wRR{n}\rightarrow\wRR{N}$ by defining $F(\infty)=\infty$ and $F(x)=f(x)$ for all $x\in\RR^n$. Every injective continuous map from a compact space onto a Hausdorff space is open. Thus, if $F$ is continuous, then $f^{-1}=F^{-1}|_{f(\RR^n)}$ is continuous too. In other words, to check that $f^{-1}$ is continuous, it suffices to prove $F$ is continuous.
Because $F|_{\RR^n}=f$ is continuous, the full map $F$ is continuous if and only if  $f(x_i)\rightarrow\infty$ for every sequence $(x_i)_{i=1}^\infty$ in $\RR^n$ such that $x_i\rightarrow\infty$.

Let $(x_i)_{i=1}^\infty$ be any sequence in $\RR^n$ such that $x_i\rightarrow\infty$. By truncating a finite number of terms, we may assume without loss of generality that $r_i:=|x_i-x_1| \geq |x_2-x_1|>0$ for all $i\geq 2$. Note that $r_i\rightarrow\infty$, since $x_i\rightarrow\infty$.
For all $i\geq 2$, let $f_i$ denote the restriction of $f$ to $B^n(x_1,r_i)$. Then $f_i$ is open, again because every one-to-one continuous map from a compact space onto a Hausdorff space is open. Thus, each $f_i$ is a topological embedding from $B^n(x_1,r_i)$ into $\RR^N$ with $H_{f_i}(B^n(x_1,i))\leq H_f(\RR^n)<\infty$. By Corollary \ref{c:wqs2qs}, the maps $f_i$ are uniformly $\eta$-quasisymmetric for some control function $\eta$ that is independent of $i$. Hence \begin{equation*} |f(x_i)-f(x_1)| \geq \frac{|f(x_2)-f(x_1)|}{\eta(r_2/r_i)}\rightarrow\infty \quad\text{as }i\rightarrow\infty,\end{equation*} since $\lim_{i\rightarrow\infty} \eta(r_2/r_1)=0$. It follows that $f(x_i)\rightarrow\infty$. Therefore, $f^{-1}$ is continuous.\end{proof}

A \emph{quasiconformal map}\footnote{There are three commonly used definitions of quasiconformal maps in Euclidean space, which are equivalent \emph{a posteriori}. The definition given here is called the \emph{analytic definition} of a quasiconformal map. The others are the so-called \emph{geometric} and \emph{metric} definitions; for the full story, see e.g.~ \cite{Juha-QC} or \cite{Vaisala}.}  $f:\Omega\rightarrow\RR^N$ is a topological embedding from a domain $\Omega\subset\RR^N$ ($N\geq 2$) such that $f\in W_{\loc}^{1,N}(\Omega)$ and \begin{equation*} K_f(\Omega):=\esssup_{x\in\Omega} \max\left\{\frac{\lambda_N(f,x)^N}{\lambda_1(f,x)\cdots\lambda_N(f,x)}, \frac{\lambda_1(f,x)\cdots\lambda_N(f,x)}{\lambda_1(f,x)^N}\right\}<\infty. \end{equation*} Here $0\leq \lambda_1(f,x)\leq \dots\leq \lambda_N(f,x)<\infty$ denote the singular values of the total derivative $Df(x)$ of $f$ at $x$, i.e.~ the (positive) square root of the eigenvalues of $(Df(x))^TDf(x)$, which are defined at almost every $x\in \Omega$. The quantity $K_f(\Omega)$ is called the \emph{maximal dilatation} of the map $f$ in $\Omega$. A quasiconformal map $f$ is called \emph{$K$-quasiconformal} if $K_f(\Omega)\leq K<\infty$. If $f:\Omega\rightarrow\RR^N$ is $K$-quasiconformal, then the inverse $g=f^{-1}:f(\Omega)\rightarrow\Omega$ is also $K$-quasiconformal.

Every quasisymmetric map $f:\Omega\rightarrow\RR^N$ on a domain $\Omega\subset\RR^N$ ($N\geq 2$) is quasiconformal with $K_f(\Omega)\leq H_f(\Omega)^{N-1}$. In the other direction, the situation is as follows.

\begin{theorem}[{\cite[Theorem 11.14]{Heinonen}}] \label{t:qc2qs} Every quasiconformal map $f:\RR^N\rightarrow\RR^N$ is $\eta_{N,K}$-quasisymmetric for some control function $\eta_{N,K}$ depending only on $N$ and $K=K_f(\RR^N)$.\end{theorem}

Quasiconformal maps exhibit special behavior when $K=1$. Recall that a homeomorphism $f:X\rightarrow X$ in a metric space $(X,d_X)$ is a \emph{similarity} if there exists a constant $0<\lambda<\infty$ such that $d_X(f(x),f(y))=\lambda d_X(x,y)$ for all $x,y\in X$. The group of similarities in Euclidean space is generated by compositions of translations, rotations, reflections, and dilations.

\begin{theorem}[{\cite[Theorem II.2]{Ahlfors-qc}}] \label{t:a1qc} If $N=2$ and $f:\Omega\rightarrow\RR^2$ is a 1-quasiconformal map, then $f$ is a conformal map. \end{theorem}

\begin{theorem}[{\cite[Theorem 16]{Gehring-rings}}] \label{t:g1qc} If $N\geq 3$ and $f:\Omega\rightarrow\RR^N$ is a 1-quasiconformal map, then $f$ is the restriction of a M\"obius transformation of $\wRR{N}=\RR^N\cup\{\infty\}$ to $\Omega$.\end{theorem}

\begin{corollary} \label{c:1qs} If $N\geq 2$ and $f:B^N(x,r)\rightarrow \RR^N$ is weakly $1$-quasisymmetric for some $x\in\RR^N$ and $r>0$, then $f$ is the restriction of a similarity of $\RR^N$ to $B^N(x,r)$. \end{corollary}

\begin{proof} Suppose $N\geq 2$. Since the composition of a weakly 1-quasisymmetric map with a similarity in the domain is still weakly 1-quasisymmetric, it suffices to prove the lemma on the unit ball. Suppose that $f:B^N(0,1)\rightarrow\RR^N$ is a weakly 1-quasisymmetric map. Replacing $f(x)$ by $f(x)-f(0)$ for all $x\in\RR^N$, which leaves the quasisymmetry of $f$ untouched, we may also suppose without loss of generality that $f(0)=0$. On one hand, \begin{equation*}|x-a|=|y-a|\Longrightarrow |f(x)-f(a)|=|f(y)-f(a)|\quad\text{for all }x,y,a\in B^N(0,1),\end{equation*} because $f$ is weakly 1-quasisymmetric. Hence $f$ maps $B^N(0,1)$ onto a ball in $\RR^N$ centered at $0$. On the other hand, by Corollary \ref{c:wqs2qs}, $f$ is quasisymmetric. Thus, the restriction $f_\circ=f|_{\oB^N(0,1)}$ of $f$ to the open unit ball is quasiconformal with $K_{f_\circ}(\oB^N(0,1))\leq H_f(B^N(0,1))^{N-1}=1$. That is, $f$ is a 1-quasiconformal map. When $N\geq 3$, we conclude that $f_\circ$ is the restriction of some M\"obius transformation $F$ on $\wRR{N}$ by Theorem \ref{t:g1qc}. When $N=2$, we conclude that $f_\circ$ is the restriction of some M\"obius transformation $F$ on the Riemann sphere $\wRR{2}$, because $f_\circ$ is  conformal by Theorem \ref{t:a1qc} and maps the unit disk \emph{onto} a disk. Finally, since $F$ maps a ball centered at the origin onto a ball centered at the origin, $F$ must fix the point at infinity. Therefore, the map $f_\circ$ is the restriction of a similarity of $\RR^N$. The same conclusion extends to $f$ by continuity.
\end{proof}

Quasiconformal maps are locally H\"older continuous with exponent depending only on the dimension and the maximal dilatation of the map.

\begin{theorem}[{\cite[Theorem 11.14]{Vuorinen}}]\label{t:holder} Given $N\geq 2$ and $1\leq K<\infty$, put $\alpha = K^{1/(1-N)}$. If $f:\oB^N(x,r)\rightarrow \RR^N$ is $K$-quasiconformal, then \begin{equation*} |f(y)-f(z)| \lesssim_{N,K} \left(\sup_{|w-x|<r}|f(w)-f(x)|\right) \left|\frac{y}{r}-\frac{z}{r}\right|^{\alpha}\quad\text{for all }y,z\in B^N(x,r/2).\end{equation*}\end{theorem}

For any domain $\Omega\subset\RR^N$ and map $f:\Omega\rightarrow\RR^N$, the \emph{maximal stretching}   $L_f:\Omega\rightarrow[0,\infty]$ of $f$ is defined by \begin{equation*}L_{f}(x)= \limsup_{y\rightarrow x} \frac{|f(x)-f(y)|}{|x-y|}\quad \text{for all }x\in \Omega.\end{equation*}  If $f$ is quasiconformal, then $L_f(x)=\lambda_N(f,x)$  and $Jf(x)\leq L_f(x)^N\leq K_f(\Omega) Jf(x)$ at $\Leb^N$-a.e.~ $x$, where $J f(x)=\lambda_1(f,x)\cdots\lambda_N(f,x)$ denotes the Jacobian determinant of $f$ at $x$. Gehring \cite{Gehring73} proved that if $f$ is quasiconformal, then $L_{f}$ satisfies a reverse H\"older inequality.

\begin{theorem}[{\cite[Lemmas 3,4]{Gehring73}}] \label{t:gehring} If $N\geq 2$ and $f:\Omega\rightarrow \RR^N$ is a quasiconformal map, then there are constants $c>0$ and $p>0$ depending only on $N$ and $K_f(\Omega)$ such that for every closed cube $Q\subset \Omega$ satisfying
$\diam f(Q)<\dist(f(Q),\partial f(\Omega))$, \begin{equation}\label{e:rh1}\left(\avint_{Q}L_{f}^{N+p}\,d\Leb^{N}\right)^{1/(N+p)} \leq c \avint_{Q}L_{f}\,d\Leb^N.\end{equation}
\end{theorem}

\begin{corollary}\label{c:gehring} If $N\geq 2$ and $f:\RR^N\rightarrow\RR^N$ is a quasiconformal map, then there is a constant $q>0$ depending only on $N$ and $K_f(\RR^N)$ such that \begin{equation} \label{e:g1}\frac{\Leb^N(f(A))}{\Leb^N(f(Q))} \geq \frac{1}{2} \exp\left(-q\frac{\Leb^N(Q)}{\Leb^N(A)}\right)\end{equation} for every closed cube $Q\subset\RR^N$ and every Borel set $A\subset Q$.
 \end{corollary}

 \begin{proof} Suppose $f:\RR^N\rightarrow\RR^N$ is quasiconformal and let $K:=K_f(\RR^N)$. By Theorem \ref{t:gehring}, $L_f$ satisfies the reverse H\"older inequality (\ref{e:rh1}) for some constants $c>0$ and $p>0$ depending only on $N$ and $K$. Let $Q\subset\RR^N$ be any closed cube. Because $Jf(x)\leq L_f(x)^N\leq K\, Jf(x)$  at $\Leb^N$-a.e.~ $x\in\RR^N$, we see that $Jf$ also satisfies a reverse H\"older inequality:
\begin{equation}\begin{split}
\left(\avint_{Q} Jf^{\frac{N+p}{N}}\,d\Leb^N\right)^{\frac{N}{N+p}}
&\leq \left(\avint_{Q} L_{f}^{N+p}\,d\Leb^N\right)^{\frac{N}{N+p}}
\leq c^N\left(\avint_{Q} L_{f}\,d\Leb^N\right)^N\\
&\leq c^N\avint_{Q} L_{f}^N\,d\Leb^N \leq K c^N \avint_{Q}Jf\,d\Leb^N. \label{e:rh2}
\end{split}\end{equation} In particular, the Jacobian $Jf$ of $f$ is an $A_\infty$ weight with respect to Lebesgue measure $\Leb^N$; e.g., see \cite[Chapter 9]{Grafakos2} or \cite[Chapter V]{Big-Stein}. Therefore, by Hru\v{s}\v{c}ev's inequality for $A_\infty$ weights \cite[(7)]{Hruscev}, there is $q>0$ such that for all cubes $Q\subset\RR^N$ and Borel sets $A\subset Q$, \begin{equation*}\frac{w(A)}{w(Q)} \geq \frac{1}{1+\exp\left(q\frac{\Leb^N(Q)}{\Leb^N(A)}\right)}\geq \frac{1}{2}\exp\left(-q\frac{\Leb^N(Q)}{\Leb^N(A)}\right),\end{equation*} where $w(E)=\int_E Jf\,d\Leb^N=\Leb^N(f(E))$ for all Borel sets $E\subset\RR^N$. The constant $q$ depends only on the constants in (\ref{e:rh2}), and thus, $q$ ultimately depends only on $N$ and $K$. \end{proof}

The conclusion of Corollary \ref{c:gehring} does not hold for quasisymmetric maps in $\RR^N$ when $N=1$; in fact, by an example of Beurling and Ahlfors \cite{BA}, a quasisymmetric map $f:\RR\rightarrow\RR$ can map a set of positive Lebesgue measure onto a set of Lebesgue measure zero.


\section{Preliminaries II: local flatness and bi-Lipschitz parameterizations}
\label{s:prelim2}

In this section and implicitly below, whenever using the quantities defined in Definition \ref{d:flatness}, we assume that $1\leq n\leq N-1$.
Let $\mathcal{G}=\mathcal{G}_{N,n}$ denote the affine Grassmannian of $n$-dimensional planes in $\RR^N$, and let $\mathcal{G}(x)=\mathcal{G}_{N,n}(x)=\{V\in\mathcal{G}_{N,n}:x\in V\}$ denote the subcollection of planes containing $x\in\RR^N$. We write $a\vee b$ to denote the maximum of $a,b\in\RR$.

\begin{definition}[Measurements of local flatness of sets] \label{d:flatness} For all $E\subset\RR^N$, $x\in E$ and $r>0$, define the quantities $0\leq \beta_E(x,r)\leq \betac_E(x,r)\leq \theta_E(x,r)\leq 1$ by
\begin{equation*}\beta_E(x,r) := \inf_{V\in\mathcal{G}}\frac{1}{r}\left(\sup_{y\in E\cap B^N(x,r)} \dist(y,V)\right),\end{equation*}
\begin{equation*}\betac_E(x,r) := \inf_{V\in\mathcal{G}(x)}\frac{1}{r}\left( \sup_{y\in E\cap B^N(x,r)} \dist(y,V)\right), \end{equation*} and
\begin{equation*}\theta_E(x,r) := \inf_{V\in\mathcal{G}(x)} \frac{1}{r}\left(\left(\sup_{y\in E\cap B^N(x,r)}\dist(y,V)\right)\vee \left(\sup_{z\in V\cap B^N(x,r)}\dist(z,E)\right)\right).\end{equation*}
\end{definition}

Each of the measurements of flatness defined in Definition \ref{d:flatness} satisfy a monotonicity property: an estimate of flatness at one scale yields (worse) estimates of flatness on smaller scales. Namely, for all $E\subset\RR^N$, $x\in E$, $r>0$ and $s\in(0,1]$, \begin{equation}\begin{split} \label{e:mono} \beta_E(x,sr)\leq &s^{-1}\beta_E(x,r), \quad \betac_E(x,sr)\leq s^{-1}\beta_E(x,r),\\ &\text{and}\quad \theta_E(x,sr)\leq s^{-1}\theta_E(x,r).\end{split}\end{equation} In addition, if $B^N(y,sr)\subset B^N(x,r)$ for some $x,y\in E$ and $r,s>0$, then \begin{equation} \label{e:monob} \beta_E(y,sr) \leq s^{-1}\beta_E(x,r).\end{equation}

\begin{remark}[Origins and choice of conventions] \label{r:flat} Beta numbers were originally introduced by Jones \cite{Jones-TSP} in order to characterize subsets of rectifiable curves in the plane. For analogues of Jones' Traveling Salesman Theorem in higher dimensions, see \cite{O-TSP} and \cite{Schul-TSP}. Because  \begin{equation} \beta_E(x,r)\leq \betac_E(x,r)\leq 2\beta_E(x,r) \label{e:betas}\end{equation} for all $E\subset\RR^N$, $x\in E$ and $r>0$, the decision to use ``uncentered'' beta numbers $\beta_E(x,r)$ or ``centered'' beta numbers $\betac_E(x,r)$ is largely a matter of taste and may depend on the application. We use the former below, except in a theorem which we quote from \cite{DT12} that chose the latter.

In some instances, see e.g.~ \cite{Toro95}, \cite{BGRT}, the theta numbers $\theta_E(x,r)$ are replaced by the strictly larger numbers \begin{equation*}\thetahd_E(x,r) := \inf_{V\in\mathcal{G}(x)} \frac{1}{r}\HD\left(E\cap B^N(x,r),V\cap B^N(x,r)\right),\end{equation*} where $\HD(Y,Z)=(\sup_{y\in Y}\dist(y,Z))\vee (\sup_{z\in Z}\dist(z,Y))$ denotes the Hausdorff distance between bounded sets $Y,Z\subset\RR^N$. The quantity $\thetahd_E(x,r)$ is more difficult to estimate than $\theta_E(x,r)$ (e.g., $\thetahd_E(x,r)$ does not satisfy (\ref{e:mono})). Thus we choose to use the latter below.

Closed sets that are locally uniformly close to planes at all locations and scales first appeared in Reifenberg's solution of the Plateau problem in arbitrary codimension \cite{Reifenberg}; following \cite{KT97} these sets are now called Reifenberg flat sets. Precisely, in this paper, we say that a closed set $\Sigma\subset\RR^N$ is \emph{$(\delta,R)$-Reifenberg flat} if  $\theta_{\Sigma}(x,r)\leq \delta$ for all $x\in\Sigma$ and $0<r<R$. Mattila and Vuorinen \cite{MV90} (independently of Jones \cite{Jones-TSP}) introduced the following related definition, in the context of obtaining upper Minkowski and Hausdorff dimension bounds for quasispheres. A set $\Sigma\subset\RR^N$ is said to have the \emph{$(\delta,R)$-linear approximation property} if $\betac_{E}(x,r)\leq \delta$ for all $x\in\Sigma$ and $0<r<R$. Trivially every subset of a $(\delta,R)$-Reifenberg flat set has the $(\delta,R)$-linear approximation property. However, there exist sets with the $(\delta,R)$-linear approximation property that do not belong to any $(\delta,R')$-Reifenberg flat sets; e.g., see \cite[Counterexample 12.4]{DT12}.\end{remark}

We now present a version of Reifenberg's topological disk theorem, which gives a sufficient condition  for a closed set $\Sigma\subset\RR^N$ to be locally bi-H\"older equivalent to open subsets of $\RR^n$.

\begin{theorem}[Local version of Reifenberg's topological disk theorem {\cite[Theorem 1.1]{DT12}}] There exists $\delta_0=\delta_0(n,N)>0$ with the following property. If $\Sigma\subset\RR^N$ is closed, $x_0\in \Sigma$, $r_0>0$, $0<\delta\leq \delta_0$, and $\theta_\Sigma(x,r)\leq \delta$ for all $x\in \Sigma\cap B^N(x_0,10r_0)$ and $0<r\leq 10r_0$, \label{t:reif} then there exist a bijective mapping $g:\RR^N\rightarrow\RR^N$ and an $n$-dimensional plane $V$ containing $x_0$ such that \begin{equation*} |g(x)-x|\leq \frac{r_0}{100}\quad\text{for all }x\in\RR^N,\end{equation*} \begin{equation*} \frac{r_0}{4}\left|\frac{x}{r_0}-\frac{y}{r_0}\right|^{1.01}\leq |g(x)-g(y)|\leq 3r_0\left|\frac{x}{r_0}-\frac{y}{r_0}\right|^{0.99}\end{equation*} for all $x,y\in\RR^n$ such that $|x-y|\leq r_0$, and \begin{equation*}\Sigma\cap B^N(x_0,r_0)=g(V)\cap B^N(x_0,r_0).\end{equation*}\end{theorem}

In \cite{DT12}, the third named author, together with Guy David, found several conditions that guarantee that the parameterization in Reifenberg's topological disk theorem is bi-Lipschitz.

\begin{theorem}[Local bi-Lipschitz parameterization {\cite[Theorem 1.3]{DT12}}] \label{t:bilip} For every $M<\infty$, there exists $L=L(n,N,M)<\infty$ with the following property. If $\Sigma\subset\RR^N$ is closed, $x_0\in \Sigma$, $r_0>0$, $0<\delta\leq\delta_0$, and $\theta_\Sigma(x,r)\leq \delta$ for all $x\in \Sigma\cap B^N(x_0,10r_0)$ and $0<r\leq 10r_0$, and \begin{equation} \label{e:betasum} \sup_{x\in \Sigma\cap B^N(x_0,10r_0)} \sum_{k=0}^\infty \betac_\Sigma(x,10^{-k}r_0)^2\leq M<\infty,\end{equation} then the mapping $g$ provided by Theorem \ref{t:reif} can be chosen to satisfy \begin{equation} \label{e:bilip} \frac{|x-y|}{L}\leq |g(x)-g(y)|\leq L|x-y|\quad\text{for all }x,y\in\RR^N.\end{equation}\end{theorem}

\begin{corollary}[Global bi-Lipschitz parameterization] \label{c:bilip} Suppose that $\Sigma\subset\RR^N$ is closed, $x_0\in \Sigma$, $0<\delta<\delta_0$, and $\theta_\Sigma(x,r)\leq \delta$ for all $x\in \Sigma$ and $r>0$. If there exists $M<\infty$ such that (\ref{e:betasum}) holds for all $r_0>0$, then there exists a map $g:\RR^N\rightarrow\RR^N$ satisfying \eqref{e:bilip} such that $\Sigma=g(\RR^n)$.
\end{corollary}

\begin{proof} Let $\delta_0>0$ be the constant from Theorem \ref{t:reif}. Suppose that $\Sigma\subset\RR^N$ is a closed set, $x_0\in \Sigma$, $0<\delta<\delta_0$ and $\theta_\Sigma(x,r)\leq\delta$ for all $x\in \Sigma$ and $r>0$. Furthermore, suppose that for some $M<\infty$ condition \eqref{e:betasum} holds for all $r_0>0$. By Theorem \ref{t:bilip}, applied with $r_0=i\geq 1$, for all $i\geq 1$  there exists a an $n$-dimensional plane $V^i$ containing $x_0$ and a map $g^i:\RR^N\rightarrow\RR^N$ satisfying \eqref{e:bilip} such that $\Sigma\cap B^N(x_0,i)=g^i(V^i)\cap B^N(x_0,i)$. For all $i\geq 1$, choose an isometry $h^i:\RR^N\rightarrow\RR^N$ such that $h^i(\RR^n)=V^i$ and $g^i(h^i(0))=x_0$. The composed maps $f^j:= g^j\circ h^j$ have the property that  $\Sigma\cap B^N(x_0,i)=f^j(\RR^n)\cap B^n(x_0,i)$ for all $1\leq i\leq j$, $f^j(0)=x_0$, and \begin{equation}\label{e:clip1} L^{-1}|x-y|\leq |f^j(x)-f^j(y)|\leq L|x-y|\quad\text{for all }x,y\in\RR^N\text{ and }j\geq 1.\end{equation} In particular, the family $\{f^j:j\geq 1\}$ is equicontinuous, pointwise bounded, and \begin{equation}\label{e:clip2} f^j(B^n(0,Li))\cap B^N(x_0,i)=\Sigma\cap B^N(x_0,i)\quad\text{for all }1\leq i\leq j.\end{equation} By the Arzel\`a-Ascoli theorem, there exists a continuous map $g:\RR^N\rightarrow\RR^N$ and a subsequence of $(f^j)_{j=1}^\infty$ that converges to $g$ uniformly on compact sets. From \eqref{e:clip1} and \eqref{e:clip2}, we conclude that $g$ satisfies \eqref{e:bilip} and $g(\RR^n)=\Sigma$.
\end{proof}

\begin{remark} \label{r:Lsmall} A careful reading of the proof of \cite[Theorem 1.3]{DT12} shows that in Theorem \ref{t:bilip} and Corollary \ref{c:bilip}, when $n$ and $N$ are fixed, the constant $L=L(n,N,M)\rightarrow 1$ as $M\rightarrow 0$.\end{remark}

We end this section with a short computation related to (\ref{e:betasum}).

\begin{lemma}\label{l:check} If $\Sigma\subset\RR^N$ is closed, $x\in \Sigma$, and $r_0>0$, then \begin{equation}\label{e:check} \sum_{k=0}^\infty \betac_{\Sigma}(x,10^{-k}r_0)^2 \leq \frac{400}{\log(10)}\int_0^{10r_0} \beta_{\Sigma}(x,r)^2\frac{dr}{r}.\end{equation} \end{lemma}
\begin{proof} Let $\Sigma\subset\RR^N$ closed, $x\in\Sigma$, and $r_0>0$ be given. If $r\in[10^{-k}r_0,10^{-(k-1)}r_0]$, then \begin{equation*} \betac_\Sigma(x,10^{-k}r_0) \leq 10 \betac_\Sigma(x,r)\leq 20\beta_\Sigma(x,r),\end{equation*} where the first inequality holds by (\ref{e:mono}) and the second inequality holds by (\ref{e:betas}). Therefore, \begin{equation*}\begin{split} \int_0^{10r_0} &\beta_\Sigma(x,r)^2\frac{dr}{r} = \sum_{k=0}^\infty \int_{10^{-k}r_0}^{10^{-(k-1)} r_0} \beta_\Sigma(x,r)^2\frac{dr}{r} \\ &\geq \frac{1}{400}\sum_{k=0}^\infty \int_{10^{-k}r_0}^{10^{-(k-1)} r_0} \betac_\Sigma(x,10^{-k}r_0)^2\frac{dr}{r}=\frac{\log(10)}{400}\sum_{k=0}^\infty \betac_\Sigma (x,10^{-k}r_0)^2.\end{split}\end{equation*} Rearranging the inequality yields (\ref{e:check}). \end{proof}


\section{Outline of new ingredients in and proofs of the main theorems}
\label{s:outline}


\subsection{Quasisymmetry and local flatness of quasiplanes} \label{ss:qsflat} The connection between distortion of local flatness and quasiconformal maps was first recognized by Mattila and Vuorinen \cite{MV90} as a tool to establish upper bounds on the Minkowski and Hausdorff dimensions of quasispheres. Prause \cite{Prause} obtained improved estimates on the dimension of quasispheres, by estimating the distortion of beta numbers using the quasisymmetry of a global quasiconformal map in place of the maximal dilatation. The following theta number variant of \cite[Theorem 5.1]{Prause} was a key tool in the proof of Theorem \ref{t:bgrt} stated above.

\begin{lemma}[{\cite[Lemma 2.3]{BGRT}}]\label{l:bgrt} Suppose that $1\leq n=N-1$. Let $V$ be an $n$-dimensional plane in $\RR^N$, let $v\in V$ and let $e\in (V-v)^\perp$ be a unit vector. For any topological embedding $f:B^N(v,r)\rightarrow\RR^N$, \begin{equation*}\thetahd_{f(V)}\left(f(v),\frac{1}{4}|f(v+re)-f(v-re)|\right) \leq 20 \tH_f(B^N(v,r)).\end{equation*} \end{lemma}

Below we generalize the previous lemma to arbitrary codimension, at the expense of obtaining a beta number estimate instead of a theta number estimate. Lemma \ref{l:bflat} (which we prove in \S\ref{s:flatness}) is a quantitative local version of \cite[Theorem 5.6]{Prause}.

\begin{lemma}\label{l:bflat} Suppose that $1\leq n\leq N-1$. Let $V$ be an $n$-dimensional plane in $\RR^N$, let $v\in V$, and let $e$ be a unit vector in $\RR^N$. For any topological embedding $f:B^N(v,2r)\rightarrow\RR^N$, \begin{equation*}\beta_{f(V)}\left(f(v),\frac{1}{2}|f(v+re)-f(v)|\right) \leq 72N \tH_f(B^N(v, 2r)).\end{equation*}\end{lemma}

When combined with the local H\"older continuity of quasiconformal maps, Lemma \ref{l:bflat} yields the following corollary. See \S\ref{s:flatness} for details.

\begin{corollary} \label{c:bflat} Suppose that $1\leq n\leq N-1$ and $H\geq 1$. There is $C=C(N,H)>1$ such that if $z\in\RR^n$, $t>0$, $f:\oB^N(z,2t)\rightarrow\RR^N$ is quasiconformal, and $H_f(B^N(z,t))\leq H$, then
\begin{equation}\label{e:dc1} \int_0^{\diam f(B^N(z,t))/C} \beta_{f(\RR^n)}(f(z),s)^2\frac{ds}{s}\leq C \int_0^{t} \tH_f(B^N(z,s))^2\frac{ds}{s}.\end{equation}\end{corollary}


\subsection{Almost affine quasisymmetric maps and extension theorems}
\label{ss:almost} Throughout this section and implicitly below, whenever using the concepts defined in Definitions \ref{d:comp} and \ref{d:aa}, we assume that $1\leq n\leq N$. For all affine maps $A:\RR^n\rightarrow\RR^N$, let $A'$ denote the linear part of $A$, let $\|A'\|$ denote the operator norm of $A'$, and let $\lambda_1(A')\leq \dots \leq \lambda_n(A')$ denote the singular values of $A'$. We recall that for all $x\in\RR^n$ and $r>0$, \begin{equation}\label{e:lvar}\lambda_1(A')r=\inf_{|x-y|=r}|A(x)-A(y)|\quad\text{and}\quad \lambda_n(A')r=\sup_{|x-y|=r}|A(x)-A(y)|=\|A'\|r.\end{equation}

\begin{definition}\label{d:comp} A \emph{family of affine maps over $E\subset\RR^n$} is a set \begin{equation*}\mathcal{A}=\{A_{x,r}:x\in E,r>0\}\end{equation*} whose members are (indexed) affine maps $A_{x,r}:\RR^n\rightarrow\RR^N$ for all $x\in E$ and $r>0$. We say that $\mathcal{A}$ is \emph{$\varepsilon$-compatible} for some $\varepsilon>0$ if \begin{equation*}\Adiff{x,r}{y,s}\leq \varepsilon \Amin{x,r}{y,s}\end{equation*} for all $x,y\in E$ and $r,s>0$ such that $|x-y|\leq \max\{r,s\}$ and $1/2\leq r/s\leq 2$.\end{definition}

\begin{definition}\label{d:aa} Let $E\subset X\subset\RR^n$ and let $\varepsilon>0$. A map $f:X\rightarrow\RR^N$ is \emph{$\varepsilon$-almost affine over $E$} if there exists an $\varepsilon$-compatible family $\mathcal{A}$ of affine maps over $E$ such that
\begin{equation*} \sup_{z\in E\cap B^n(x,r)} |f(z)-A_{x,r}(z)|\leq \varepsilon\|A_{x,r}'\| r \quad\text{for all }x\in E, r>0.
\end{equation*} To emphasize a choice of some family $\mathcal{A}$ with this property, we say \emph{$(f,E,\mathcal{A})$ is $\varepsilon$-almost affine}.
\end{definition}

\begin{remark}The definition of an almost affine map is designed so that being almost affine is invariant under translation, rotation, reflection, and dilation in the domain \emph{and} the image of the map. That is, if $\phi:\RR^n\rightarrow\RR^n$ and $\psi:\RR^N\rightarrow\RR^N$ are similarities in $\RR^n$ and $\RR^N$, respectively, then $(f,E,\mathcal{A})$ is $\varepsilon$-almost affine if and only if ($\psi\circ f\circ \phi,\phi^{-1}(E),\psi\circ \mathcal{A}\circ \phi$) is $\varepsilon$-almost affine. For related classes of maps that also admit uniform approximations by affine maps but do not enjoy the same scale-invariance property as almost affine maps, see \cite{DPK09} and \cite{hardsard}.\end{remark}

We record a number of useful estimates for compatible families of affine maps and for almost affine maps in \S\ref{s:estimates}.

The next lemma provides a criterion to check the theta number hypothesis in Theorem \ref{t:bilip} and Corollary \ref{c:bilip} for a set $\Sigma\subset\RR^N$ of the form $\Sigma=f(\RR^n)$.

\begin{lemma}\label{l:rflat} For all $\delta>0$ there exists $\delta_*=\delta_*(\delta)$ with the following property. Suppose that $f:\RR^n\rightarrow\RR^N$ is quasisymmetric and $H_f(\RR^n)\leq H$. If $f$ is $\delta_*$-almost affine over $B^n(x_0,2r_0)$ and $\tH_f(B^n(x_0,2r_0))\leq \delta_*$ for some $x_0\in\RR^n$ and $r_0>0$, then
\begin{equation}\label{e:rfa}\theta_{f(\RR^n)}(f(x),r) \leq H\delta\quad\text{for all }x\in B^n(x_0,r_0)\text{ and } 0<r\leq \frac{1}{54H}\diam f(B^n(x_0,r_0)).\end{equation} Thus, if $f$ is $\delta_*$-almost affine over $\RR^n$ and $\tH_f(\RR^n)\leq \delta_*$, then $f(\RR^n)$ is $(H\delta,\infty)$-Reifenberg flat, i.e.~ $\theta_{f(\RR^n)}(f(x),r)\leq H\delta$ for all $x\in\RR^n$ and $r>0$.
\end{lemma}

The following theorem says that quasisymmetric maps with small constant between Euclidean spaces of the \emph{same} dimension are almost affine when restricted to lower dimensional subspaces.

\begin{theorem}\label{t:qsaa} Suppose $N\geq 2$. For all $\tau>0$, there exists $\tau_*=\tau_*(\tau,N)>0$ such that if $B^N(x,3r)\subset Y\subset \RR^N$ for some $x\in\RR^n$ and $r>0$, $f:Y\rightarrow\RR^N$ is quasisymmetric and $\tH_f(B^N(x,3r))\leq \tau_*$, then $f|_{Y\cap \RR^n}$ is $\tau$-almost affine over $B^n(x,r)$.
\end{theorem}

See \S\ref{s:affine} for the proofs of Lemma \ref{l:rflat} and Theorem \ref{t:qsaa}. At this point, we have collected enough tools to prove Theorems \ref{t:abt} and \ref{t:abt2}.

The final ingredient in the proof of Theorem \ref{t:main} is the following extension theorem.

\begin{theorem}\label{t:beta} Suppose $1\leq n\leq N-1$. For all $\varepsilon>0$, there  exists $\varepsilon_*=\varepsilon_*(\varepsilon,n)>0$ with the following property. If for some $x\in\RR^n$ and $r>0$ a map $f:\RR^N\rightarrow\RR^N$ is $\varepsilon_*$-almost affine over $B^n(x,9r)$, $f|_{B^N(x,3r)}$ is a topological embedding and  $\tH_f(B^N(x,3r))\leq \varepsilon_*$, and there exist a closed set $E\subset B^n(x,r)$ and constants $\gamma_E>0$ and $C_E>0$ such that
\begin{equation}
\label{e:beta-d} \diam E \geq \gamma_E \diam B^n(x,r)
\end{equation} and
\begin{equation}
\label{e:beta-h}
\int_{0}^{r} \tH_f(B^N(y,s))^2\, \frac{ds}{s}\leq C_E\quad\text{for all }y\in E,
\end{equation}
then there exists a quasisymmetric map $F:\RR^n\rightarrow\RR^N$ such that $F|_E= f|_E$, $F$ is $\varepsilon$-almost affine over $\RR^n$, $\tH_F(\RR^n)\leq \varepsilon$, $\diam F(B^n(x,r))\sim_{n,N,\gamma_E} \diam f(B^n(x,r))$, and \begin{equation}
\label{e:beta-c}
\int_0^{\infty} \beta_{F(\RR^n)}(F(y),s)^2\, \frac{ds}{s}\lesssim_{n,N} C_E+\varepsilon^2\quad\text{for all }y\in \RR^n.
\end{equation}
\end{theorem}

The proof of Theorem \ref{t:beta} is somewhat involved, and so, we break the proof into several steps. In \S\ref{s:extend}, we prove general extension theorems for almost affine maps and for quasisymmetric almost affine maps, which are interesting in their own right; see Theorem \ref{t:extend1} and Theorem \ref{t:extend2}. Then we establish beta number estimates on the extensions and verify Theorem \ref{t:beta} in \S\ref{s:extend2}.

\begin{remark}Theorem \ref{t:qsaa} and Theorem \ref{t:beta} were inspired by Tukia and V\"ais\"al\"a's work on extensions of quasisymmetric maps that are close to similarities; see \cite{TV84} and \cite{Vaisala86}.\end{remark}


\subsection{Proofs of Theorem \ref{t:abt}, Theorem \ref{t:abt2}, and Theorem \ref{t:main}} \label{ss:proofs}
The proofs of Theorem \ref{t:abt} and \ref{t:abt2} are very similar. We shall first prove Theorem \ref{t:abt2} and then indicate how to modify the proof for Theorem \ref{t:abt}. We then end with the proof of Theorem \ref{t:main}.

\begin{proof}[Proof of Theorem \ref{t:abt2}] Assume that $1\leq n\leq N-1$ and $H\geq 1$. We will work with certain parameters, chosen as follows. \begin{enumerate}
 \item Pick any $\delta\in(0,\delta_0/H]$ where $\delta_0=\delta_0(n,N)$ is the constant from Theorem \ref{t:reif}.
 \item Let $\delta_{*}=\delta_{*}(\delta)$ be the constant from Lemma \ref{l:rflat} corresponding to $\delta$.
 \item Let $\tau_*=\tau_*(\tau,N)$ denote the constant from Theorem \ref{t:qsaa} corresponding to $\tau=\delta_*$.
\end{enumerate}

Let $f:\RR^N\rightarrow\RR^N$ be a quasiconformal map such that (\ref{e:abt-dini}) holds and suppose $H_f(\RR^n)=H$. We want to show that the quasiplane $f(\RR^n)$ is locally bi-Lipschitz equivalent to subsets of $\RR^n$. Fix any $x_0\in\RR^n$. Then \begin{equation} \label{e:pabt1} \sup_{x\in B^n(x_0,1)}\int_0^1\tH_f(B^N(x,r))^2\,\frac{dr}{r}=:A<\infty\end{equation} by \eqref{e:abt-dini}. In particular, since $\tH_f(B^N(x,r))$ is increasing as a function of $r$, $\tH_f(B^N(x_0,r))\rightarrow 0$ as $r\rightarrow 0$. Hence we can find $0<r_0\leq 1/6$ such that \begin{equation}\label{e:tH1}\tH_f(B^N(x_0,6r_0))\leq \min\{1,\delta_*,\tau_*\}.\end{equation} First off, $f$ is $\delta_*$-almost affine over $B^n(x_0,2r_0)$ by Theorem \ref{t:qsaa}, since $\tH_f(B^N(x_0,6r_0))\leq\tau_*$. Thus, writing $s_0:=(1/540H)\diam f(B^n(x_0,r_0))$, we see that \begin{equation}\label{e:tc1}\theta_{f(\RR^n)}(y,s)\leq H\delta\leq \delta_0\quad\text{for all }y\in f(\RR^n)\cap B^N(f(x_0),10s_0)\text{ and }0<s\leq 10s_0\end{equation} by Lemma \ref{l:rflat}, since $f$ is $\delta_*$-almost affine over $B^n(x_0,2r_0)$ and $\tH_f(B^n(x_0,2r_0))\leq \delta_*$. Next, by (\ref{e:pabt1}) and Corollary \ref{c:bflat} there is a constant $C=C(N,H')>1$ such that \begin{equation}\label{e:dc4} \int_0^{\diam f(B^N(x,r_0))/C} \beta_{f(\RR^n)}(f(x),s)^2\frac{ds}{s} \leq AC\end{equation} for all $x\in B^N(x_0,r_0)$, where $H'=H_f(B^N(x,r_0))\leq 2$ by (\ref{e:tH1}). Hence $C$ actually depends on at most $N$.  We would like to replace $\diam f(B^N(x,r_0))$ in the upper limit of integration in (\ref{e:dc4}) with $\diam f(B^n(x_0,r_0))$. To that end, we note that $f|_{B^N(x_0,6r_0)}$ is $\eta$-quasisymmetric for some control function $\eta$ that depends only on $n$ and $N$, by Corollary \ref{c:wqs2qs} and (\ref{e:tH1}). Thus, by Lemma \ref{l:QS-comp}, \begin{equation}\label{e:d11} \frac{\diam f(B^n(x_0,r_0))}{\diam f(B^N(x_0,6r_0))} \leq 2\eta(6)\eta(1/3) \frac{\diam f(B^N(x,r_0))}{\diam f(B^N(x_0,6r_0))}.\end{equation} Let $C'=2\eta(6)\eta(1/3)$, which depends on at most $n$ and $N$. Then, by (\ref{e:dc4}) and (\ref{e:d11}), \begin{equation} \label{e:dc5} \int_0^{\diam f(B^n(x_0,r_0))/CC'} \beta_{f(\RR^n)}(f(x),s)^2 \frac{ds}{s} \leq AC\end{equation} for all $x\in B^n(x_0,r_0)$. Let $10t_0=\min\{10s_0,\diam f(B^n(x_0,r_0))/CC'\}$. Then, by (\ref{e:tc1}), \begin{equation}\label{e:tc2}\theta_{f(\RR^n)}(y,t)\leq \delta_0\quad\text{for all }y\in f(\RR^n)\cap B^N(f(x_0),10t_0)\text{ and }0<t\leq 10t_0,\end{equation} and, by Lemma \ref{e:check} and (\ref{e:dc5}), \begin{equation}\label{e:dc6} \sup_{y\in f(\RR^n)\cap B^N(x_0,10t_0)}\sum_{k=0}^\infty \betac_{f(\RR^n)}(y,10^{-k}t_0)^2 \leq \frac{400}{\log(10)}AC. \end{equation} By (\ref{e:tc2}), (\ref{e:dc6}), and Theorem \ref{t:bilip}, there exists an $n$-dimensional plane $V$ and an $L^2$-bi-Lipschitz map $g:\RR^N\rightarrow\RR^N$ for some $L=L(n,N,A)$ (with $L\rightarrow 1$ as $A\rightarrow 0$ by Remark \ref{r:Lsmall}) such that \begin{equation*}f(\RR^n)\cap B^N(f(x_0),t_0) = g(V) \cap B^N(f(x_0),t_0).\end{equation*} Therefore, for every $x_0\in\RR^n$ there exists $t_0>0$ such that $f(\RR^n)\cap B^N(f(x_0),t_0)$ is bi-Lipschitz equivalent to a subset of $\RR^n$; that is, $f(\RR^n)$ is locally bi-Lipschitz equivalent to subsets of $\RR^n$.
\end{proof}

\begin{proof}[Proof of Theorem \ref{t:abt}] Let $f:\RR^N\rightarrow\RR^N$ be a quasiconformal map and assume that (\ref{e:bgrt-dini}) holds. We want to show that the quasiplane $f(\RR^n)$ is locally $(1+\delta)$-bi-Lipschitz equivalent to subsets of $\RR^n$ for all $\delta>0$. Fix any $x_0\in\RR^n$. Then\begin{equation*} \int_{0}^{1}\sup_{x\in B^n(x_0,1)} \tH_{f}(B^N(x,r))^{2} \frac{dr}{r} <\infty.\end{equation*} Thus, given any $A>0$, we can find $\rho\in(0,1)$ such that \begin{equation}\label{e:pabt2} \sup_{x\in B^n(x_0,1)}\int_0^{\rho}\tH_f(B^N(x,r))^2\frac{dr}{r}\leq \int_{0}^{\rho}\sup_{x\in B^n(x_0,1)} \tH_{f}(B^N(x,r))^{2} \frac{dr}{r}\leq A.\end{equation} Notice the similarity between (\ref{e:pabt2}) and (\ref{e:pabt1}). By mimicking the proof of Theorem \ref{t:abt2}, we can find $t_0>0$, an $n$-dimensional plane $V$, and an $L^2$-bi-Lipschitz map $g:\RR^N\rightarrow\RR^N$ for some $L=L(n,N,A)$ (with $L\rightarrow 1$ as $A\rightarrow 0$) such that \begin{equation*}f(\RR^n)\cap B^N(f(x_0),t_0) = g(V) \cap B^N(f(x_0),t_0).\end{equation*} Therefore, because $A>0$ can be chosen arbitrarily small, $f(\RR^n)$ is locally $(1+\delta)$-bi-Lipschitz equivalent to subsets of $\RR^n$ for all $\delta>0$.\end{proof}

\begin{proof}[Proof of Theorem \ref{t:main}]

Assume that $2\leq n\leq N-1$. We will work with certain parameters, chosen as follows. \begin{enumerate}
 \item Pick any $\delta\in(0,\delta_0/2]$, where $\delta_0=\delta_0(n,N)$ is the constant from Theorem \ref{t:reif}.
 \item Let $\delta_{*}=\delta_{*}(\delta)$ be the constant from Lemma \ref{l:rflat} corresponding to $\delta$.
 \item Let $\varepsilon_*=\varepsilon_*(\varepsilon,n)$ be the constant from Theorem \ref{t:beta} corresponding to $\varepsilon=\min\{1,\delta_*,C_f^{1/2}\}$.
 \item Let $\tau_*=\tau_*(\tau,N)$ denote the constant from Theorem \ref{t:qsaa} corresponding to $\tau=\varepsilon_*$.
 \item Choose $\rho\leq \min\{\tau_*,\varepsilon_*\}$ sufficiently small such that $\exp(-C_f 2^n/\rho^2)/2<1/2$.
 \end{enumerate}

Let $f:\RR^N\rightarrow\RR^N$ be a quasiconformal map such that for some $C_f>0$ the Carleson condition (\ref{e:carleson}) holds for all $x_0\in\RR^n$ and $r_0>0$. Our goal is to identify big pieces of bi-Lipschitz images of $\RR^n$ in $f(\RR^n)\cap B^N(\xi,s)$ for all $\xi\in f(\RR^n)$ and $s>0$. We shall do this indirectly, starting with a location and scale in the domain.

Fix $x_0\in\RR^n$ and $r_0>0$. Put $\sigma=\exp(-C_f 2^n/\rho^2)/2<1/2$. There exists $27r_1\in (\sigma r_0,r_0/2)$ and $x_1\in B^n(x_0,r_0/2)$ such that $\tH_f(B^N(x_1,27r_1))\leq \rho$, otherwise \begin{align*}
\int_{B^n(x_0,r_0)}&\int_0^{r_0} \tH_f(B^N(x,r))^2\frac{dr}{r}d\Leb^n(x) \\
   &> \int_{B^n(x_0,r_0/2)} \int_{\sigma r_0}^{r_0/2} \rho^2\frac{dr}{r}d\Leb^n(x)=C_f \Leb^n(B^n(x_0,r_0)),
\end{align*}
which violates (\ref{e:carleson}). Consider the set \begin{equation*}G:=\left\{x\in B^n(x_1,r_1):\int_0^{r_1} \tH_f(B^N(x,r))^2 \frac{dr}{r} \leq 2C_f \right\}.\end{equation*} By Chebyshev's inequality and (\ref{e:carleson}), the complement of $G$ in $B^n(x_1,r_1)$ has \begin{equation*}
\Leb^n(B^n(x_1,r_1)\setminus G)
\leq \frac{1}{2C_f}\int_{B^n(x_1,r_1)}\int_0^{r_1} \tH_f(B^N(x,r))^2 \frac{dr}{r}d\Leb^n(x)\leq  \frac{1}{2}\Leb^n(B^n(x_1,r_1)).\end{equation*} Hence $\Leb^n(G)\geq\frac12\Leb^n(B^n(x_1,r_1))$. Since Lebesgue measure is inner regular, we may select a compact set $E\subset G$ such that $\Leb^n(E)\geq \frac{1}{2}\Leb^n(G)$. For the record, since $(\sigma/27)r_0\leq r_1$,\begin{equation}\label{e:areas}\Leb^n(E)\geq \frac14 \Leb^n(B^n(x_1,r_1))\geq \frac14\left(\frac{\sigma}{27}\right)^n\Leb^n(B^n(x_0,r_0))\gtrsim_{n,C_f}\Leb^n(B^n(x_0,r_0))\end{equation} and \begin{equation}\label{e:diams}\diam E\gtrsim_n \diam B^n(x_1,r_1))\gtrsim_{n,C_f} \diam B^n(x_0,r_0).\end{equation}
Now, on one hand, $\tH_f(B^N(x_1,27r_1))\leq \rho\leq \tau_*$. Hence $f|_{B^n(x_1,9r_1)}$ is $\varepsilon_*$-almost affine over $B^n(x_1,9r_1)$ by Theorem \ref{t:qsaa}. On the other hand, we also have $\tH_f(B^N(x_1,3r_1))\leq \rho\leq \varepsilon_*$. Thus, by Theorem \ref{t:beta}, there exists a quasisymmetric map  $F:\RR^n\rightarrow\RR^N$ such that $F|_E=f|_E$, \begin{equation}\label{e:thstar} \tH_F(\RR^n)\leq \varepsilon\leq \min\{1,\delta_*\},\end{equation} \begin{equation}\label{e:aastar} \text{$F$ is $\delta_*$-almost affine over $\RR^n$},\end{equation} \begin{equation} \label{e:dstar} \diam F(B^n(x_1,r_1)) \sim_{n,N} \diam f(B^n(x_1,r_1)),\text{ and}\end{equation}  \begin{equation}\label{e:apple1} \int_0^\infty \beta_{F(\RR^n)}(F(x),s)^2\frac{ds}{s}\lesssim_{n,N} C_f+\varepsilon^2\lesssim C_f\quad\text{for all }x\in\RR^n.\end{equation} By (\ref{e:thstar}), (\ref{e:aastar}), and Lemma \ref{l:rflat}, we conclude that $F(\RR^n)$ is $((1+\varepsilon)\delta,\infty)$-Reifenberg flat, where $(1+\varepsilon)\delta\leq 2\delta\leq \delta_0$. Also by  (\ref{e:apple1}) and Lemma \ref{l:check}, for all $x\in\RR^n$ and $s>0$, \begin{equation*} \sup_{y\in F(\RR^n)\cap B^N(F(x),10s)} \sum_{k=0}^\infty \betac_{F(\RR^n)}(y,10^{-k}s)^2\lesssim_{n,N}C_f.\end{equation*} Therefore, by Corollary \ref{c:bilip}, there exist $L=L(n,N,C_f)>1$ (with $L\rightarrow 1$ as $C_f\rightarrow 0$ by Remark \ref{r:Lsmall}) and an $L^2$-bi-Lipschitz map $g:\RR^N\rightarrow\RR^N$ such that $g(F(\RR^n))=\RR^n$.

We now estimate the $n$-dimensional Hausdorff measure of $f(E)=F(E)$. It is at this point that the restriction $n\geq 2$ enters the discussion. First note that $F$ is quasisymmetric with a control function depending only on $n$ and $N$, by (\ref{e:thstar}) and Corollary \ref{c:wqs2qs}. Thus, since $g$ has bi-Lipschitz constant depending on at most on $n$, $N$, and $C_f$, the composition $h=g\circ F:\RR^n\rightarrow\RR^n$  is $\eta$-quasisymmetric for some control function $\eta$ depending only on $n$, $N$, and $C_f$ (see Figure \ref{f:1}).
\begin{figure}
\begin{picture}(430,150)(0,0)
\put(0,0){\includegraphics[width=.95\textwidth]{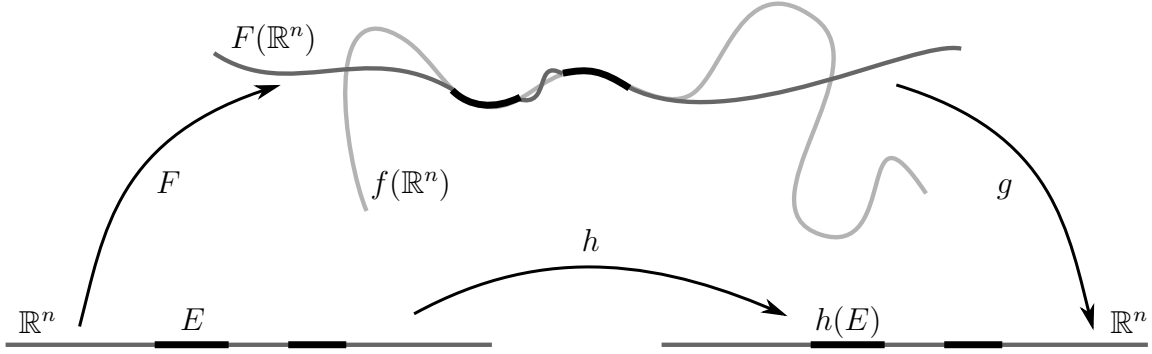}}
\put(66,8){$E$}
\put(57,60){$F$}
\put(138,60){$f(\RR^{n})$}
\put(85,115){$F(\RR^{n})$}
\put(218,38){$h$}
\put(375,60){$g$}
\put(306,8){$h(E)$}
\put(5,7){$\RR^n$}
\put(418,7){$\RR^n$}
\end{picture}
\caption{The light gray set represents the quasiplane $f(\RR^n)$. We extend $f|_{E}$ to an almost affine map $F:\RR^n\rightarrow\RR^{N}$, whose image $F(\RR^n)$ (the dark gray set) is mapped onto $\RR^n$ by a bi-Lipschitz map $g:\RR^N\rightarrow\RR^N$. The black set represents $E$ and its images $F(E)=f(E)$ and $h(E)=g(F(E))=g(f(E))$.}
\label{f:1}
\end{figure}
Hence $h(B^n(x_1,r_1))\subset\RR^n$ has bounded eccentricity (depending only $n$, $N$, and $C_f$) and \begin{equation*}\left(\diam h(B^n(x_1,r_1))\right)^n\sim_{n,N,C_f}\Leb^n(h(B^n(x_1,r_1))).\end{equation*} Pick any closed cube $Q\subset\RR^n$ such that $B^n(x_1,r_1)\subset Q$ and $\Leb^n(B^n(x_1,r_1))\sim_n \Leb^n(Q)$. Since $n\geq 2$ and $h$ is quasiconformal with maximal dilatation $K_h(\RR^n)\leq \eta(1)^{n-1}$ depending only on $n$, $N$ and $C_f$, by Corollary \ref{c:gehring} there exists $q=q(n,N,C_f)>0$ such that \begin{equation*}\begin{split}\frac{\Leb^n(h(E))}{(\diam h(B^n(x_1,r_1)))^n}&\sim_{n,N,C_f}\frac{\Leb^n(h(E))}{\Leb^n(h(B^n(x_1,r_1)))}\\ &\geq \frac{\Leb^n(h(E))}{\Leb^n(h(Q))}\geq\frac{1}{2}\exp\left(-q\frac{\Leb^n(Q)}{\Leb^n(E)}\right) \gtrsim_{n,N,C_f} 1.\end{split}\end{equation*} Since $g$ bi-Lipschitz with constant depending only on $n$, $N$ and $C_f$, it follows that \begin{equation}\label{e:pause1}\frac{\Haus^n(F(E))}{\left(\diam F(B^n(x_1,r_1))\right)^n} \gtrsim_{n,N,C_f} 1.\end{equation} Thus, by (\ref{e:dstar}) and (\ref{e:pause1}), we obtain \begin{equation}\label{e:apple2}\Haus^n(f(E))=\Haus^n(F(E))\gtrsim_{n,N,C_f}\left(\diam f(B^n(x_1,r_1))\right)^n.\end{equation} We would like to replace $\diam f(B^n(x_1,r_1))$ in (\ref{e:apple2}) by $\diam f(B^n(x_0,r_0))$. To that end, note that the restriction $f|_{\RR^n}$ is quasisymmetric with a control function depending only on $n$, $N$, and $H:=H_f(\RR^n)$ by Corollary \ref{c:wqs2qs}. Thus, by (\ref{e:diams}) and Lemma \ref{l:QS-comp}, \begin{equation}\label{e:diams2} \diam f(B^n(x_1,r_1))\gtrsim_{n,N,C_f,H} \diam f(B^n(x_0,r_0)).\end{equation} Therefore, \begin{equation}\label{e:apple3} \Haus^n(f(E))\gtrsim_{n,N,C_f,H}\left(\diam f(B^n(x_0,r_0))\right)^n.\end{equation}
We have argued that for all $x_0\in \RR^n$ and $r_0>0$ there exist a closed set $E\subset B^n(x_0,r_0)$ and a $L(n,N,C_f)$-bi-Lipschitz map $g:f(E)\rightarrow\RR^n$ (with $L\rightarrow 1$ as $C_f\rightarrow 0$) such that (\ref{e:apple3}) hold.

To finish the proof of the theorem, we now show that $f(\RR^n)$ has big pieces of bi-Lipschitz images of $\RR^n$. Let $\xi\in f(\RR^n)$ and $s>0$ be given. Put $x=f^{-1}(\xi)\in\RR^n$ and set \begin{equation*}r=\max\left\{t:f(B^n(x,t))\subset B^N(\xi,s)\right\}.\end{equation*} Since $r$ is maximal, there exists $y\in B^n(x,r)$ such that $|f(y)-f(x)|=s$. As we argued above, there exists $E\subset B^n(x,r)$ such that $f(E)\subset f(\RR^n)\cap B^N(\xi,s)$ is $L(n,N,C_f)$-bi-Lipschitz equivalent to a subset of $\RR^n$ and \begin{equation*}  \Haus^n(f(E)) \gtrsim_{n,N,C_f,H} \left(\diam f(B^n(x,r))\right)^n \geq |f(y)-f(x)|^n\geq s^n.\end{equation*} Therefore, since $\xi\in f(\RR^n)$ and $s>0$ were arbitrary, $f(\RR^n)$ has big pieces of bi-Lipschitz images of $\RR^n$ with BPBI constants depending on at most $n$, $N$, $C_f$, and $H_F(\RR^n)$. \end{proof}


\section{Distortion of beta numbers by quasisymmetric maps}
\label{s:flatness}

In this section, we examine the distortion of beta numbers by weakly quasisymmetric maps. Our primary goal is to prove Lemma \ref{l:bflat}, which for convenience we now restate.

\begin{lemma}\label{l:bflat2} Suppose that $1\leq n\leq N-1$. Let $V$ be an $n$-dimensional plane in $\RR^N$, let $v\in V$ and let $e$ be a unit vector in $\RR^N$. For any topological embedding $f:B^N(v,2r)\rightarrow\RR^N$, \begin{equation*}\beta_{f(V)}\left(f(v),\frac{1}{2}|f(v+re)-f(v)|\right) \leq 72N \tH_f(B^N(v, 2r)).\end{equation*}\end{lemma}

\begin{proof}[Proof of Lemma \ref{l:bflat} / Lemma \ref{l:bflat2}] Without loss of generality, by applying a translation, rotation, and dilation in the domain, and a dilation in the image, we may assume that $1\leq n\leq N-1$, $V=\RR^n$, $v=0$ and $r=1$, and $f:B^N(0,2)\rightarrow\RR^N$ is an embedding such that $|f(e)-f(0)|=1$ for some unit vector $e$. Also, by applying a translation in the image, we may assume that \begin{equation*} \sum_{i=1}^N f(e_i)+f(-e_i)=0.\end{equation*} Fix $0<\delta\leq 1/4$ to be specified later, ultimately depending only on $N$. If $\tH_f(B^N(0,2)) > \delta$, then $\beta_{f(\RR^n)}(f(0),1/2) \leq 1 \leq (1/\delta)\tH_f(B^N(0,2))$ trivially. Thus, to continue, we assume that \begin{equation*}\tH_f(B^N(0,2))=:\varepsilon\leq \delta.\end{equation*} Because $|f(e)-f(0)|=1$, $f$ is a topological embedding, and $\tH_f(B^N(0,1))\leq \varepsilon\leq 1/4$, it follows that \begin{equation}\label{e:bf2a} \frac{1}{1+\varepsilon}\leq |f(e')-f(0)| \leq 1+\varepsilon\quad\text{for every unit vector }e',\end{equation} \begin{equation}\label{e:bf2b} B^N\left(f(0),\frac{1}{1+\varepsilon}\right) \subset f(B^N(0,1))\subset B^N(f(0),1+\varepsilon),\end{equation}
and \begin{equation} \label{e:bf2c} \frac{2}{1+\varepsilon}\leq \diam f(B^N(0,1)) \leq 2(1+\varepsilon)\leq 5/2.\end{equation}
For all $1\leq i\leq N$, put \begin{equation*} y_i:= \frac{f(e_i)+f(-e_i)}{2}\quad\text{and}\quad z_i:=\frac{f(e_i)-f(-e_i)}{2}.\end{equation*} We note that $y_i\pm z_i = f(\pm e_i)$.
Let $A:\RR^{N}\rightarrow \RR^{N}$ be the unique affine map such that
\begin{equation}
A(0)=\frac{1}{N}\sum_{i=1}^{N}y_i=0,\quad A(e_{i})=z_i\quad \text{for all }1\leq i\leq N.
\label{e:A}
\end{equation} We will show that $A(\RR^n)$ is an $n$-dimensional plane and use $A(\RR^n)$ to estimate $\beta_{f(\RR^n)}(f(0),1/2)$.

To start, we show that the vectors $A(e_i)$ and $A(e_j)$ are almost orthogonal for all $1\leq i,j\leq N$, $i\neq j$.
Let $x\in e_i^\perp\cap B^N(0,1)$. Since $\tH_f(B^N(0,1))\leq \varepsilon$ and $|x-e_i|=|x-(-e_i)|$, we have
\begin{equation*} \frac{1}{1+\varepsilon}\leq \frac{|f(x)-f(e_{i})|}{|f(x)-f(-e_{i})|}\leq 1+\varepsilon.\end{equation*}
Hence, by the polarization identity,
\begin{equation*}\begin{split}
|\langle z_i,f(x)-y_i\rangle|&=\frac{1}{4}\left||f(x)-f(-e_i)|^{2}-|f(x)-f(e_i)|^{2}\right|\\&\leq \frac{1}{4}((1+\varepsilon)^{2}-1)|f(x)-f(e_i)|\leq \frac{45}{32}\varepsilon\leq 1.5\varepsilon,
\end{split}\end{equation*} where in the last line we used the estimates $\varepsilon\leq 1/4$ and $\diam f(B^N(0,1))\leq 5/2$. In particular, for all $1\leq j\leq N$, $j\neq i$, we have
\begin{equation*}|\langle z_i,f(\pm e_{j})-y_i\rangle|\leq 1.5\varepsilon.\end{equation*} Hence $|\langle z_i, y_j-y_i \rangle|\leq 1.5\varepsilon$, as well. Averaging over all $1\leq j\leq N$, we obtain \begin{equation*}|\langle z_i, A(0)-y_i\rangle|\leq 1.5\varepsilon.\end{equation*}
Thus, for all $x\in e_{i}^{\perp}\cap B^N(0,1)$,
\begin{equation*}
|\langle z_i,f(x)-A(0)\rangle|\\
\leq |\langle z_i,f(x)-y_i \rangle| + |\langle z_i, A(0)-y_i \rangle|\leq 3\varepsilon.
\end{equation*} Recall that $A(0)=0$, by assumption. Therefore,
\begin{equation}
 |\langle A(e_{i}),f(x) \rangle|= |\langle z_i,f(x) \rangle|\leq 3\varepsilon\quad \text{for all }x\in e_{i}^{\perp}\cap B^N(0,1),
 \label{e:Aeifx}
 \end{equation}
and
\begin{equation}
|\langle A(e_{i}),A(e_{j})\rangle|\leq \frac{|\langle A(e_{i}),f(e_{j})\rangle|+|\langle A(e_{i}),f(-e_{j})\rangle|}{2}\leq 3\varepsilon\quad\text{for all }i\neq j.
\label{e:AeiAej}
\end{equation} That is, the vectors $A(e_i)$ and $A(e_j)$ are almost orthogonal for all $1\leq i\leq j\leq N$, $i\neq j$.

Next, we claim that \begin{equation}\label{e:Aei>1-e^2} (1-\varepsilon)^2\leq  |A(e_{i})|\leq 1+\varepsilon \quad\text{for all }1\leq i\leq N. \end{equation} To see this, fix $1\leq i\leq N$. For the upper bound, recall that $\diam f(B^N(0,1))\leq 1+\varepsilon$. Hence \begin{equation*}|A(e_{i})|=\frac{|f(e_{i})-f(-e_{i})|}{2}\leq 1+\varepsilon.\end{equation*}
For the lower bound, write $r_{\pm}=|f(0)-f(\pm e_{i})|$. Since $\tH_f(B^N(0,2))\leq \varepsilon$, we know that
\begin{equation*} |f(y)-f(\pm e_{i})|\geq (1+\varepsilon)^{-1}r_{\pm} \quad\text{for all }y\in \partial B^N(\pm e_{i},1).\end{equation*}
Hence $f(B^N(\pm e_{i},1))\supseteq B^N(f(\pm e_{i}),(1+\varepsilon)^{-1}r_{\pm})=:B^N_\pm$, because $f$ is a homeomorphism onto its image.
Moreover, \begin{equation*}f(B^N(e_{i},1))\cap f(B^N(-e_{i},1))=\{f(0)\},\end{equation*} so the balls $B^N_+$ and $B^N_-$ intersect in exactly one point. It follows that
\begin{equation*}|f(e_{i})-f(-e_{i})|\geq (1+\varepsilon)^{-1}(r_{+}+r_{-}).\end{equation*}
Recalling that $r_{\pm}\geq  (1+\varepsilon)^{-1}$ by (\ref{e:bf2a}), we conclude that
\[ |A(e_{i})| =\frac{|f(e_{i})-f(-e_{i})|}{2} \geq \frac{ (1+\varepsilon)^{-1}(r_{+}+r_{-})}{2}\geq  (1+\varepsilon)^{-2}\geq (1-\varepsilon)^2,\]
where the last inequality holds, since $1\geq (1-\varepsilon^{2})^{2}=(1-\varepsilon)^{2}(1+\varepsilon)^{2}$. Thus,  \eqref{e:Aei>1-e^2} holds.

We now examine how $A$ distorts the length of arbitrary vectors. Let $v\in\RR^N$, and expand $v=\sum_{i=1}^N v_ie_i$. If $|v|=1$, then \begin{align*}
\left||A(v)|^{2}-1\right|
 &=\left| \sum_{i\neq j}\langle A(e_{i}),A(e_{j}) \rangle v_{i}v_{j}+\sum_{i=1}^{N}(|A(e_{i})|^{2}-1)v_{i}^{2}\right|\\
 &\leq
 3\varepsilon N+ (1-(1-\varepsilon)^4)\leq 3N\varepsilon+4\varepsilon+4\varepsilon^3 \leq (3N+4.25)\varepsilon\leq 6N\delta
\end{align*}  by  \eqref{e:AeiAej} and \eqref{e:Aei>1-e^2}, and the bounds $\varepsilon\leq \delta \leq 1/4$ and $2\leq N$. By homogeneity, it follows that \begin{equation*}
\sqrt{1-6N\delta}  \leq \frac{|A(v)|}{|v|} \leq \sqrt{1+6N\delta} \quad\text{ for all }v\in\RR^N.
\end{equation*} In particular, stipulating that $6N\delta=3/4$ (that is, $\delta=1/8N$),
\begin{equation} \label{e:Av/v}
\frac{1}{2} \leq \frac{|A(v)|}{|v|} \leq \frac{\sqrt{7}}{2}\quad\text{for all }v\in\RR^N.\end{equation} Therefore, $A:\RR^N\rightarrow\RR^N$ is invertible and $A(\RR^n)$ is an $n$-dimensional plane in $\RR^N$.

Let $\xi\in f(\RR^n)\cap B^N(f(0),1/2)$. Then $\xi=f(x)$ for some $x\in B^n(0,1)$ by \eqref{e:bf2b}. Since $A$ is invertible, we can find a unique $y\in\RR^N$ such that $A(y)=f(x)$. Then, by \eqref{e:A} and \eqref{e:Av/v},
\begin{equation*}\begin{split}
 |y| &\leq 2|A(y)|
=2|f(x)-A(0)|
 = \frac{2}{N}\sum_{i=1}^{N}\left| f(x)-\frac{f(e_{i})+f(-e_{i})}{2} \right|\\
 &\leq \frac{1}{N}\sum_{i=1}^{N}(|f(x)-f(e_{i})|+|f(x)-f(-e_{i})|)
 \leq 2\diam f(B^N(0,1))
\leq 5.
\end{split}\end{equation*} Write $y=u+v$ where $u\in \RR^n$ and $v\in (\RR^n)^\perp$, and expand $u=\sum_{i=1}^n u_ie_i$ and $v=\sum_{j=n+1}^N v_je_j$. Then
\begin{align*}
|f(x)-A(u)|^{2}
& =\langle f(x)-A(u),f(x)-A(u)\rangle
=\langle f(x)-A(u),A(v)\rangle\\
& =\sum_{j=n+1}^{N} \langle f(x),A(e_{j})\rangle v_{j}-\sum_{i=1}^{n}\sum_{j=n+1}^{N}\langle A(e_{i}),A(e_{j})\rangle u_{i}v_{j}.
\end{align*}
Thus, by \eqref{e:Aeifx} and \eqref{e:AeiAej},
\begin{equation*}
|f(x)-A(u)|^2\leq \sum_{i=n+1}^{N}3\varepsilon|v_{i}|+\sum_{i=1}^{n}\sum_{j=n+1}^{N}3\varepsilon |u_{i}||v_{j}|\leq 3\varepsilon(N-n)^{1/2}|v|\left(1+n^{1/2}|u|\right).
\end{equation*}
Note that $|v|\leq 2|A(v)|=2|f(x)-A(u)|$, $|u|\leq |y|\leq 5$, and $1\leq n^{1/2}$. Hence
\begin{equation*}\dist(\xi,A(\RR^n))\leq |f(x)-A(u)|
\leq 36N\varepsilon\quad\text{for all }\xi\in f(\RR^n)\cap B^N(f(0),1/2). \end{equation*} Therefore, $\beta_{f(\RR^n)}(f(0),1/2)\leq 72 N\varepsilon = 72 N\tH_f(B^N(0,2))$.
\end{proof}

Our next task is to derive Corollary \ref{c:bflat}, which for convenience we now restate.

\begin{corollary} \label{c:bflat2} Suppose that $1\leq n\leq N-1$ and $H\geq 1$. There is $C=C(N,H)>1$ such that if $z\in\RR^n$, $t>0$, $f:\oB^N(z,2t)\rightarrow\RR^N$ is quasiconformal, and $H_f(B^N(z,t))\leq H$, then
\begin{equation}\label{e:dc2} \int_0^{\diam f(B^N(z,t))/C} \beta_{f(\RR^n)}(f(z),s)^2\frac{ds}{s}\leq C \int_0^{t} \tH_f(B^N(z,s))^2\frac{ds}{s}.\end{equation}\end{corollary}

\begin{proof}[Proof of Corollary \ref{c:bflat} / Corollary \ref{c:bflat2}] Let $1\leq n\leq N-1$ and  $H\geq 1$ be given.  Assume that $f:\oB^N(z,2t)\rightarrow\RR^N$ is quasiconformal and  $H_f(B^N(z,t))\leq H$ for some $z\in\RR^n$ and $t>0$. Then $K:=K_f(\oB^N(z,t))\leq H^{N-1}$ and the inverse $g=f|_{\oB^N(z,t)}^{-1}:f(\oB^N(z,t))\rightarrow \oB^N(z,t)$ is also a $K$-quasiconformal map. Set $\alpha:= K^{1/(1-N)}\leq 1/H$ and \begin{equation*} M:=\sup_{|w-z|=t} |f(w)-f(z)|.\end{equation*} We remark that $M\leq \diam f(B^N(z,t))\leq 2M$, since $\partial f(B^N(z,t))=f(\partial B^N(z,t))$.  Because $f$ is weakly $H$-quasisymmetric,  $f(B^N(z,t))\supseteq B^N(f(z),M/H)$. By Theorem \ref{t:holder}, there exists a constant $A=A(N,K)=A(N,H)\geq 1$ such that \begin{equation*} |f(x)-f(y)| \leq AM\left| \frac{x}{t}-\frac{y}{t}\right|^\alpha \quad\text{for all }x,y\in B^N(z,t/2).\end{equation*} In particular, \begin{equation*} |f(x)-f(z)| \leq AM\left(\frac{r}{t}\right)^\alpha\leq \frac{M}{2H}\quad\text{ for all }x\in B^N(z,r),\end{equation*} for all $r>0$ such that \begin{equation}\label{e:r-up} r\leq \frac{t}{(2AH)^{1/\alpha}}\leq \frac{t}{2}.\end{equation}

Let $r>0$ satisfy (\ref{e:r-up}). By Lemma \ref{l:bflat2}, \begin{equation*} \beta_{f(\RR^n)}\left(f(z),\frac{1}{2}|f(z+re_1)-f(z)|\right) \lesssim_N \tH_f(B^N(z,2r)).\end{equation*} Let us bound $r$ from above by a power of $u:=|f(z+re_1)-f(z)|/2$. Write $\xi=f(z+re_1)$ and $\zeta=f(z)$. Then $\xi \in B^N(\zeta,M/2H)$, since $r$ satisfies (\ref{e:r-up}). By Theorem \ref{t:holder}, \begin{equation*} r=|g(\xi)-g(\zeta)| \leq At\left|\frac{\xi}{M/H}-\frac{\zeta}{M/H}\right|^\alpha
=At\left(\frac{2Hu}{M}\right)^\alpha. \end{equation*} Thus, since $\tH_f(B^N(z,s))$ is increasing in $s$, we have \begin{equation}\label{e:betaQ}  \beta_{f(\RR^n)}\left(f(z),u\right) \lesssim_N \tH_f\left(B^N\left(z,2At\left(\frac{2Hu}{M}\right)^\alpha \right)\right)=\tH_f\left(B^N\left(z,Qt(u/M)^\alpha\right)\right), \end{equation} where $Q:=2A(2H)^{\alpha}$ depends only on $N$ and $H$. Note that (\ref{e:betaQ}) holds for all $u>0$ such that \begin{equation}\label{e:u-up} At\left(\frac{2Hu}{M}\right)^\alpha \leq \frac{t}{(2AH)^{1/\alpha}},\end{equation} because (\ref{e:u-up}) ensures that $u$ comes from some $r$ satisfying (\ref{e:r-up}).

Hence, for all $a>0$ sufficiently small, \begin{align*}
\int_0^a \beta_{f(\RR^n)}(f(z),u)^2\,\frac{du}{u}
  &\lesssim_N \int_0^a \tH_f(B^N(z,Qt(u/M)^\alpha))^2 \frac{du}{u} \\
  &= \frac{1}{\alpha} \int_0^{Qt(a/M)^{\alpha}} \tH_f(B^N(z,s))^2\frac{ds}{s},
\end{align*} where the equality follows from the change of variables $s=Qt(u/M)^{\alpha}$, $ds / s = \alpha\, du/u$. Taking $a$ to be of the form $a=\diam f(B^N(z,t))/C$ with $C$ large, we obtain \begin{equation*} \int_0^{\diam f(B^N(z,t))/C} \beta_{f(\RR^n)}(f(z),s)^2\,\frac{ds}{s} \lesssim_{N,K} \int_{0}^{Q(2/C)^\alpha t} \tH_f(B^N(z,s))^2\,\frac{ds}{s}.\end{equation*} Therefore, (\ref{e:dc2}) holds for $C>1$ sufficiently large depending only on $N$ and $H$.\end{proof}


\section{Estimates for compatible affine maps and almost affine maps}
\label{s:estimates}

To start the section, we record useful estimates for compatible affine maps (Lemma \ref{l:main-est}) and for almost affine maps (Lemma \ref{l:post-est}). Next we show that almost affine maps with small constant are H\"older continuous (Lemma \ref{l:Holder}). In Lemma \ref{l:inradius}, we make estimates on the diameter, inradius and local flatness of the images of balls under almost affine maps. To end the section, we give a pair of lemmas (Lemmas \ref{l:stable} and \ref{l:adapt}), which enable us to replace an arbitrary family compatible affine maps approximating an almost affine map with a family of compatible affine maps that satisfy additional nice properties.

For all $\varepsilon>0$, define $T_\varepsilon:[1,\infty)\rightarrow[1,\infty)$ by \begin{equation}\label{e:T} T_\varepsilon(t)=(2\log_2(t)+1)t^{2\log_2(1+\varepsilon)}\quad\text{for all }t\geq 1.\end{equation} Observe that $T_{\varepsilon}(t)$ is increasing in $\varepsilon$ and $t$; that is,  $T_{\varepsilon_1}(t_1)\leq T_{\varepsilon_2}(t_2)$ for all $0<\varepsilon_1\leq \varepsilon_2$ and $1\leq t_1\leq t_2$. For all $x,y\in \RR^n$ and $r,s>0$, define

\begin{equation}\label{e:tau} \tau(x,r,y,s):= \frac{\max\{r,s,2|x-y|\}}{\min\{r,s\}}. \end{equation}

\begin{lemma}[Estimates for compatible families of affine maps] \label{l:main-est} Let $E\subset\RR^n$ and $\varepsilon>0$. If $\mathcal{A}$ is an $\varepsilon$-compatible family of affine maps over $E$ (see \S\ref{ss:almost}), then for all $x,y\in E$ and $r,s>0$,
\begin{equation}\label{e:pre-a} \Adiff{x,r}{y,s} \leq T_\varepsilon(\tau)\,\varepsilon \Amin{x,r}{y,s}\end{equation} and \begin{equation}\label{e:pre-b} \Amax{x,r}{y,s} \leq  (1+T_\varepsilon(\tau)\varepsilon)\Amin{x,r}{y,s}\end{equation}
where $\tau=\tau(x,r,y,s)$.  In particular, if $\varepsilon\leq a$ and $\tau\leq a$ for some $a\geq 1$, then \begin{equation} \label{e:pre-c} \Adiff{x,r}{y,s} \lesssim_a \varepsilon \Amin{x,r}{y,s}\end{equation} and \begin{equation} \label{e:pre-d} \Amax{x,r}{y,s}\lesssim_a \Amin{x,r}{y,s}.\end{equation}
\end{lemma}

\begin{proof} Suppose that $\mathcal{A}$ is an $\varepsilon$-compatible family of affine maps over $E\subset\RR^n$. We shall first establish an auxiliary estimate: \begin{equation}\label{e:pre1} \Adiff{x,r}{x,2^kr}\leq \left((1+\varepsilon)^k-1\right)\Amin{x,r}{x,2^kr}\quad\text{for all }x\in E\text{ and }k\geq 0.\end{equation} Fix $x\in E$. For all $k\geq 0$,
$\Adiff{x,2^kr}{x,2^{k+1}r}\leq\varepsilon \Amin{2^kr}{2^{k+1}r}$, because $\mathcal{A}$ is $\varepsilon$-compatible. Hence, by the triangle inequality, \begin{equation*} \Amax{2^kr}{2^{k+1}r}\leq (1+\varepsilon)\Amin{x,2^kr}{x,2^{k+1}r}.\end{equation*} By induction, it follows that $(1+\varepsilon)^{-k}\Anorm{x,r}\leq \Anorm{x,2^kr}\leq (1+\varepsilon)^k \Anorm{x,r}$ for all integers $k\geq 0$. We now estimate $\Adiff{x,r}{x,2^kr}$. Since this expression vanishes trivially when $k=0$, we may assume that $k\geq 1$. Expanding the difference as a telescoping sum yields \begin{equation}\begin{split} \label{e:pre2}\Adiff{x,r}{x,2^kr}
 \leq \sum_{j=0}^{k-1} \Adiff{x,2^jr}{x,2^{j+1}r}
 \leq \sum_{j=0}^{k-1} \varepsilon(1+\varepsilon)^j\Anorm{x,r}\\
 = \varepsilon\left(\frac{(1+\varepsilon)^k-1}{(1+\varepsilon)-1}\right)\Anorm{x,r}
 = \left((1+\varepsilon)^k-1\right)\Anorm{x,r}.
\end{split}\end{equation} Similarly, telescoping in the other direction, \begin{equation}\begin{split} \label{e:pre3} \Adiff{x,2^kr}{x,r}
 \leq \sum_{l=0}^{k-1} \Adiff{x,2^{k-l}r}{x,2^{k-l-1}r}
 \leq \sum_{l=0}^{k-1} \varepsilon(1+\varepsilon)^l\Anorm{x,2^kr}\\
 = \varepsilon\left(\frac{(1+\varepsilon)^k-1}{(1+\varepsilon)-1}\right)\Anorm{x,2^kr}
 =\left((1+\varepsilon)^k-1\right)\Anorm{x,2^kr}.
\end{split}\end{equation} Therefore, \eqref{e:pre1} holds by \eqref{e:pre2} or \eqref{e:pre3}, according to whether $\Anorm{x,r}$ or $\Anorm{x,2^kr}$ is smaller, respectively.

We now aim to prove \eqref{e:pre-a}. Fix $x,y\in E$ and $r,s>0$. Without loss of generality, we assume that $r\leq s$. Define $k\geq 0$ to be the unique integer such that $2^k r\leq s<2^{k+1}r$, and let $l\geq 0$ be the smallest nonnegative integer such that $|x-y|< 2^ls$. By two applications of \eqref{e:pre1}:
\begin{equation*} \Adiff{x,r}{x,2^{k+l}r}\leq \left((1+\varepsilon)^{k+l}-1\right) \Amin{x,r}{x,2^{k+l}r}\end{equation*} and
\begin{equation*} \Adiff{y,s}{y,2^ls} \leq \left((1+\varepsilon)^{l}-1\right) \Amin{y,s}{y,2^ls}.\end{equation*} Also, since $\mathcal{A}$ is $\varepsilon$-compatible, $|x-y|<2^l s=\max\{2^{k+l}r,2^ls\}$ and $\frac12\leq (2^{k+l}r)/(2^ls)<1$, \begin{equation*} \Adiff{x,2^{k+l}r}{y,2^ls}\leq \varepsilon\min\{\|A'_{x,2^{k+l}r}\|, \|A'_{y,2^ls}\|\}.\end{equation*} By the triangle inequality, it follows that \begin{equation*}\Amax{x,2^{k+l}r}{y,2^ls}\leq (1+\varepsilon)\Amin{x,2^{k+l}r}{y,2^ls}.\end{equation*}
Combining the previous four displayed equations yields \begin{equation*} \Adiff{x,r}{y,s} \leq \left[(2+\varepsilon) \left((1+\varepsilon)^{k+l}-1\right) +\varepsilon\right] \Amin{x,2^{k+l}r}{y,2^ls}.\end{equation*} Next, by (\ref{e:pre1}) and the triangle inequality, we have $\Anorm{x,2^{k+l}r}\leq (1+\varepsilon)^{k+l}\Anorm{x,r}$ and $\Anorm{y,2^ls}\leq (1+\varepsilon)^{l}\Anorm{y,s}\leq (1+\varepsilon)^{k+l}\Anorm{y,s}$. Hence
\begin{equation*} \Adiff{x,r}{y,s}\leq(1+\varepsilon)^{k+l}\left[(2+\varepsilon)\left((1+\varepsilon)^{k+l}-1\right)+\varepsilon\right]\min\{\|A_{x,r}'\|,\|A_{y,s}'\|\}.
\end{equation*} Thus, invoking the mean value theorem (for the function $t\mapsto t^{k+l}$ between $t=1$ and $t=1+\varepsilon$) and noting that $(2+\varepsilon)/(1+\varepsilon)\leq 2$ for all $\varepsilon>0$, we conclude that
\begin{align*}\Adiff{x,r}{y,s}
&\leq (1+\varepsilon)^{k+l} \left[ (2+\varepsilon)\varepsilon(k+l)(1+\varepsilon)^{k+l-1} + \varepsilon \right] \Amin{x,r}{y,s}\\
&\leq (1+\varepsilon)^{2(k+l)} \left[2(k+l)+1\right] \varepsilon\Amin{x,r}{y,s} .\end{align*} Examining the definitions of $k$ and $l$, we see that $k\leq \log_2 (s/r)$, $l=0$ if $|x-y|<s$, and $l\leq \log_2 (2|x-y|/s)$ if $|x-y|\geq s$. Either way, $k+l \leq \log_2(\max\{s,2|x-y|\}/r)=:\log_2(\tau)$ and $(1+\varepsilon)^{2(k+l)} \leq (1+\varepsilon)^{2\log_2(\tau)} = \tau^{2\log_2(1+\varepsilon)}$. This establishes \eqref{e:pre-a} and \eqref{e:pre-b} follows from the triangle inequality

To finish, suppose that $\varepsilon\leq a$ and $\tau\leq a$ for some $a\geq 1$. Then, by \eqref{e:pre-a}, \begin{equation*} \Adiff{x,r}{y,s}\leq \left(2\log_2(a)+1\right)a^{2\log_2(1+a)}\varepsilon\Amin{x,r}{y,s}.\end{equation*} This establishes \eqref{e:pre-c} and \eqref{e:pre-d} follows from the triangle inequality.\end{proof}

\begin{lemma}[Estimate for affine maps approximating an almost affine map] Let $(f,E,\mathcal{A})$ be $\varepsilon$-almost affine for some $E\subset\RR^n$ and $\varepsilon>0$ (see \S\ref{ss:almost}). Let $x,y\in E$, let $r,s>0$ and let $a\geq 1$.  If $\varepsilon\leq a$, $|x-y|\leq a \max\{r,s\}$ and $\dist(z,\{x,y\})\leq a \max\{r,s\}$ for some $z\in\RR^n$,  then  \label{l:post-est} \begin{equation}\label{e:post-a}|A_{x,r}(z)-A_{y,s}(z)| \lesssim_a T_\varepsilon(\tau)\,\varepsilon \Amin{x,r}{y,s}\max\{r,s\},\end{equation} where $\tau=\tau(x,r,y,s)$. In particular, if in addition $\tau\leq a$, then \begin{equation}\label{e:post-b}|A_{x,r}(z)-A_{y,s}(z)| \lesssim_a \varepsilon \Amin{x,r}{y,s}\max\{r,s\},\end{equation}
\end{lemma}

\begin{proof} Let $E\subset\RR^n$, let $\varepsilon>0$, and let $(f,E,\mathcal{A})$ be $\varepsilon$-almost affine. Let $x,y\in E$ and $r,s>0$. Without loss of generality, assume that $r\leq s$. Let $a\geq 1$ and $z\in \RR^n$ be given, and assume that $\varepsilon\leq a$, $|x-y|\leq as$, and $\dist(z,\{x,y\})\leq as$. By the triangle inequality, \begin{equation*}|A_{x,r}(z)-A_{y,s}(z)|\leq |A_{x,r}(z)-A_{y,as}(z)| + |A_{y,as}(z)-A_{y,s}(z)|.\end{equation*} We estimate the two terms separately. First, expanding $A_{y,as}(z)=A'_{y,as}(z-y)+A_{y,as}(y)$ and $A_{y,s}(z)=A'_{y,s}(z-y)+A_{y,s}(y)$, we obtain \begin{align*} |A_{y,as}(z) - A_{y,s}(z)| &\leq |A'_{y,as}(z-y)-A'_{y,s}(z-y)|+|A_{y,as}(y)-f(y)|+|f(y)-A_{y,s}(y)| \\
&\leq \Adiff{y,as}{y,s} |z-y| + \varepsilon\Anorm{y,as} as + \varepsilon\Anorm{y,s} s\\
&\lesssim_a \left(\Adiff{y,as}{y,s}+\varepsilon\Anorm{y,as}+\varepsilon\Anorm{y,s}\right)s,\end{align*} since $y\in B^n(y,as)\cap B^n(y,s)$ and $(f,E,\mathcal{A})$ is $\varepsilon$-almost affine. But $\Anorm{y,as}\sim_a \Anorm{y,s}$  and $\Adiff{y,as}{y,s}\lesssim_a \varepsilon\Anorm{y,s}$  by \eqref{e:pre-c} and \eqref{e:pre-d}, since $\varepsilon\leq a$ and $\tau(y,as,y,s)=a$. Hence \begin{equation}\label{e:post1}|A_{y,as}(z)-A_{y,s}(z)| \lesssim_a  \varepsilon \Anorm{y,s}s. \end{equation}
Similarly, expanding $A_{x,r}(z)=A'_{x,r}(z-x)+A_{x,r}(x)$ and $A_{y,as}(z)=A'_{y,as}(z-x)+A_{y,as}(x)$, \begin{align*}
|A_{x,r}(z)-A_{y,as}(z)|&\leq |A'_{x,r}(z-x)-A'_{y,as}(z-x)| + |A_{x,r}(x)-f(x)|+|f(x)-A_{y,as}(x)|\\
&\leq \Adiff{x,r}{y,as}|z-x| + \varepsilon\Anorm{x,r} r + \varepsilon\Anorm{y,as}as\\
&\lesssim_a \left(\Adiff{x,r}{y,as}+\varepsilon\Anorm{x,r}+\varepsilon\Anorm{y,as}\right)s,
\end{align*} because $x\in B^n(x,r)\cap B^n(y,as)$ and $(f,E,\mathcal{A})$ is $\varepsilon$-almost affine. By the triangle inequality and the estimates for $\Adiff{y,as}{y,s}$, $\Anorm{y,as}$ and $\Anorm{y,s}$ from above, it follows that \begin{equation} \label{e:post2} |A_{x,r}(z)-A_{y,as}(z)| \lesssim_a \left(\Adiff{x,r}{y,s}+\varepsilon\Anorm{x,r}+\varepsilon\Anorm{y,s}\right)s.\end{equation}
Now, by Lemma \ref{l:main-est} (\ref{e:pre-a}) and (\ref{e:pre-b}), writing $T:=T_\varepsilon(\tau)$, $\tau=\tau(x,r,y,s)$ we have \begin{equation}\label{e:post3}\|A'_{x,r}-A'_{y,s}\|\leq T\varepsilon\Amin{x,r}{y,s}\end{equation} and  \begin{equation} \label{e:post4} \Amax{x,r}{y,s}\leq (1+T\varepsilon) \ \Amin{x,r}{y,s}\lesssim_a T\Amin{x,r}{y,s},\end{equation} because $\varepsilon\leq a$ and $1\leq T$. Thus, put together, (\ref{e:post2}), (\ref{e:post3}), and (\ref{e:post4}) give \begin{equation}\label{e:post5}
|A_{x,r}(z)-A_{y,as}(z)| \lesssim_a T\varepsilon\Amin{x,r}{y,s}s.\end{equation} Therefore, combining (\ref{e:post1}), (\ref{e:post4}) and (\ref{e:post5}), we obtain \eqref{e:post-a}. If it also happens that $\tau\leq a$, then $T\lesssim_a 1$ and \eqref{e:post-b}  follows immediately from \eqref{e:post-a}.\end{proof}

\begin{lemma}[H\"older continuity] \label{l:Holder} There exists an absolute constant $\hat\varepsilon>0$ such that if $(f,E,\mathcal{A})$ is $\varepsilon$-almost affine for some $\varepsilon<\hat\varepsilon$, then $f|_E$ is locally $\alpha$-H\"older continuous for $\alpha=\alpha(\varepsilon)<1$ such that $\alpha\uparrow 1$ as $\varepsilon\downarrow 0$. More precisely, if $\theta=1-2\log_2(1+\varepsilon)\in(0,1)$, then \begin{equation} \label{e:hexp} |f(x)-f(y)| \leq \frac{4}{\theta\log(2)} \left(\frac{|x-y|}{r_0}\right)^{(1-\varepsilon)\theta}\Anorm{x_0,r_0}r_0\quad\text{for all }x,y\in E\cap B^n(x_0,r_0/2)\end{equation} for all $x_0\in E$ and $r_0>0$.\end{lemma}

\begin{proof} Set $\hat\varepsilon=\sqrt{2}-1$. Let $\varepsilon<\hat\varepsilon$, so that $1-2\log_2(1+\varepsilon)=:\theta>0$. Suppose that $(f,E,\mathcal{A})$ is $\varepsilon$-almost affine. Let $x_0\in E$ and $r_0>0$ be given. Fix $x,y\in E\cap B^n(x_0,r_0/2)$ so that $|x-y|=r\leq r_0$. On one hand, since $(f,E,\mathcal{A})$ is $\varepsilon$-almost affine,   \begin{equation*}\begin{split} |f(x)-f(y)| &\leq |A_{x,r}(x)-A_{x,r}(y)| +|f(x)-A_{x,r}(x)|+|f(y)-A_{x,r}(y)|\\
 &\leq (1+2\varepsilon)\Anorm{x,r}r \leq 2\Anorm{x,r}r.\end{split}\end{equation*} On the other hand, since $\tau(x,r,x_0,r_0) \leq r_0/r$, by Lemma \ref{l:main-est} (\ref{e:pre-b}), \begin{align*} \Anorm{x,r} &\leq \Anorm{x_0,r_0}+\left(2\log_2\left(\frac{r_0}{r}\right)+1\right) \left(\frac{r_0}{r}\right)^{2\log_2(1+\varepsilon)}\varepsilon\Anorm{x_0,r_0}\\ &\leq \Anorm{x_0,r_0}+ \left(\frac{2}{\varepsilon\theta\log(2)}\left(\left(\frac{r_0}{r}\right)^{\varepsilon\theta}-1\right)+1\right)\left(\frac{r_0}{r}\right)^{2\log_2(1+\varepsilon)}\varepsilon\Anorm{x_0,r_0}\\
 &\leq \frac{2}{\theta\log(2)}\left(\frac{r_0}{r}\right)^{2\log_2(1+\varepsilon)+\varepsilon\theta} \Anorm{x_0,r_0}= \frac{2}{\theta\log(2)}\left(\frac{r}{r_0}\right)^{(1-\varepsilon)\theta} \Anorm{x_0,r_0}\left(\frac{r_0}{r}\right), \end{align*}
where to pass between the first and second lines we used the inequality \begin{equation*} \log_2(t)\leq \frac{t^\delta-1}{\delta \log(2)}\quad \text{for all }t\geq 1\text{ and }\delta> 0.\end{equation*} (That is, $\log(t)\leq t-1$ for all $t\geq 1$.) Combining the displayed equations immediately gives (\ref{e:hexp}). Therefore, the map $f|_E$ is locally $\alpha$-H\"older continuous, where $\alpha = (1-\varepsilon)\theta<1$ only depends $\varepsilon$. Lastly, note that $(1-\varepsilon)\theta\uparrow 1$ as $\varepsilon\downarrow 0$. \end{proof}

\begin{lemma}\label{l:inradius} Let $(f,B^n(x,r),\mathcal{A})$ be $\varepsilon$-almost affine for some $x\in\RR^n$ and $r>0$. If $\lambda_n(A'_{x,r})\leq H\lambda_1(A'_{x,r})$ and $H(t+2\varepsilon)\leq 1$, then \begin{equation} \label{e:ird}\Anorm{x,r}{r}\leq \diam f(B^n(x,r)) \leq 3\Anorm{x,r}r,\end{equation}
 \begin{equation} \label{e:irir}t\Anorm{x,r}r\leq |f(x)-f(y)|\leq2\Anorm{x,r}r \quad\text{for all } y\in\partial B^n(x,r),\end{equation} and
\begin{equation}\label{e:irt}\theta_{f(B^n(x,r))}\left(f(x),\frac{1}{3H}\diam f(B(x,r))\right)\leq 6\varepsilon H.\end{equation}
\end{lemma}

\begin{proof}Fix a parameter $0<t<1$. Suppose $(f, B^n(x,r),\mathcal{A})$ is $\varepsilon$-almost affine for some $x\in\RR^n$, $r>0$ and $\varepsilon>0$. Furthermore, suppose $\lambda_n(A_{x,r}')\leq H \lambda_1(A_{x,r}')$ for some $1\leq H<\infty$ such that $H(t+2\varepsilon)\leq 1$. We will compare $f(B^n(x,r))\subset \RR^N$ with $A_{x,r}(B^n(x,r))\subset\RR^N$.

To start, observe that \begin{equation*} \diam A_{x,r}(B^n(x,r)) = 2\Anorm{x,r}r\end{equation*} and \begin{equation}\label{e:ir1} \inf_{|y-x|=r} |A_{x,r}(y)-A_{x,r}(x)| =\lambda_1(A'_{x,r})r\geq H^{-1} \Anorm{x,r}r,\end{equation} since $A_{x,r}$ is affine. Because $(f,B^n(x,r),\mathcal{A})$ is $\varepsilon$-almost affine, it follows that \begin{equation*}\begin{split} |f(y)-f(z)| &\leq |A_{x,r}(y)-A_{x,r}(z)|+|f(y)-A_{x,r}(y)|+|f(z)-A_{x,r}(z)|\\
&\leq (2+2\varepsilon)\Anorm{x,r}r \end{split}\end{equation*} for all $y,z\in B^n(x,r)$. Hence $\diam f(B^n(x,r))\leq (2+2\varepsilon)\Anorm{x,r}r$. Similarly, choosing $y_0,z_0\in B^n(x,r)$ such that $|A_{x,r}(y_0)-A_{x,r}(z_0)|=\diam A_{x,r}(B^n(x,r))$, we see \begin{equation*}\begin{split} |f(y_0)-f(z_0)| &\geq |A_{x,r}(y_0)-A_{x,r}(z_0)|-|f(y_0)-A_{x,r}(y_0)|-|f(z_0)-A_{x,r}(z_0)|\\
&\geq (2-2\varepsilon)\Anorm{x,r}r. \end{split}\end{equation*} Hence $\diam f(B^n(x,r))\geq (2-2\varepsilon)\Anorm{x,r}r$. A parallel argument gives, for any $|y-x|=r$, \begin{equation*}\begin{split} |f(y)-f(x)| &\leq |A_{x,r}(y)-A_{x,r}(x)|+|f(y)-A_{x,r}(y)|+|f(x)-A_{x,r}(x)|\\
&\leq (1+2\varepsilon)\Anorm{x,r}r \end{split}\end{equation*} and \begin{equation*}\begin{split} |f(y)-f(x)| &\geq |A_{x,r}(y)-A_{x,r}(x)|-|f(y)-A_{x,r}(y)|-|f(x)-A_{x,r}(x)|\\ &\geq (H^{-1}-2\varepsilon)\Anorm{x,r}r.\end{split}\end{equation*} The inequalities (\ref{e:ird}) and (\ref{e:irir}) now follow, since $2\varepsilon\leq 1$ and $H^{-1}-2\varepsilon\geq t$ by our assumption  that $H(2\varepsilon+t)\leq 1$.

To continue, we estimate the local flatness $\theta_{f(B^n(x,r))}(f(x),s)$ at scale $s=H^{-1}\Anorm{x,r}r$. Let $V$ be the $n$-dimensional hyperplane containing $f(x)$ given by \begin{equation*}V=f(x)-A_{x,r}(x)+A_{x,r}(\RR^n).\end{equation*} On one hand, if $w\in f(B^n(x,r))\cap B^N(f(x),s)$, say $w=f(z)$ for some $z\in B^n(x,r)$, then \begin{equation*}\begin{split} \dist(w,V) &\leq |f(x)-A_{x,r}(x)+A_{x,r}(z)-w|\\ &\leq |f(x)-A_{x,r}(x)|+|A_{x,r}(z)-f(z)|\leq 2\varepsilon \Anorm{x,r}r= 2\varepsilon Hs.\end{split}\end{equation*} On the other hand, suppose that $v\in V\cap B^N(f(x),s)$, say $v=f(x)-A_{x,r}(x)+A_{x,r}(z)$ for some $z\in \RR^n$. Since $|A_{x,r}(z)-A_{x,r}(x)|=|f(x)-v|\leq s= H^{-1}\Anorm{x,r}r$, we know that $z\in B^n(x,r)$ by \eqref{e:ir1}. Thus \begin{equation*}\begin{split} \dist(v,f(B^n(x,r))) \leq |v-f(z)| &\leq |f(x)-A_{x,r}(x)|+|A_{x,r}(z)-f(z)| \\ &\leq 2\varepsilon \Anorm{x,r}r=2\varepsilon Hs.\end{split}\end{equation*} We conclude that $\theta_{f(B^n(x,r))}(f(x),s)\leq 2\varepsilon H$.
Finally, shrinking scales using (\ref{e:mono}) and (\ref{e:ird}) yields (\ref{e:irt}).
\end{proof}

\begin{definition}\label{d:stable} Let $E\subset\RR^n$ be bounded. A family $\mathcal{A}$ of affine maps over $E$ is \emph{stable on large scales} if there exists $x_*\in E$ such that $A_{x,r}=A_{x_*,\diam E}$ for all $x\in E$ and for all $r>\diam E$.\end{definition}

\begin{lemma} \label{l:stable} Let $E\subset X\subset \RR^n$ with $E$ bounded. If $f:X\rightarrow\RR^n$ is $\varepsilon$-almost affine over $E$ for some $\varepsilon>0$, then $(f,E,\mathcal{A})$ is $\varepsilon$-almost affine for some $\mathcal{A}$ that is stable at large scales. In fact, given any $\mathcal{B}$ such that $(f,E,\mathcal{B})$ is $\varepsilon$-almost affine and any $x_*\in E$, $(f,E,\mathcal{A})$ is $\varepsilon$-almost affine for the family $\mathcal{A}$ defined by \begin{equation} \label{e:stable1} A_{x,r}=\left\{\begin{array}{cl} B_{x,r} & \text{if\ \ } 0<r\leq \diam E,\\ B_{x_*, \diam E} & \text{if\ \  } r>\diam E,\end{array}\right.\quad\text{for all }x\in E\text{ and } r>0.\end{equation}
\end{lemma}

\begin{proof} Let $E\subset\RR^n$ be bounded, suppose that $(f,E,\mathcal{B})$ is $\varepsilon$-almost affine, and let $x_*\in E$ be given. Define $\mathcal{A}$ by (\ref{e:stable1}).  Then $\mathcal{A}$ is stable on large scales. We will show that $\mathcal{A}$ is $\varepsilon$-compatible and $(f,E,\mathcal{A})$ is $\varepsilon$-almost affine.

To show that $\mathcal{A}$ is $\varepsilon$-compatible, suppose that $x,y\in E$ and $r,s>0$ satisfy  $|x-y|\leq \max\{r,s\}$ and $1/2\leq r/s\leq 2$. We start with two easy cases. On one hand, if $r,s\leq \diam E$, then \begin{equation}\label{e:comp} \|A'_{x,r}-A'_{y,s}\|\leq\varepsilon \min\{\|A'_{x,r}\|,\|A'_{y,s}\|\},\end{equation} since $\mathcal{B}$ is $\varepsilon$-compatible and $A_{x,r}=B_{x,r}$ and $A_{y,s}=B_{y,s}$. On the other hand, if $r,s>\diam E$, then (\ref{e:comp}) holds since $A_{x,r}=B_{x_*,\diam E}=A_{y,r}$. Next we look at the case of mixed scales. Assume without loss of generality that $r>\diam E$ and $s\leq \diam E$ so that $A_{x,r}=B_{x_*,\diam E}$ and $A_{y,s}=B_{y,s}$. Note that $\diam E\geq s\geq \frac12 r>\frac12 \diam E$. Hence, $\frac12\leq (\diam E)/s\leq 2$ and $|x_*-y|\leq \diam E=\max(\diam E,s)$. Thus, in this case (\ref{e:comp}) holds, since $\mathcal{B}$ is $\varepsilon$-compatible. Therefore, $\mathcal{A}$ is $\varepsilon$-compatible.

To check that $(f,E,\mathcal{A})$ is $\varepsilon$-almost affine, let $x\in E$ and $r>0$. If $r\leq \diam E$, then $A_{x,r}=B_{x,r}$. Hence \begin{equation}\label{e:close} \sup_{z\in E\cap B(x,r)}\|f(z)-A_{x,r}(z)\| \leq\varepsilon \|A_{x,r}'\|r,\end{equation} since $(f,E,\mathcal{B})$ is $\varepsilon$-almost affine. Similarly, if $r>\diam E$, then $A_{x,r}=B_{x_*,\diam E}$ and (\ref{e:close}) holds, since $E\cap B(x,r)=E=E\cap B(x_*,\diam E)$ and $(f,E,\mathcal{B})$ is $\varepsilon$-almost affine. Therefore, $(f,E,\mathcal{A})$ is $\varepsilon$-almost affine. \end{proof}

\begin{definition} \label{d:adapt} Let $f:\RR^n\rightarrow\RR^N$ and let $E\subset\RR^n$ be bounded. A family $\mathcal{A}$ of affine maps over $E$ is \emph{adapted to $f$ on small scales} if, for all $x\in E$ and $r\leq \diam E$, \begin{equation}A_{x,r}(x+re_i)=f(x+re_i)\quad\text{for all }i=0,1,\dots,n,\end{equation} where $e_0=0$ and $e_1,\dots,e_n$ is the standard basis for $\RR^n$.\end{definition}

\begin{lemma} \label{l:adapt} For all $n\geq 1$, there exists $P=P(n)>1$ such that if $P\varepsilon\leq 1$ and $f:\RR^n\rightarrow\RR^N$ is $\varepsilon$-almost affine over $B^n(x_0,3r_0)$ for some $x_0\in\RR^n$ and $r_0>0$, then $(f, B^n(x_0,r_0),\mathcal{A})$ is $P\varepsilon$-almost affine for some $\mathcal{A}$ that is adapted to $f$ at small scales. In fact, given any $\mathcal{B}$ such that $(f,B^n(x_0,3r_0),\mathcal{B})$ is $\varepsilon$-almost affine, there exists a family $\mathcal{A}$ of affine maps over $B^n(x_0,r_0)$ that is adapted to $f$ at small scales such that $(f,B^n(x_0,r_0),\mathcal{A})$ is $P\varepsilon$-almost affine and such that $A_{x,r}=B_{x,r}$ for all $x\in B^n(x_0,r_0)$ and for all $r>2r_0$.\end{lemma}

We first prove an auxiliary lemma. For all bounded $V\subset\RR^n$ with positive diameter, define \begin{equation} \label{e:Psi} \Psi(V):=(\diam V)^n/\Leb^n(\co V)\in(0,\infty],\end{equation} where $\co V$ denotes the closed convex hull of $V$ and by convention $\Psi(V)=\infty$ whenever $\Leb^n(\co V)=0$. We remark that the \emph{isodiametric inequality} asserts that $\Psi(V)\geq \Psi(B^n(0,1))$; see e.g.~ \cite[Chapter 2]{EG}.

\begin{lemma}\label{l:A-B} Suppose $V=\{v_0,\dots, v_n\}\subset \RR^n$. If $A,B:\RR^n\rightarrow\RR^N$ are affine maps such that $|A(v)-B(v)|\leq \varepsilon \diam V$ for all $v\in V$, then
\begin{equation}\label{e:l:A-B}|A(z)-B(z)| \leq  \varepsilon \left(\diam V + \frac{4 n^{(n+1)/2}}{n!}\Psi(V)\,\dist(z,V)\right)\quad\text{for all }z\in\RR^n.\end{equation}
\end{lemma}
\begin{proof} If $\Psi(V)=\infty$, then there is nothing to prove. Thus, assume that $\Psi(V)<\infty$, which ensures that $v_0,\dots v_n$ are affinely independent.
 Let $z\in\RR^n$. By relabeling the elements of $V$, we may assume without loss of generality that $|z-v_0|=\dist(z,V)$ and $|v_1-v_0|=\max_i|v_i-v_0|$. Then $\frac12\diam V\leq |v_1-v_0|\leq\diam V$. Furthermore, after a harmless translation, we may assume without loss of generality that $v_0=0$. Let $\{e_1,\dots,e_n\}$ be the standard basis for $\RR^n$ and let $T:\RR^n\rightarrow\RR^n$ be the invertible linear transformation such that  $T(e_{i})=v_{i}$ for all $i=1,\dots, n$. Then, letting $E=\{0,e_1,\dots,e_n\}$,
\begin{equation*} \Leb^n(\co V) =\Leb^n(T(\co E)) =\Leb^n(\co E)|\det T| =\frac{|\det T|}{n!} .\end{equation*}
Hence $\Psi(V)= n!(\diam V)^n/|\det T|$. Next, note that
\begin{equation*}
\frac12\diam V\leq |T(e_1)|\leq \|T\| \leq (|v_1|^2+\dots+|v_n|^2)^{1/2}\leq \sqrt{n} \max_i|v_i|=\sqrt{n}\diam V, \end{equation*}
where the third inequality follows from the Cauchy-Schwarz inequality. Thus, we see that
\begin{equation*} \Psi(V)
 \geq \frac{n!\|T\|^{n}}{n^{n/2} |\det T|}
 \geq \frac{n!}{n^{n/2}}\|T\|\|T^{-1}\|
 \geq \frac{n!}{n^{n/2}}\frac{\diam V}{2}\|T^{-1}\|.
 \end{equation*}
 Let $A'= A-A(0)$ and $B'=B-B(0)$ denote the linear parts of $A$ and $B$, respectively. By the hypothesis, $|A'(v)-B'(v)|\leq |A(v)-B(v)|+|A(0)-B(0)|\leq 2\varepsilon\diam V$ for all $v\in V$. Expanding $z=a_1v_1+\dots+a_nv_n$, it follows that
 \begin{align*}
|A(z)-B(z)|
&\leq |A(0)-B(0)| + \sum_{i=1}^n |a_i|\,|A'(v_i)- B'(v_i)|\\
&\leq \left(1+2(|a_1|+\dots+|a_n|)\right)\, \varepsilon\diam V.
\end{align*}
To continue, observe that
\begin{equation*}\sum_{i=1}^{n}|a_i| \leq\sqrt{n}\left|\sum_{i=1}^{n}a_ie_i\right|
=\sqrt{n}\,|T^{-1}(z)|\\
\leq \sqrt{n}\,\|T^{-1}\| |z|
=\sqrt{n}\,\|T^{-1}\|\dist(z,V).\end{equation*}
Combining the previous three displayed equations yields \eqref{e:l:A-B}.
\end{proof}

\begin{proof}[Proof of Lemma \ref{l:adapt}] Fix $P>1$ to be specified later.  Suppose that $f$ is $\varepsilon$-almost affine over $B^n(x_0,3r_0)$ for some $x_0\in\RR^n$ and $r_0>0$, for some $\varepsilon>0$ such that $P\varepsilon\leq 1$. Choose any $\mathcal{B}$ such that $(f,B^n(x_0,3r_0),\mathcal{B})$ is $\varepsilon$-almost affine. Define a family $\mathcal{A}$ of affine maps over $B^n(x_0,r_0)$ as follows.  For all $x\in B^n(x_0,r_0)$ and $0<r\leq 2r_0$, define $A_{x,r}$ to be the unique affine map such that $A_{x,r}(x)=f(x)$ and $A_{x,r}(x+re_i)=f(x+re_i)$ for all $1\leq i\leq n$. And, for all $x\in B^n(x_0,r_0)$ and $r>2r_0$, set $A_{x,r}:=B_{x,r}$. Then $\mathcal{A}$ is adapted to $f$ at small scales. It remains to show that $(f,B^n(x_0,r_0),\mathcal{A})$ is $P\varepsilon$-almost affine.

Fix $x\in B(x_0,r_0)$ and $r\leq 2r_0$. Note that $B(x,r)\subset B(x_0,3r_0)$. Let $Y=\{y_0,\dots,y_n\}\subset\RR^n$, where $y_0=x$ and $y_i=x+re_i$ for all $i=1,\dots,n$. Since $\diam Y= r\sqrt{2}$ and $\Leb^n(\co Y)=r^n/n!$, we have $\Psi(Y)=2^{n/2}n!\,$.  Furthermore, for all $y\in Y$, \begin{equation*} |A_{x,r}(y)-B_{x,r}(y)| =|f(y)-B_{x,r}(y)|\leq \varepsilon \|B'_{x,r}\| r = \frac{\varepsilon}{\sqrt{2}} \|B'_{x,r}\| \diam Y,\end{equation*}
because $\mathcal{A}$ is adapted to $f$ at small scales, $(f,B^n(x_0,3r_0),\mathcal{B})$ is $\varepsilon$-almost affine, and $Y\subset B^n(x,r)\cap B^n(x_0,3r_0)$. Thus, for all $z\in B^n(x,r)$,
\begin{equation}\begin{split}
|{A_{x,r}}(z)-B_{x,r}(z)|
&\leq  \frac{\varepsilon}{\sqrt{2}}\|B_{x,r}'\|\diam Y +   \frac{4n^{(n+1)/2}}{n!}2^{n/2}n!\frac{\varepsilon}{\sqrt{2}}\|B'_{x,r}\|\dist(z,Y) \\
 & \leq \left(1+2(2n)^{(n+1)/2}\right)\varepsilon\|B_{x,r}'\|r \label{e:adapt1}
\end{split}\end{equation} by Lemma \ref{l:A-B}.
 Hence, writing $\theta:=1+2(2n)^{(n+1)/2}\geq 5$, we see that
\begin{align*}
(1-2\theta\varepsilon)\|B_{x,r}'\|r
& =\sup_{|z-x|=r}|B_{x,r}(z)-B_{x,r}(x)|-2\theta\varepsilon\|B_{x,r}'\|r  \notag \\
& \leq \sup_{|z-x|= r}| A_{x,r}(z)- A_{x,r}(x)|
 = \| A_{x,r}'\|r.
\end{align*} We now specify $P\geq 4\theta$ so that $\varepsilon\leq 1/P\leq 1/4\theta$ and \begin{equation}\label{e:adapt2} \|B_{x,r}'\|\leq 2\Anorm{x,r}.\end{equation} It follows that
\begin{equation}\begin{split}
 |f(z)- A_{x,r}(z)| &\leq |f(z)-B_{x,r}(z)|+| A_{x,r}(z)-B_{x,r}(z)| \\&\leq (1+\theta)\varepsilon\|B_{x,r}'\|r \leq (2+2\theta)\varepsilon\|A_{x,r}'\|r \leq 3\theta \varepsilon\Anorm{x,r}r \label{e:adapt3} \end{split}\end{equation} for all $z\in B^n(x,r)$, for all $x\in B^n(x_0,r_0)$ and for all $0<r\leq 2r_0$.

Next, observe that \begin{equation} \label{e:adapt4} |f(z)-A_{x,r}(z)|\leq \varepsilon \Anorm{x,r}{r}\end{equation} for all $z\in B^n(x,r)\cap B^n(x_0,r_0)$, for all $x\in B^n(x_0,r_0)$, and for all $r>2r_0$, because $A_{x,r}=B_{x,r}$ and $(f,B^n(x_0,3r_0),\mathcal{B})$ is $\varepsilon$-almost affine.

We now verify that $\mathcal{A}$ is $C(n)\varepsilon$-compatible. Fix $x,y\in B^n(x_0,r_0)$ and $r,s>0$ such that $s/2\leq r\leq s$ and $|x-y|\leq s$. We proceed by cases.

\setcounter{case}{0}

\begin{case} Suppose that $2r_0<r$ and $2r_0<s$. Then $\Adiff{x,r}{y,s}\leq \varepsilon\Amin{x,r}{y,s}$, since $A_{x,r}=B_{x,r}$, $A_{y,s}=B_{y,s}$ and $\mathcal{B}$ is $\varepsilon$-compatible.\end{case}

\begin{case} Suppose that $r\leq 2r_0<s$. On one hand, by (\ref{e:adapt1}) (twice), \begin{equation}\begin{split} \|A'_{x,r}-B'_{x,r}\| &= \sup_{|z-x|=r}\frac{1}{r}|A_{x,r}(z)-B_{x,r}(z) - (A_{x,r}(x)-B_{x,r}(x))| \\
 &\leq 2\theta\varepsilon \|B'_{x,r}\|.\label{e:adapt5} \end{split}\end{equation} On the other hand, since $\mathcal{B}$ is $\varepsilon$-compatible, \begin{equation}\label{e:adapt6} \|B'_{x,r}-B'_{y,s}\| \leq \varepsilon \min\{ \|B_{x,r}'\|, \|B_{y,s}'\| \}.\end{equation} It follows that $\|B_{x,r}'\| \leq 2\|B_{y,s}'\|$ and \begin{equation*} \|A'_{x,r}-B'_{y,s}\| \leq \|A'_{x,r}-B'_{x,r}\| + \|B'_{x,r}-B'_{y,s}\| \leq (4\theta+1)\varepsilon \|B_{y,s}'\|\leq 5\theta\varepsilon\|B_{y,s}'\|.\end{equation*} Hence $\|B_{y,s}'\| \leq 5\theta \varepsilon\|B_{y,s}'\|+\Anorm{x,r}$ by the triangle inequality. We now require that $P\geq 10\theta$ so that $\varepsilon\leq 1/P\leq 1/10\theta$ and $\|B_{y,s}'\|\leq 2\Anorm{x,r}$. Therefore, since $A_{y,s}=B_{y,s}$,  \begin{equation*} \Adiff{x,r}{y,s} = \|A'_{x,r}-B'_{y,s}\| \leq 10\varepsilon \min\{\Anorm{x,r},\|B'_{y,s}\|\}=\Amin{x,r}{y,s}.\end{equation*}
 \end{case}

\begin{case} Suppose that $r\leq 2r_0$ and $s\leq 2r_0$. Then \begin{align*} \Adiff{x,r}{y,s} &\leq \|A'_{x,r}-B'_{x,r}\| + \|B'_{x,r}-B'_{y,s}\| + \|B'_{y,s}-A'_{y,s}\|\\
&\leq 2\theta\varepsilon\|B'_{x,r}\| + \varepsilon \min\{\|B'_{x,r}\|, \|B'_{y,s}\|\} + 2\theta\varepsilon \|B'_{y,s}\|\end{align*} by (\ref{e:adapt5}) (twice) and (\ref{e:adapt6}). Hence, by (\ref{e:adapt2}), \begin{equation*} \Adiff{x,r}{y,s}\leq (8\theta+2)\varepsilon\Amax{x,r}{y,s}\leq 9\theta\varepsilon\Amax{x,r}{y,s}.\end{equation*} Using the triangle inequality (twice), it follows that \begin{equation*} \Anorm{x,r} \leq 9\theta\varepsilon \Anorm{x,r}+\Anorm{y,s}\quad\text{and}\quad \Anorm{y,s}\leq 9\theta\varepsilon\Anorm{y,s}+\Anorm{x,r}.\end{equation*} We now require that $P=18\theta$ so that $\varepsilon\leq 1/P\leq 1/18\theta$. Therefore, $\Amax{x,r}{y,s}\leq 2\Amin{x,r}{y,s}$ and \begin{equation*} \Adiff{x,r}{y,s} \leq 18\theta\varepsilon \Amin{x,r}{y,s}.\end{equation*}
\end{case}

We have verified that $\mathcal{A}$ is $P\varepsilon$-compatible, where $P=18\theta=18\left(1+2(2n)^{(n+1)/2}\right)$. Therefore, $(f,B^n(x_0,r_0),\mathcal{A})$ is $P\varepsilon$-almost affine by (\ref{e:adapt3}) and (\ref{e:adapt4}).
\end{proof}


\section{Almost affine quasisymmetric maps with small constants}
\label{s:affine}

To open this section, we supply a proof of Lemma \ref{l:rflat}, which for convenience we now restate.

\begin{lemma}\label{l:rflat2} For all $\delta>0$ there exists $\delta_*=\delta_*(\delta)$ with the following property. Suppose that $f:\RR^n\rightarrow\RR^N$ is quasisymmetric and $H_f(\RR^n)\leq H$. If $f$ is $\delta_*$-almost affine over $B^n(x_0,2r_0)$ and $\tH_f(B^n(x_0,2r_0))\leq \delta_*$ for some $x_0\in\RR^n$ and $r_0>0$, then \begin{equation}\label{e:rf}\theta_{f(\RR^n)}(f(x),r) \leq H\delta\quad\text{for all }x\in B^n(x_0,r_0)\text{ and } 0<r\leq \frac{1}{54H}\diam f(B^n(x_0,r_0)).\end{equation} Thus, if $f$ is $\delta_*$-almost affine over $\RR^n$ and $\tH_f(\RR^n)\leq \delta_*$, then $f(\RR^n)$ is $(H\delta,\infty)$-Reifenberg flat, i.e.~ $\theta_{f(\RR^n)}(f(x),r)\leq H\delta$ for all $x\in\RR^n$ and $r>0$.
\end{lemma}

\begin{proof}[Proof of Lemma \ref{l:rflat} / Lemma \ref{l:rflat2}] Given $0<\delta<1$, put $\varepsilon=\delta/18$ and $t=\frac12-\frac19\delta$. Observe that this choice of parameters satisfies $2(t+2\varepsilon)=1$. Let $f:\RR^n\rightarrow\RR^N$ be quasisymmetric with $H_f(\RR^n)\leq H$, and let $x_0\in\RR^n$ and $r_0>0$. Assume that  $(f,B^n(x_0,2r_0),\mathcal{A})$ is $\delta_*$-almost affine and $\tH_f(B^n(x_0,2r_0))\leq \delta_*$ for some $\delta_*\leq \varepsilon$ to be specified below. Let $x\in B^n(x_0,r_0)$ and $0<r\leq r_0$. Then, recalling (\ref{e:lvar}), we find that \begin{equation*}\begin{split} \lambda_n(A'_{x,r})r&\leq 2\delta_* \Anorm{x,r}r+\sup_{|y-x|=r}|f(y)-f(x)|\\
&\leq 2\delta_*\Anorm{x,r}r + (1+\delta_*)\inf_{|y-x|=r}|f(y)-f(x)| \\ &\leq 2\delta_*(2+\delta_*)\Anorm{x,r}r +(1+\delta_*)\lambda_1(A'_{x,r})r.\end{split}\end{equation*}
In particular, $\lambda_n(A_{x,r}')\leq 2\lambda_1(A_{x,r}')$ provided that we choose $\delta_*$ to be smaller than some absolute constant. By Lemma \ref{l:inradius}, it follows that \begin{equation}\label{e:rf0} \Anorm{x,r}r \leq \diam f(B^n(x,r))\leq 3\Anorm{x,r}r,\end{equation} \begin{equation}\label{e:rf1} \inf_{|y-x|=r} |f(y)-f(x)| \geq \left(\frac{1}{6}-\frac{\delta}{27}\right)\diam f(B^n(x,r))\geq \frac{1}{9}\diam f(B^n(x,r)),\end{equation} and \begin{equation}\label{e:rf2} \theta_{f(B^n(x,r))}\left(f(x),\frac{1}{6} \diam f(B^n(x,r))\right)\leq \frac{2}{3}\delta.\end{equation} On one hand, using (\ref{e:mono}) to shrink scales in (\ref{e:rf2}), we obtain \begin{equation}\label{e:rf4} \theta_{f(B^n(x,r))}\left(f(x),\frac{1}{9H}\diam f(B^n(x,r))\right) \leq H\delta.\end{equation}
On the other hand, (\ref{e:rf1}) and the hypothesis $H_f(\RR^n)\leq H$ yield \begin{equation*} \inf_{|z-x|\geq r}|f(z)-f(x)| \geq \frac{1}{H}\sup_{|y-x|=r}|f(y)-f(x)| \geq \frac{1}{9H}\diam f(B^n(x,r)),\end{equation*} which implies that \begin{equation}\label{e:rf3} f(\RR^n)\cap B^N\left(f(x),\frac{1}{9H}\diam f(B^n(x,r))\right)\subset f(B^n(x,r)).\end{equation} Together (\ref{e:rf4}) and (\ref{e:rf3}) yield \begin{equation*} \theta_{f(\RR^n)}\left(f(x),\frac{1}{9H}\diam f(B^n(x,r))\right) \leq H\delta\quad\text{for all }x\in B^n(x_0,r_0)\text{ and }0<r\leq r_0.\end{equation*} Finally, since $\mathcal{A}$ is $1$-compatible and $|x-x_0|\leq r_0$, we have \begin{equation*} \frac{1}{9H}\diam f(B^n(x,r_0)) \geq \frac{1}{9H}\Anorm{x,r_0}r_0 \geq \frac{1}{18H} \Anorm{x_0,r_0}r_0 \geq \frac{1}{54H} \diam f(B^n(x_0,r_0))\end{equation*} by (\ref{e:rf0}). Combining the previous two displayed equations yields (\ref{e:rf}). \end{proof}

Next up, we aim to prove Theorem \ref{t:qsaa}, but first we prove an intermediate statement.

\begin{lemma} \label{l:Sxr}
 For all $N\geq 2$ and for all $\varepsilon>0$, there exists $\delta>0$ such that if $f:B^N(x,r)\rightarrow \RR^N$ is quasisymmetric and $\tH_f(B^N(x,r))\leq \delta$, then there is a similarity $S_{x,r}:\RR^N\rightarrow\RR^N$ such that $|f(y)-S_{x,r}(y)|\leq \varepsilon\|S'_{x,r}\|r$ for all $y\in B^N(x,r)$.
\end{lemma}

\begin{proof} Recall that similarities are the compositions of translations, rotations, reflections, and dilations in $\RR^N$.
Let $N\geq 2$ be given. Suppose for contradiction that there exists $\varepsilon>0$ and a sequence of quasisymmetric maps $f^i:B^N(x^i,r^i)\rightarrow\RR^N$ such that $x^i\in\RR^N$,  $r^i>0$, and \begin{equation*}\tH_{f^i}(B^N(x^i,r^i))\leq 1/i,\end{equation*} but for every similarity $S:\RR^N\rightarrow\RR^N$ there exists $y_S^i\in B^N(x^i,r^i)$ such that \begin{equation*} |f^i(y_S^i)-S(y_S^i)|>\varepsilon \|S'\| r.\end{equation*} For each $i\geq 1$, let $\phi^i:\RR^N\rightarrow\RR^N$ be any similarity such that $\phi^i(B^N(0,1))=B^N(x^i,r^i)$, let $\psi^i:\RR^N\rightarrow\RR^N$ be any similarity such that $\psi^i(f(x^i))=0$ and $\psi^i(f(\phi^i(e_1)))=e_1$, and set \begin{equation*}g^i=\psi^i\circ f^i\circ \phi^i:B^N(0,1)\rightarrow \RR^N.\end{equation*} Then each $g^i$ is a quasisymmetric map such that $g(0)=0$, $g(e_1)=e_1$, and $\tH_{g^i}(B^N(0,1))\leq 1/i$,  but for every similarity $S:\RR^N\rightarrow\RR^N$ there exists $z_S^i\in B^N(0,1)$ such that \begin{equation*} |g^i(z_S^i)-S(z_S^i)|>\varepsilon \|S'\|.\end{equation*} The family $\{g_i:i\geq 1\}$ is sequentially compact by Theorem \ref{t:wqs-cpt}. Thus, we can find a weakly 1-quasisymmetric map $g:B^N(0,1)\rightarrow\RR^N$ and a sequence $i_k\rightarrow\infty$ such that $g^{i_k}\rightarrow g$ uniformly as $k\rightarrow\infty$. By Corollary \ref{c:1qs}, there exists an similarity $T:\RR^N\rightarrow\RR^N$ such that $g=T|_{B^N(0,1)}$. Passing to a further subsequence, we can assume that $g^i\rightarrow g$ uniformly and $z_T^i\rightarrow z_T$ for some $z_T\in B^N(0,1)$. This leads to a contradiction:  \begin{equation*} 0=|g(z_T)-T(z_T)|=\liminf_{i\rightarrow\infty} |g^i(z^i_T)-T(z^i_T)|\geq \varepsilon \|T'\|>0. \end{equation*} Therefore, for all $\varepsilon>0$, there is $\delta>0$ such that if $f:B^N(x,r)\rightarrow \RR^N$ is quasisymmetric and $\tH_f(B^N(x,r))\leq \delta$, then there is a similarity $S:\RR^N\rightarrow\RR^N$ such that $|f(y)-S(y)|\leq \varepsilon\|S'\|r$ for all $y\in B^N(x,r)$.
\end{proof}

We now give a proof of Theorem \ref{t:qsaa}, which for convenience we now restate.

\begin{theorem}\label{t:qsaa2} Suppose $N\geq 2$. For all $\tau>0$, there exists $\tau_*=\tau_*(\tau,N)>0$ such that if $B^N(x,3r)\subset Y\subset \RR^N$ for some $x\in\RR^n$ and $r>0$, $f:Y\rightarrow\RR^N$ is quasisymmetric and $\tH_f(B^N(x,3r))\leq \tau_*$, then $f|_{Y\cap \RR^n}$ is $\tau$-almost affine over $B^n(x,r)$.
\end{theorem}

\begin{proof}[Proof of Theorem \ref{t:qsaa} / Theorem \ref{t:qsaa2}] It suffices to establish the theorem with $Y=B^N(0,3)$, $x=0$ and $r=3$. Let $N\geq 2$ and $\tau>0$ be given, and fix $0<\tau_*\leq 1$ to be specified below. Without loss of generality, we shall assume that $\tau\leq 1$. Assume that $f:B^N(0,3)\rightarrow\RR^N$ is quasisymmetric and $\tH_f(B^N(0,3))\leq \tau_*$. Note that $B^N(x,r)\subset B^N(0,3)$ for all $x\in B^n(0,1)$ and $r\leq 2=\diam B^n(0,1)$. Let $\delta>0$ be the constant from Lemma \ref{l:Sxr} corresponding to $\varepsilon = \min\{1/12,\tau/128\}$. Assume $\tau_*\leq \delta$. By Lemma \ref{l:Sxr}, for all $x\in B^n(0,1)$ and $0<r\leq 2$ we can find similarities $S_{x,r}:\RR^N\rightarrow\RR^N$ such that \begin{equation*}\sup_{y\in B^N(x,r)}|f(y)-S_{x,r}(y)|\leq \varepsilon \|S'_{x,r}\| r.\end{equation*} For each $x\in B^n(0,1)$ and $0<r\leq 2$, let $A_{x,r}:\RR^n\rightarrow\RR^N$ to be the restriction of $S_{x,r}$ to $\RR^n$. Then \begin{equation*}\begin{split} \|S'_{x,r}\|r &\leq 2\varepsilon \|S'_{x,r}\|r + \sup_{|y-x|=r}|f(y)-f(x)| \\
&\leq 2\varepsilon\|S'_{x,r}\|r+(1+\tau_*)\inf_{|y-x|=r}|f(y)-f(x)|\\
&\leq 2\varepsilon(2+\tau_*)\|S'_{x,r}\|+(1+\tau_*)\lambda_1(S_{x,r}').
\end{split}\end{equation*} Since $\tau_*\leq 1$ and $\varepsilon\leq 1/12$, we have $2\varepsilon(2+\tau_*)\leq 1/2$ and \begin{equation*} \|S'_{x,r}\| \leq 2(1+\tau_*) \lambda_1(S'_{x,r})\leq 4\lambda_1(S'_{x,r})\leq 4\|A'_{x,r}\|.\end{equation*} Hence, for all $x\in B^n(0,1)$ and $0<r\leq 2$, \begin{equation} \label{e:qsaa1} \sup_{y\in B^n(x,r)}|f(y)-A_{x,r}(y)| \leq 4\varepsilon \Anorm{x,r}\leq  \frac{\tau}{32}\Anorm{x,r}r.\end{equation} For all $x\in B^n(0,1)$ and $r>2=\diam B^n(0,1)$, assign $A_{x,r}=A_{0,2}$. Then \begin{equation*} \mathcal{A}=\{A_{x,r}:x\in B^n(0,1),r>0\}\end{equation*} is a family of affine maps over $B^n(0,1)$ that is stable at large scales. Using (\ref{e:qsaa1}), it readily follows (cf. the proof of Lemma \ref{l:stable} above) that for all $x\in B^n(0,1)$ and $r>0$, \begin{equation*} |f(y)-A_{x,r}(y)| \leq \frac{\tau}{32} \|A'_{x,r}\| r \quad\text{for all }y\in B^n(x,r)\cap B^n(0,1).\end{equation*} Thus, to show that $(f,B^n(0,1),\mathcal{A})$ is $\tau$-almost affine, it is enough to check $\mathcal{A}$ is $\tau$-compatible.

To that end, suppose that $x,z\in B^n(0,1)$ and $0<s\leq r\leq 2s$. We must show that \begin{equation}\label{e:qsaa2}\Adiff{x,r}{z,s}\leq \tau \Amin{x,r}{z,s}.\end{equation}
In view of (the proof of) Lemma \ref{l:stable} above, we may assume without loss of generality that $r\leq 2=\diam B^n(0,1)$.
Let $w=\frac{1}{4}x+\frac{3}{4}z$ so that $B^n(w,r/4)\subset B^n(x,r)\cap B^n(z,s)$. By (\ref{e:qsaa1}) (four times), \begin{align*} \Adiff{x,r}{z,s}\frac{r}{4} &= \sup_{|v-w|=r/4}\left| \left(A_{x,r}(v)-A_{z,s}(v)\right) - \left(A_{x,r}(w)-A_{z,s}(w)\right)\right|  \\ &\leq 2\,\frac{\tau}{32}\Anorm{x,r}r + 2\,\frac{\tau}{32} \Anorm{z,s}s \leq \frac{\tau}{8} \Amax{x,r}{z,s}r.\end{align*} That is, $\Adiff{x,r}{z,s}\leq (\tau/2)\Amax{x,r}{z,s}$.  Thus, since $\tau/2\leq 1/2$, \begin{equation*} \Amax{x,r}{z,s} \leq 2\Amin{x,r}{z,s}\end{equation*} and  (\ref{e:qsaa2}) holds. Therefore, the family $\mathcal{A}$ is $\tau$-compatible and, by the discussion above, the map $f$ is $\tau$-almost affine over $B^n(0,1)$.\end{proof}


\section{Extensions of almost affine maps I}
\label{s:extend}

The goal of this section is to prove the following extension theorems for almost affine maps with small constant; cf.~ bi-Lipschitz extensions of ``Reifenberg flat functions'' constructed by the first named author and Raanan Schul \cite[Theorem III]{hardsard}. Throughout this section, we freeze dimensions $1\leq n\leq N$. See \S\ref{ss:almost} above to recall the definition of an almost affine map.

\begin{theorem} \label{t:extend1} There exist constants $\varepsilon_{0}=\varepsilon_{0}(n)>0$ and $C_0=C_0(n)>1$ with the following property. If $1\leq n\leq N$, $E\subsetneq\RR^{n}$ is closed, and $(f,E,\mathcal{A})$ is $\varepsilon$-almost affine for some $\varepsilon\leq \varepsilon_0$, then the map $f:E\rightarrow\RR^N$ can be extended to a $C_0\varepsilon$-almost affine map $F:\RR^{n}\rightarrow \RR^{N}$ such that $(F,\RR^n,\mathcal{A}^+)$ is $C_0\varepsilon$-almost affine for some $C_0\varepsilon$-compatible family $\mathcal{A}^+$ of affine maps over $\RR^n$ extending $\mathcal{A}$, i.e. $F|_E=f$ and $A^+_{x,r}=A_{x,r}$ for all $x\in E$ and $r>0$.
\end{theorem}

\begin{theorem}\label{t:extend2} For all $1\leq H<\infty$ and $1<\rho\leq 2$, there exist $\varepsilon_1=\varepsilon_1(n,H,\rho)>0$ and $C_1=C_1(n)>0$ with the following property. If $1\leq n\leq N$, $E\subsetneq \RR^n$ is closed, $f:E\rightarrow\RR^N$ is nonconstant, and $(f,E,\mathcal{A})$ is $\varepsilon$-almost affine for some $\varepsilon\leq \varepsilon_1$ and some $\mathcal{A}$ such that \begin{equation}\lambda_n(A_{x,r}')\leq H\lambda_1(A_{x,r}')\quad\text{for all }x\in E\text{ and }r>0,\end{equation} then $f$ can be extended to a $C_1\varepsilon$-almost affine map $F:\RR^n\rightarrow\RR^N$ such that $(F,\RR^n,\mathcal{A}^+)$ is $C_1\varepsilon$-almost affine for some $C_1\varepsilon$-compatible family $\mathcal{A}^+$ of affine maps over $\RR^n$ extending $\mathcal{A}$, i.e. $F|_E=f$ and $A_{x,r}^+=A_{x,r}$ for all $x\in E$ and $r>0$. Moreover,  the extension $F$ is weakly $\rho H$-quasisymmetric and the extension $\mathcal{A}^+=\{A_{x,r}:x\in\RR^n,r>0\}\supset \mathcal{A}$ satisfies \begin{equation}\lambda_n(A_{x,r}')\leq \frac{\rho+1}{2}H\lambda_1(A_{x,r}')\quad\text{for all }x\in \RR^n\text{ and }r>0.\end{equation}\end{theorem}

We split the proofs of Theorems \ref{t:extend1} and \ref{t:extend2} into several steps. First, to each map $f:E\rightarrow\RR^N$ defined on a closed set $E\subsetneq\RR^n$ and each family $\mathcal{A}$ of affine maps over $E$, we use a Whitney decomposition of $\RR^n\setminus E$ together with the maps in $\mathcal{A}$ to extend $f$ to a map $F:\RR^n\rightarrow\RR^N$ and extend $\mathcal{A}$ to a family $\mathcal{A}^+$ of affine maps over $\RR^n$  (see Definition \ref{d:extend}).
Second, we make a series of estimates on $F$ and $\mathcal{A}^+$ under the assumption that $(f,E,\mathcal{A})$ is $\varepsilon$-almost affine and $\varepsilon$ is small (see Hypothesis \ref{h:setup} and Lemmas \ref{l:AQest} -- \ref{l:Omega}). Third, we combine these estimates and prove Theorem \ref{t:extend1}. Finally, at the end of the section, we derive Theorem \ref{t:extend2} form Theorem \ref{t:extend1}.

\begin{definition}[Extensions of $f$ and $\mathcal{A}$] \label{d:extend}
Let $1\leq n\leq N$, $E\subsetneq\RR^n$ closed, $f:E\rightarrow\RR^N$ and $\mathcal{A}$ a family of affine maps over $E$ be given. For all $x\in \RR^{n}$, pick $x'\in E$ such that
\begin{equation*} d(x):=\dist(x,E)=|x-x'|.\end{equation*}
(1) (Whitney cubes)
Let $\mathcal{W}$ be a Whitney decomposition of $\RR^n\setminus E$, constructed by taking $\mathcal{W}$ to be the collection of all maximal almost disjoint closed dyadic cubes $Q\subset \RR^{n}$ such that $3Q\cap E=\emptyset$, where $\lambda Q$ denotes the concentric cube about $Q$ that is obtained by dilating $Q$ by a factor of $\lambda>0$.  The collection $\mathcal{W}$ of cubes satisfies the following properties:
 \begin{enumerate}
\item[(a)] $\bigcup_{Q\in \mathcal{W}}Q=\bigcup_{Q\in \mathcal{W}}2Q=\RR^{n}\setminus E$;
\item[(b)] $(1/\sqrt{n})\diam Q \leq d(x)\leq 4\diam Q$ for all $Q\in\mathcal{W}$ and for all $x\in Q$;
\item[(c)] $(1/2\sqrt{n})\diam Q\leq d(y)\leq (9/2)\diam Q$ for all $Q\in\mathcal{W}$ and for all $y\in 2Q$;
\item[(d)] if $Q,R\in\mathcal{W}$ and $2Q\cap 2R\neq\emptyset$, then $\diam R\leq 9\sqrt{n}\diam Q$;
\item[(e)] $\sum_{Q\in W} \chi_{2Q}(x)\lesssim_n 1$ for all $x\in\RR^n$.
\end{enumerate}
(2) (partition of unity) Let $\Phi=\{\phi_{Q}:Q\in\mathcal{W}\}$ be a smooth partition of unity subordinate to $2\mathcal{W}=\{2Q:Q\in\mathcal{W}\}$, i.e. a collection of $C^\infty$ functions $\phi_Q:\RR^n\rightarrow[0,1]$ such that
\begin{enumerate}
\item[(a)] $0\leq \phi_{Q}\leq \chi_{2Q}$ and $|\partial^\alpha\phi_{Q}|\lesssim_{n,|\alpha|} (\diam Q)^{-|\alpha|}\chi_{2Q}$ for all $Q\in\mathcal{W}$ and each multi-index $\alpha$ of order $|\alpha|\geq 1$; and,
\item[(b)] $\sum_{Q\in\mathcal{W}} \phi_Q\equiv \chi_{\RR^n\setminus E}$ and $\sum_{Q\in\mathcal{W}} \partial^\alpha\phi_Q\equiv 0$ for each multi-index $\alpha$ of order $|\alpha|\geq 1$.
\end{enumerate}
(3) (extension of $f$) For all $Q\in\mathcal{W}$, choose some $w_Q\in Q$ such that $|w_Q-w'_Q|=\inf_{x\in Q}d(x)$. Set $z_Q:=w'_Q$, $r_Q:=\diam Q$ and $A_Q:=A_{z_Q,r_Q}$.  Define $F:\RR^n\rightarrow\RR^N$ by the rule
\begin{equation*}
F(x)=\left\{\begin{array}{cl} f(x) &  \text{if }x\in E, \\ \sum_{Q\in\mathcal{W}} \phi_Q(x)A_Q(x) & \text{if } x\in\RR^n\setminus E.\end{array}\right.
\end{equation*}
(4) (extension of $\mathcal{A}$) Construct $\mathcal{A}^+=\{A_{x,r}:x\in\RR^n,r>0\}\supset \mathcal{A}$ by defining the maps $A_{x,r}:\RR^n\rightarrow\RR^N$ at each $x\in\RR^n\setminus E$ as follows. For all $0<r<d(x)/2$, define $A_{x,r}$ to be the first-order Taylor approximation of $F$ at $x$, i.e.~ the affine map given by the rule \begin{equation*} A_{x,r}(y)=F(x)+DF(x)(y-x)\quad\text{for all }y\in\RR^n.\end{equation*} For all $r\geq  d(x)/2$, define $A_{x,r}=A_{x',r}$.
\end{definition}

\begin{hypothesis}\label{h:setup} Let $E\subsetneq\RR^n$ be closed, let $f:E\rightarrow\RR^N$, and let $\mathcal{A}$ be a family of affine maps over $E$. Define $\{x'\}_{x\in\RR^n}$, $\mathcal{W}$, $\Phi$, $\{w_Q\}_{Q\in\mathcal{W}}$, $\{z_Q\}_{Q\in\mathcal{W}}$, $\{r_Q\}_{Q\in\mathcal{W}}$, $\{A_Q\}_{Q\in\mathcal{W}}$, $F:\RR^n\rightarrow\RR^N$ and $\mathcal{A}^+\supset\mathcal{A}$ by Definition \ref{d:extend}. Assume that $(f,E,\mathcal{A})$ is $\varepsilon$-almost affine for some $0<\varepsilon\leq\varepsilon_0<\sqrt{2}-1$.\end{hypothesis}

\begin{lemma} \label{l:AQest} Assume Hypothesis \ref{h:setup}. If $Q,R\in\mathcal{W}$ and $2Q\cap 2R\neq\emptyset$, then
$\|A'_Q\|\sim_n \|A'_R\|$, $\|A_{Q}'-A_R'\|\lesssim_n \varepsilon \|A'_Q\|$ and $|A_Q(x)-A_R(x)| \lesssim_n \varepsilon \|A_Q'\|\diam Q$ for all $x\in 2Q\cup 2R$.
\end{lemma}

\begin{proof} Let $Q,R\in\mathcal{W}$ and assume that there exists $y\in 2Q\cap 2R$. Recall that $\diam Q\sim_n \diam R$ by Definition \ref{d:extend}.1(d). It follows that \begin{equation*}\begin{split}
|z_{Q}-z_{R}| &\leq |z_Q-w_Q|+|w_Q-y|+|y-w_R|+|w_R-z_R|\\
& \leq 4\diam Q + \diam 2Q+\diam 2R+4\diam R \lesssim_n \diam Q.
\end{split}\end{equation*} Therefore, since $\varepsilon\lesssim 1$ and
\begin{equation*}
\tau(z_Q,r_Q,z_R,r_R)
=\frac{\max\{\diam Q,\diam R,2|z_Q-z_R|\}}{\min\{\diam Q,\diam R\}} \lesssim_n 1, \end{equation*}  we have $\Adiff{Q}{R}\lesssim_n \varepsilon \Amin{Q}{R}$ and $\Anorm{Q}\sim_n\Anorm{R}$  by Lemma \ref{l:main-est} (\ref{e:pre-c}) and (\ref{e:pre-d}). To continue, observe that for all $x\in 2R$, \begin{equation*} |x-z_Q| \leq |x-y|+|y-w_Q|+|w_Q-z_Q| \leq \diam 2R +\diam 2Q + 4\diam Q \lesssim_n \diam Q. \end{equation*} Similarly, $|x-z_R|\lesssim_n \diam R\sim_n \diam Q$ for all $x\in 2Q$. Hence $\dist(x,\{z_Q,z_R\})\lesssim_n \diam Q$ for all $x\in 2Q\cup 2R$. We conclude that $|A_Q(x)-A_R(x)| \lesssim_n \varepsilon \Anorm{Q}\diam Q$ for all $x\in 2Q\cup 2R$ by Lemma \ref{l:post-est} (\ref{e:post-b}).
\end{proof}

\begin{lemma} \label{l:far1} Assume Hypothesis \ref{h:setup}.
Let $x\in \RR^{n}$ and $r>0$. If $r\geq  d(x)/2$, then
\begin{equation} \label{e:far-a}
\|DF(y)-A_{x',r}'\|\lesssim_n T_\varepsilon\left(C(n)\frac{r}{ d(y)}\right) \varepsilon\Anorm{x',r} \quad\text{for all }y\in B^n(x,2r)\setminus E,\end{equation}
where $C(n)>0$ denotes some constant depending on at most $n$, and $T_\varepsilon:[1,\infty)\rightarrow[1,\infty)$ was defined above Lemma \ref{l:main-est}. For all $a\geq 1$, there exists $\tilde\varepsilon=\tilde\varepsilon(n,a)>0$ such that if, in addition, some $y_0\in B^n(x,2r)\setminus E$ satisfies $r\leq a d(y_{0})$ and $\varepsilon\leq\tilde\varepsilon$, then \begin{equation} \label{e:far-b} \max\{\|DF(y_0)\|,\Anorm{x',r}\} \leq 2 \min\{\|DF(y_0)\|,\Anorm{x',r}\}\end{equation} and
\begin{equation}
\|DF(y_0)-A_{x',r}'\|\lesssim_{n,a} \varepsilon\min\{\|DF(y_0)\|,\Anorm{x',r}\}.
\label{e:badDFbound}
\end{equation}\end{lemma}

\begin{proof} Fix $x\in\RR^n$ and $r>0$ such that $r\geq d(x)/2$. To start we first establish an auxiliary inequality for the affine maps $A_Q$ such that $2Q$ intersects $B^n(x,2r)$.

Suppose that $Q\in\mathcal{W}$ satisfies $2Q\cap B^n(x,2r)\neq\emptyset$. Let $p\in2Q\cap B^n(x,2r)$. For all $z\in B^n(z_Q,\diam Q)$,
\begin{equation*}\begin{split} |z-x'| &\leq |z-z_Q|+|z_Q-w_Q|+|w_Q-p|+|p-x|+|x-x'| \\ &\leq \diam Q+ 4\diam Q+\diam 2Q+2r+ d(x)\leq 7\diam Q+4r.\end{split}\end{equation*}
Since $p\in 2Q\cap B^n(x,2r)$, we have \begin{equation*}(1/2\sqrt{n})\diam Q\leq d(p)\leq |p-x|+|x-x'|\leq 4r.\end{equation*} Hence $\diam Q\lesssim_n r$. We conclude that $|z-x'|\lesssim_n r$ for all $z\in B^n(z_Q,\diam Q)$. Thus,
\begin{equation*} \tau(z_Q,r_Q,x',r) = \frac{\max\{\diam Q,r,2|z_Q-x'|\}}{\min\{\diam Q, r\}} \lesssim_n \frac{r}{\diam Q}.\end{equation*}
Therefore, for all $Q$ such that $2Q\cap B(x,2r)\neq\emptyset$,
\begin{equation}
\|A_{Q}'-A'_{x',r}\|\leq  T_\varepsilon\left(C(n)\frac{r}{\diam Q}\right)\varepsilon \Amin{Q}{x',r}
\label{e:Cr1}
\end{equation}
by Lemma \ref{l:main-est} (\ref{e:pre-a}),
where $C(n)>0$ is a constant depending on at most $n$ such that $C(n)r\geq \diam Q$.

Fix $y\in B(x,2r)\setminus E$ and choose $R\in \mathcal{W}$ such that $y\in R$.
For vectors $u\in \RR^{N}$ and $v\in\RR^{n}$, let $u\otimes v:\RR^n\rightarrow\RR^N$ be the linear transformation given by $(u\otimes v)(w)=\langle v, w\rangle u$ for all $w\in\RR^n$. By the product rule,
\begin{equation*}DF(y)-A_{x',r}'
 =\sum_{Q\in\mathcal{W}}A_{Q}(y)\otimes D\phi_{Q}(y)+\phi_Q(y)(A_{Q}'-A_{x', r}').
\end{equation*} Recall that the partition of unity was defined so that $\phi_Q(y)=0$ and $D\phi_Q(y)=0$ unless $y\in 2Q$, $\sum_{Q\in\mathcal{W}} \phi_Q\equiv \chi_{\RR^n\setminus E}$, and $\sum_{Q\in\mathcal{W}} D\phi_Q\equiv 0$. Thus, recalling the definition of the cube $R$ above,  \begin{equation*}DF(y)-A_{x',r}'=\sum_{\{Q\in\mathcal{W}:y\in 2Q\}}(A_{Q}(y)-A_{R}(y))\otimes D\phi_{Q}(y)+\phi_Q(y)(A_{Q}'-A_{x', r}').\end{equation*}
Therefore, by Lemma \ref{l:AQest}, Definition \ref{d:extend}.2(a), \eqref{e:Cr1}, the bound $1\leq T_\delta(t)$ for all $\delta>0$ and $t\geq 1$, the assumption $\varepsilon\lesssim 1$, and Definition \ref{d:extend}.1(e),
\begin{align*}
\|DF(y)  -A_{x', r}'\|
  &\lesssim_n \sum_{\{Q\in\mathcal{W}:y\in 2Q\}}\varepsilon \Anorm{Q}(\diam Q)(\diam Q)^{-1}
  + T_\varepsilon\left(C(n)\frac{r}{\diam Q}\right)\varepsilon\Anorm{x',r}\\
  &\lesssim_n \max_{\{Q\in\mathcal{W}:y\in 2Q\}} T_\varepsilon\left(C(n)\frac{r}{\diam Q}\right)\varepsilon \Anorm{x',r}.
\end{align*} Because $\diam Q\sim_n d(y)$ for $Q\in\mathcal{W}$ such that $y\in 2Q$, and $T_\varepsilon(t)$ is increasing in $t$, we obtain by increasing the value $C(n)>0$ as necessary that \begin{equation*}
\|DF(y)  -A_{x', r}'\| \lesssim_n T_\varepsilon\left(C(n)\frac{r}{d(y)}\right) \varepsilon \Anorm{x',r}.\end{equation*} This establishes \eqref{e:far-a}.

To conclude suppose that $y_0\in B(x,2r)\setminus E$ satisfies $r\leq ad(y_{0})$. Then, by \eqref{e:far-a}, \begin{equation*} \|DF(y_0)-A'_{x',r}\| \lesssim_n T_\varepsilon\left(C(n)a\right)\varepsilon \Anorm{x',r} \lesssim_{n,a} \varepsilon \Anorm{x',r}.\end{equation*} That is, $\|DF(y_0)-A'_{x',r}\|\leq C(n,a) \varepsilon \Anorm{x',r}$ for some constant $C(n,a)>0$ depending only on $n$ and $a$. Hence, by the triangle inequality (twice), \begin{equation*} \|DF(y_0)\| \leq C(n,a)\varepsilon\Anorm{x',r}+\Anorm{x',r} \quad \text{and} \quad
\Anorm{x',r}\leq C(n,a)\varepsilon\Anorm{x',r}+\|DF(y_0)\|.\end{equation*}
Therefore, \eqref{e:far-b} and \eqref{e:badDFbound} hold provided that $\varepsilon\leq 1/2C(n,a)=:\tilde\varepsilon(n,a)$.
\end{proof}

\begin{lemma} \label{l:far2} Assume Hypothesis \ref{h:setup}. Let $x\in\RR^n$ and $r>0$. If $r\geq d(x)/2$, then
\begin{equation}
|F(y)-A_{x',r}(y)|\lesssim_{n,\varepsilon_0} \varepsilon \Anorm{x',r} r\quad\text{for all }y\in B^n(x,r).
\label{e:F-Ax'}
\end{equation}
\end{lemma}

\begin{proof} Suppose that $x\in\RR^n$ and $r>0$ satisfy $r\geq d(x)/2$. There are two cases.

\setcounter{case}{0}

\begin{case} Suppose that $y\in B(x,2r)\cap E$. Then $|y-x'|\leq |y-x|+|x-x'|\leq 2r+2r\leq 4r$, $F(y)=f(y)$ and \begin{equation*} |F(y)-A_{x',r}(y)| \leq |f(y)-A_{x',4r}(y)| + |A_{x',4r}(y)-A_{x',r}(y)|.\end{equation*} On one hand, $|f(y)-A_{x',4r}(y)|\leq \varepsilon \Anorm{x',4r} 4r$, because $(f,E,\mathcal{A})$ is $\varepsilon$-almost affine and $y\in E\cap B^n(x',4r)$. On the other hand, \begin{equation*}|A_{x',4r}(y)-A_{x',r}(y)| \lesssim \varepsilon \Amin{x',4r}{x',r}4r\end{equation*} by Lemma \ref{l:post-est} (\ref{e:post-b}), because $\varepsilon\leq 1$, $|x'-x'|=0$, $\dist(y,\{x',x'\})\leq 4r$ and $\tau(x',4r,x',r)=4$. Moreover, $\Anorm{x',4r}\sim\Anorm{x',r}$, by Lemma \ref{l:main-est} (\ref{e:pre-d}). All together, $|F(y)-A_{x',r}(y)| \lesssim \varepsilon \Anorm{x',r}r$. \end{case}

\begin{case} Suppose that $y\in B(x,r)\setminus E$. Note that, away from its endpoints, the line segment connecting $y$ and $y'$ lies wholly within $B^n(x,2r)\setminus E$. Thus, by Case (1) and Lemma \ref{l:far1},
\begin{align*}
 |F(y)-A_{x',r}(y)| &\leq |F(y')-A_{x',r}(y')|+\int_{0}^{|y-y'|} \left\|DF\left(y'+t\frac{y-y'}{|y-y'|}\right)-A'_{x',r}\right\|dt\\
 &\lesssim_n \varepsilon\Anorm{x',r}r + \varepsilon\Anorm{x',r}\int_0^{|y-y'|}T_{\varepsilon}\left(C(n)\frac{r}{t}\right)dt.\end{align*} By a change of variables $u=t/(C(n)r)$, we obtain \begin{equation*}\int_0^{|y-y'|}T_{\varepsilon}\left(C(n)\frac{r}{t}\right)dt= C(n)r\int_0^{|y-y'|/(C(n)r)} T_\varepsilon(u^{-1})du
 \leq C(n)r \int_0^{3/C(n)} T_{\varepsilon_0}(u^{-1})du,\end{equation*} since $|y-y'|=\dist(y,E)\leq |y-x|+|x-x'|\leq r+2r=3r$ and $T_\varepsilon\leq T_{\varepsilon_0}$ pointwise. Finally, observe that $T_{\varepsilon_0}(u^{-1})= u^{-2\log_2(1+\varepsilon_0)}\left[2\log(u^{-1})+1\right]$ is integrable at $u=0$, since $\varepsilon_0<\sqrt{2}-1$ (i.e.~ $2\log_2(1+\varepsilon_0)<1$). It follows that $|F(y)-A_{x',r}(y)| \lesssim_{n,\varepsilon_0} \varepsilon \Anorm{x',r}r$.
\end{case}

Therefore, in both cases, $|F(y)-A_{x',r}(y)|\lesssim_{n,\varepsilon_0} \varepsilon \Anorm{x',r} r$ for all $y\in B(x,r)$.\end{proof}

\begin{lemma} \label{l:Omega} Assume Hypothesis \ref{h:setup}. There exists $\dot\varepsilon=\dot\varepsilon(n)$ such that if, in addition, $\varepsilon\leq \dot\varepsilon$, then for all $x\in\RR^n\setminus E$, for all $0<r<d(x)/2$, and for all $y\in B^n(x,r)\subset \RR^n\setminus E$,
\begin{equation} |F(y)-A_{x,r}(y)| \lesssim_n \varepsilon\frac{r}{d(x)} \|A_{x,r}'\|r,
\label{e:Omega<r/d}
\end{equation} and
\begin{equation} \|DF(y)-DF(x)\| \lesssim_n \varepsilon \frac{r}{d(x)}\min\{\|DF(y)\|,\|DF(x)\|\}.
\label{e:DF-lip}
\end{equation}
\end{lemma}

\begin{proof} First let us introduce some notation to facilitate the proof. Let $\mathcal{L}(U,V)$ denote the space of bounded linear transformations from a normed vector space $U$ to a normed vector space $V$, equipped with the operator norm.  For any $u\in U$  and $v\in\RR^n$, define $u\otimes v\in\mathcal{L}(\RR^n,U)$ by \begin{equation*}(u\otimes v)(w)=\langle v,w\rangle u\quad\text{for all }w\in \RR^n.\end{equation*} Also, for any $u\in U$ and $B\in \mathcal{L}(\RR^n,\RR^n)$, define $u\otimes B\in\mathcal{L}(\RR^n,\mathcal{L}(\RR^n,U))$ by \begin{equation*} (u\otimes B)(v)=u\otimes B(v)\quad\text{for all }v\in\RR^n.\end{equation*} If $G:\RR^n\rightarrow\RR^N$ is smooth near $y$, let $D^2G(y)\in \mathcal{L}(\RR^n,\mathcal{L}(\RR^n,\RR^N))$ denote the total derivative of the map $z\mapsto DG(z)$. We note for use below that \begin{equation}\label{e:D2G}\|D^2G(y)\| = \sup_{|p|=|q|=1} \left|\left(\sum_{j,k=1}^np_j\,q_k\,\frac{\partial^2 G^i}{\partial x_j \partial x_k}(y)\right)_{i=1}^N\right|\geq \max_{1\leq j,k\leq n}\left|\left(\frac{\partial^2 G^i}{\partial x_j \partial x_k}(y)\right)_{i=1}^N\right|,\end{equation} where $G=(G^1,\dots,G^N)$ and the inequality follows by letting $p$ and $q$ range over $\{e_1,\dots,e_n\}$.

Fix $\dot\varepsilon\leq \tilde\varepsilon(n,1)$ (see Lemma \ref{l:far1}) to be specified below and assume that $\varepsilon\leq \dot\varepsilon$. Let $x\in \RR^n\setminus E$, let $0<r<d(x)/2$, and let $y\in B^n(x,r)$.  Then \begin{equation}\label{e:Fz} F(z)=\sum_{\{Q\in\mathcal{W}:z\in 2Q\}}\phi_Q(z)A_Q(z)\quad \text{for all } z\in B^n(y,d(y)/2).\end{equation} First, differentiating (\ref{e:Fz}) at $z$ near $y$, we obtain
\begin{equation} \label{e:Fz1} DF(z)=\sum_{\{Q\in\mathcal{W}:z\in 2Q\}}A_{Q}(z)\otimes D\phi_{Q}(z)+\phi_{Q}(z)A_{Q}'\quad\text{for all }z\in B^n(y,d(y)/4).\end{equation}
Second, differentiating (\ref{e:Fz1}) at $z=y$, we obtain \begin{equation} \label{e:Fz2} D^2F(y)=\sum_{\{Q\in\mathcal{W}:y\in 2Q\}} 2A_{Q}'\otimes D\phi_{Q}(y)+A_{Q}(y)\otimes D^{2}\phi_{Q}(y).\end{equation}
Choose any cube $R\in\mathcal{W}$ such that $y\in R$. Because $\sum_{Q\in\mathcal{W}} D\phi_Q\equiv 0$ and $\sum_{Q\in\mathcal{W}} D^2\phi_Q\equiv 0$, we can rewrite (\ref{e:Fz2}) as
\begin{equation*}
 D^2F(y)=\sum_{\{Q\in\mathcal{W}:y\in 2Q\}}2(A_{Q}'-A_{R}')\otimes D\phi_{Q}(y)+ (A_{Q}(y)-A_{R}(y))\otimes D^{2}\phi_{Q}(y).\end{equation*}
Thus, by Lemma \ref{l:AQest}, Definition \ref{d:extend}.2(a), and Definition \ref{d:extend}.1(e),
\begin{equation}\label{e:D2norm} \|D^2F(y)\| \lesssim_n \sum_{\{Q\in\mathcal{W}:y\in 2Q\}} \varepsilon\frac{\Anorm{Q}}{\diam Q}+\varepsilon\frac{\Anorm{Q}\diam Q}{(\diam Q)^{2}} \lesssim_n \max_{\{Q\in\mathcal{W}:y\in 2Q\}} \varepsilon \frac{\Anorm{Q}}{\diam Q}.\end{equation} Suppose that $Q\in\mathcal{W}$ is such that $y\in 2Q$. On one hand, since $d(x)>0$ and $y\in B^n(x,d(x)/2)$, \begin{equation}\label{e:dxy} d(x)\sim d(y) \sim_n \diam Q.\end{equation} On the other hand, $|x-z_Q| \leq |x-y| + |y-w_Q|+|w_Q-z_Q| \lesssim_n \diam Q$. It follows that \begin{equation*} \tau\left(z_Q,\diam Q, x, \frac{d(x)}{2}\right) = \frac{\max\{\diam Q, d(x)/2,2|z_Q-x|\}}{\min\{\diam Q, d(x)/2\}}\lesssim_n 1.\end{equation*} Hence $\Anorm{Q}\sim_n \Anorm{x,d(x)/2}=\Anorm{x',d(x)/2}$ by Lemma \ref{l:main-est} (\ref{e:pre-d}). Combining this observation with (\ref{e:D2norm}), (\ref{e:dxy}), and Lemma \ref{l:far1} (\ref{e:far-b}), we conclude that \begin{equation}\label{e:D2FDF} \|D^2F(y)\| \lesssim_n \varepsilon \frac{\Anorm{x',d(x)/2}}{d(x)} \lesssim_n \varepsilon\frac{\|DF(x)\|}{d(x)}\quad\text{for all }y\in B^n(x,r).\end{equation}
Therefore, there exists a constant $C(n)>0$ depending on at most $n$ such that
\begin{equation}\label{e:DF-lip0}\|DF(y)-DF(x)\|\leq C(n)\varepsilon\frac{\|DF(x)\|}{d(x)} |y-x|\leq C(n) \varepsilon\frac{r}{d(x)}\|DF(x)\|\end{equation} for all $y\in B^n(x,r)$, where the first inequality holds by the mean value theorem and (\ref{e:D2FDF}). Applying the triangle inequality, (\ref{e:DF-lip0}) and the bound $r\leq d(x)/2$ (twice each), we see that \begin{equation}\begin{split}\label{e:DFnorms} \|DF(x)\|\leq C(n)\frac {\varepsilon}{2}\|DF(x)\|+\|DF(y)\|\quad\text{and}\\  \|DF(y)\|\leq C(n)\frac{\varepsilon}{2}\|DF(x)\| + \|DF(x)\|.\end{split}\end{equation} We now insist that $\dot\varepsilon \leq 1/C(n)$, which ensures that \begin{equation}\label{e:DFm2m} \max\{\|DF(x)\|,\|DF(y)\|\} \leq 2\min\{\|DF(x)\|,\|DF(y)\|\}\end{equation}  by (\ref{e:DFnorms}). Combining (\ref{e:DF-lip0}) and (\ref{e:DFm2m}) yields (\ref{e:DF-lip}). Finally, by Taylor's remainder theorem,
\begin{equation*} \begin{split}
F(y)-A_{x,r}(y) &= F(y)-F(x)-DF(x)(y-x) \\ &=\sum_{j,k=1}^n (y_j-x_j)(y_k-x_k)\int_0^1 (1-t)\left(\frac{\partial^2 F^i}{\partial x_j\partial x_k}(x+t(y-x))\right)_{i=1}^N\,dt. \end{split}\end{equation*} Therefore, by the triangle inequality, our assumption that $|y-x|\leq r$, and  (\ref{e:D2G}), \begin{align*}|F(y)-A_{x,r}(y)| &\leq \sum_{j,k=1}^n |y_j-x_j||y_k-x_k|\int_0^1 \left|\left(\frac{\partial^2 F^i}{\partial x_j\partial x_k}(x+t(y-x))\right)_{i=1}^N\right|\,dt\\
 &\leq r^2\int_0^1 \sum_{j,k=1}^n\left| \left(\frac{\partial^2 F^i}{\partial x_j\partial x_k}(x+t(y-x))\right)_{i=1}^N\right|dt\\
 &\leq r^2 \sup_{z\in[x,y]} \left(n^2\max_{1\leq j,k\leq n}\left| \left(\frac{\partial^2 F^i}{\partial x_j\partial x_k}(z)\right)_{i=1}^N\right|\right)
 \lesssim_n \sup_{z\in[x,y]}\|D^2F(z)\| r^2.\end{align*} Applying (\ref{e:D2FDF}) yields (\ref{e:Omega<r/d}).
\end{proof}

We are ready to prove Theorem \ref{t:extend1}.

\begin{proof}[Proof of Theorem \ref{t:extend1}] Assume Hypothesis \ref{h:setup} with parameter $\varepsilon_0:=\min\{2/5,\tilde\varepsilon(n,1),\dot\varepsilon(n)\}<\sqrt{2}-1$ (see Lemmas \ref{l:far1} and \ref{l:Omega}). We proceed in two steps.

\setcounter{step}{0}

\begin{step}The family $\mathcal{A}^+$ is $C\varepsilon$-compatible over $\RR^n$ for some constant $C=C(n)>1$.\end{step}

Fix $x,y\in\RR^n$ and $r,s>0$ such that $|x-y|\leq \max\{r,s\}$ and $1/2\leq r/s\leq 2$.  We shall estimate $\|A_{x,r}'-A_{y,s}'\|$ in three separate cases:

\setcounter{case}{0}

\begin{case} Assume that $r\geq  d(x)/2$ and $s\geq  d(y)/2$. Then
\begin{equation*}
|x'-y'|\leq |x'-x|+|x-y|+|y-y'| \leq  d(x) + \max\{r,s\} +  d(y)\leq 5\max\{r,s\}.
\end{equation*}
In particular,
\begin{equation*}\tau(x',r,y',s) = \frac{\max\{r,s,2|x'-y'|\}}{\min\{r,s\}} \leq \frac{10\max\{r,s\}}{\min\{r,s\}}\leq 20.\end{equation*}
Therefore,
$\|A_{x,r}'-A_{y,s}'\|= \|A_{x',r}'-A_{y',s}'\| \lesssim \varepsilon \min \{\|A_{x',r}'\|,\|A_{y',s}'\|\} =\varepsilon\min\{\|A_{x,r}'\|,\|A_{y,s}'\|\}$ by Lemma \ref{l:main-est} (\ref{e:pre-c}).
 \end{case}

\begin{case} Assume that $r\geq  d(x)/2$ and $s< d(y)/2$. Since $y\in B^n(x,2r)\setminus E$ and $r\leq 2s<d(y)$, \begin{equation*} \|A'_{x,r}- A'_{y,s}\| = \|A_{x',r}- DF(y)\| \lesssim_n \varepsilon \min\{\Anorm{x',r},\|DF(y)\|\}=\Amin{x,r}{y,s}\end{equation*} by Lemma \ref{l:far1}.
\end{case}

\begin{case} Assume that $r< d(x)/2$ and $s< d(y)/2$. Since $|x-y|\leq r$,  \begin{equation*} \Adiff{x,r}{y,s}=\|DF(x)-DF(y)\| \lesssim_n \varepsilon \min\{\|DF(x)\|,\|DF(y)\|\}=\varepsilon\Amin{x,r}{y,s}\end{equation*} by Lemma \ref{l:Omega}.
\end{case}

Therefore, $\mathcal{A}^+$ is $C\varepsilon$-compatible for some constant $C>1$ depending only on $n$.

\begin{step}$(F,\RR^n,\mathcal{A}^+)$ is $C_0\varepsilon$-almost affine for some constant $C_0=C_0(n)>1$.\end{step}

Let $y\in B^n(x,r)$. On one hand, if  $r\geq  d(x)/2$, then $|F(y)-A_{x,r}(y)|\lesssim_{n} \varepsilon \|A_{x,r}'\| r$ for all $y\in B^n(x,r)$, by Lemma \ref{l:far2}. On the other hand, if $r< d(x)/2$, then $|F(y)-A_{x,r}(y)|\lesssim_n \varepsilon \|A_{x,r}'\| r$ for all $y\in B^n(x,r)$, by Lemma \ref{l:Omega}.
Therefore, $(F,\RR^n,\mathcal{A}^+)$ is $C_0\varepsilon$-almost affine for some constant $C_0>1$ depending only on $n$.\end{proof}

We now derive Theorem \ref{t:extend2} from Theorem \ref{t:extend1}.

\begin{proof}[Proof of Theorem \ref{t:extend2}] Let $H\geq 1$, and let $1<\rho\leq 2$ be given. Fix $\varepsilon_1\in(0,\varepsilon_0]$ to be chosen later and put $C_1=C_0$, where $\varepsilon_0$ and $C_0$ are the constants from Theorem \ref{t:extend1}. Assume Hypothesis \ref{h:setup} with $\varepsilon\leq \varepsilon_1$. In addition, assume that $f$ is nonconstant and $\lambda_n(A'_{x,r})\leq H\lambda_1(A_{x,r}')$ for all $x\in E$ and $r>0$. By (the proof of) Theorem \ref{t:extend1}, $(F,\RR^n,\mathcal{A}^+)$ is $C_1\varepsilon$-almost affine. Thus, to establish Theorem \ref{t:extend2}, all that remains is to show that $F$ is weakly $\rho H$-quasisymmetric. We break the argument into three steps.
\setcounter{step}{0}
\begin{step} If $\varepsilon_1$ is sufficiently small, then $\lambda_n(A_{x,r}')\leq ((\rho+1)/2)H\lambda_1(A_{x,r}')$ for all $A_{x,r}\in\mathcal{A}^+$.\end{step}

Fix $x\in\RR^n\setminus E$. On one hand, if $r\geq  d(x)/2$, then $A_{x,r}=A_{x',r}\in\mathcal{A}$. Hence $\lambda_n(A'_{x,r}) \leq H\lambda_1(A'_{x,r})$ for all $r\geq d(x)/2$. On the other hand, suppose that $0<r< d(x)/2=:\delta$. Then $A_{x,r}$ is the first-order Taylor approximation of $F$ at $x$ and $A'_{x,r}=DF(x)$.
By Lemma \ref{l:far1} (\ref{e:badDFbound}), we have $\Adiff{x,r}{x',\delta}\leq C_2\varepsilon \|A_{x,r}'\|=C_2\varepsilon\lambda_n(A'_{x,r})$ for some $C_2=C_2(n)>0$.  Thus, since $\lambda_n(A'_{x',\delta})\leq H\lambda_1(A'_{x',\delta})$, \begin{equation*}\begin{split}
\lambda_n(A'_{x,r}) =\|A_{x,r}'\| \leq \Adiff{x,r}{x',\delta}+\Anorm{x',\delta} \leq C_2\varepsilon \lambda_n(A'_{x,r}) + H\inf_{|z|=1}|A'_{x',\delta}z|\\
\leq (1+H)C_2\varepsilon\lambda_n(A'_{x,r})+H\inf_{|z|=1}|A'_{x,r}z| \leq 2HC_2\varepsilon\lambda_n(A'_{x,r})+H\lambda_1(A'_{x,r}).
\end{split}\end{equation*} In particular, $\lambda_n(A'_{x,r}) \leq ((\rho+1)/2)H \lambda_1(A'_{x,r})$ if $\varepsilon_1\leq ((\rho-1)/(\rho+1))/2HC_2$. Therefore, $\lambda_n(A_{x,r}')\leq ((\rho+1)/2)H\lambda_1(A_{x,r}')$ for all $A_{x,r}\in\mathcal{A}^+$ if $\varepsilon_1$ is small enough depending only on $n$, $H$ and $\rho$.

\begin{step} If $\varepsilon_1$ is sufficiently small, then $H_F(\RR^n)\leq \rho H$. \end{step}

Fix $x,y,z\in \RR^n$ such that $|y-x|\leq |z-x|=:r$. Assume that $\varepsilon_1$ satisfies the constraints of Step 1. Then, since $F$ is $C_1\varepsilon$-close to $\mathcal{A}^+$ and $A_{x,r}$ is weakly $((\rho+1)/2)H$-quasisymmetric, \begin{align*}|F(y)-F(x)| &\leq 2C_1\varepsilon\|A_{x,r}'\|r + |A_{x,r}(y)-A_{x,r}(x)|\\ &\leq 2C_1\varepsilon\|A_{x,r}'\|r + \left(\frac{\rho+1}{2}\right)H |A_{x,r}(z)-A_{x,r}(x)|\\
&\leq (2+(\rho+1)H)C_1\varepsilon\|A_{x,r}'\|r + \left(\frac{\rho+1}{2}\right)H|F(z)-F(x)|.\end{align*} To continue, observe by similar reasoning that
 \begin{equation*}\begin{split}
 \|A_{x,r}'\|r
 &\leq \frac{\rho+1}{2}H|A_{x,r}(z)-A_{x,r}(x)| \\
 &\leq (\rho+1)HC_1\varepsilon\|A'_{x,r}\|r + \left(\frac{\rho+1}{2}\right)H|F(z)-F(x)|.
 \end{split}\end{equation*}
 Hence, if $\varepsilon_1 \leq 1/2(\rho+1)HC_1$, then $\|A_{x,r}'\|r \leq (\rho+1)H|F(z)-F(x)|$ and \begin{equation*}\begin{split}
 |F(y)-F(x)|
 &\leq (2+(\rho+1)H)C_1\varepsilon (\rho+1)H|F(z)-F(x)| + \frac{\rho+1}{2}H|F(z)-F(x)|\\
 &\leq 15HC_1\varepsilon H|F(z)-F(x)| + \frac{\rho+1}{2} H|F(z)-F(x)|.
 \end{split}\end{equation*}
 Therefore, if $\varepsilon_1\leq (\rho-1)/30HC_1$, then $|F(y)-F(x)|\leq \rho H|F(z)-F(x)|$ for all $x,y,z$ such that $|x-y|\leq |x-z|$. That is, $H_F(\RR^n)\leq \rho H$ if $\varepsilon_1$ is sufficiently small.

\begin{step} If $\varepsilon_1$ is sufficiently small, then $F$ is weakly $\rho H$-quasisymmetric. \end{step}

First, assume that $C_1\varepsilon_1<\hat\varepsilon$, which guarantees that $F$ is (locally H\"older) continuous by Lemma \ref{l:Holder}. Second, note that $F$ is nonconstant, since $f$ is nonconstant and $F$ extends $f$. Third,  assume that $\varepsilon_1$ is small enough so that the conclusion of Step 2 holds. Then, because $F$ is continuous and nonconstant and $H_F(\RR^n)\leq \rho H$, the map $F$ is weakly $\rho H$-quasisymmetric, by Lemma \ref{l:wqs-crit}.

To complete the proof of the theorem, choose $\varepsilon_1$ sufficiently small so that the conclusion of Steps 1 and 3 hold. Reviewing each of the constraints imposed on $\varepsilon_1$ in Steps 1 through 3 above, we see that $\varepsilon_1$ can be chosen to depend only on $n$, $H$ and $\rho$.
\end{proof}


\section{Extensions of almost affine maps II: beta number estimates}
\label{s:extend2}

The goal of this section is to prove Theorem \ref{t:beta}, which for convenience we now restate.

\begin{theorem}\label{t:beta2} Suppose $1\leq n\leq N-1$. For all $\varepsilon>0$, there  exists $\varepsilon_*=\varepsilon_*(\varepsilon,n)>0$ with the following property. If for some $x\in\RR^n$ and $r>0$ a map $f:\RR^N\rightarrow\RR^N$ is $\varepsilon_*$-almost affine over $B^n(x,9r)$, $f|_{B^N(x,3r)}$ is a topological embedding and  $\tH_f(B^N(x,3r))\leq \varepsilon_*$, and there exist a closed set $E\subset B^n(x,r)$ and constants $\gamma_E>0$ and $C_E>0$ such that
\begin{equation}
\label{e:beta-d2} \diam E \geq \gamma_E \diam B^n(x,r)
\end{equation} and
\begin{equation}
\label{e:beta-h2}
\int_{0}^{r} \tH_f(B^N(y,s))^2\, \frac{ds}{s}\leq C_E\quad\text{for all }y\in E,
\end{equation}
then there exists a quasisymmetric map $F:\RR^n\rightarrow\RR^N$ such that $F|_E= f|_E$, $F$ is $\varepsilon$-almost affine over $\RR^n$, $\tH_F(\RR^n)\leq \varepsilon$, $\diam F(B^n(x,r))\sim_{n,N,\gamma_E} \diam f(B^n(x,r))$, and \begin{equation}
\label{e:beta-c2}
\int_0^{\infty} \beta_{F(\RR^n)}(F(y),s)^2\, \frac{ds}{s}\lesssim_{n,N} C_E+\varepsilon^2\quad\text{for all }y\in \RR^n.
\end{equation}
\end{theorem}

The extension in Theorem \ref{t:beta} / Theorem \ref{t:beta2} will be constructed by applying Theorem \ref{t:extend2} with a compatible family $\mathcal{A}$ of affine maps over $E$  satisfying two additional properties:
$\mathcal{A}$ is stable at large scales (recall Definition \ref{d:stable}); and, $\mathcal{A}$ is adapted to $f$ at small scales (recall Definition \ref{d:adapt}). More precisely, we use $\mathcal{A}$ given by the following lemma.

\begin{lemma} \label{l:family} For all $n\geq 1$ and $\varepsilon>0$, there exists $\varepsilon'=\varepsilon'(\varepsilon,n)>0$ with the following property. Let $x\in\RR^n$, let $r>0$, and let $E\subset B^n(x,r)$ be a closed set. If $f:\RR^n\rightarrow\RR^N$ is $\varepsilon'$-almost affine over $B^n(x,9r)$ and $\tH_f(B^n(x,3r))\leq \varepsilon'$, then $(f,E,\mathcal{A})$ is $\varepsilon$-almost affine for some $\varepsilon$-compatible family $\mathcal{A}$ of affine maps  over $E$ such that \begin{equation}\label{e:family1} \lambda_n(A_{y,s})\leq (1+\varepsilon)\lambda_1(A_{y,s})\quad\text{for all }A_{y,s}\in\mathcal{A},\end{equation} $\mathcal{A}$ is adapted to $f$ at small scales, and $\mathcal{A}$ is stable at large scales. Moreover, $\mathcal{A}$ can be chosen such that for all $y\in E$ and $0<s\leq\diam E$, \begin{equation}\label{e:family2}
|f(z)-A_{y,s}(z)|\leq \varepsilon\Anorm{y,s}s\quad\text{for all }z\in B^n(y,s).\end{equation}
\end{lemma}

\begin{proof} Let $x\in \RR^n$, let $r>0$, and let $E\subset B^n(x,r)$ be closed. Let $\varepsilon>0$ arbitrary be given, and fix $\varepsilon'>0$ to be specified later. Suppose that $f:\RR^n\rightarrow\RR^N$ is $\varepsilon'$-almost affine over $B^n(x,9r)$, and $\tH_f(B^n(x,3r))\leq \varepsilon'$. We require $P\varepsilon'\leq 1$. Then, by Lemma \ref{l:adapt}, there exists a $P\varepsilon'$-compatible family $\mathcal{B}$ of affine maps over $B^n(x,3r)$ such that $\mathcal{B}$ is adapted to $f$ at small scales and $(f,B^n(x,3r),\mathcal{B})$ is $P\varepsilon'$-almost affine. Let $\mathcal{B}_E=\{B_{x,r}:x\in E,r>0\}$ denote the restriction of $\mathcal{B}$ to affine maps over $E$. If $B_{y,s}\in\mathcal{B}_E$ for some $y\in E$ and $0<s\leq \diam E$, then \begin{equation*} B_{y,s}(y)=f(y)\quad \text{and}\quad B_{y,s}(y+se_i)=f(y+se_i)\quad\text{for all }i=1,\dots,n,\end{equation*} since $s\leq \diam B^n(x,3r)$ and $\mathcal{B}$ is adapted to $f$ at small scales. In other words, $\mathcal{B}_E$ is adapted to $f$ at small scales, as well. Choose any $y_*\in E$. Then, by Lemma \ref{l:stable}, we know that $(f,E,\mathcal{A})$ is $P\varepsilon'$-almost affine, where the family $\mathcal{A}$ of affine maps over $E$ is defined by \begin{equation*} A_{y,s}=\left\{\begin{array}{cl} B_{y,s} &\text{if }s\leq \diam E, \\ B_{y_*,\diam E} &\text{if }s>\diam E.\end{array}\right.\end{equation*} In particular, $\mathcal{A}$ is a $P\varepsilon'$-compatible family of almost affine maps over $E$ that is simultaneously adapted to $f$ at small scales and stable at large scales.

Next we estimate the weak quasisymmetry of affine maps in $\mathcal{A}$. Fix $y\in E$ and $0<s\leq \diam E$. Then $A_{y,s}=B_{y,s}$ and $B^n(y,s)\subset B^n(x,3r)$. It follows that \begin{equation}\label{e:family2a}
|f(z)-A_{y,s}(z)|\leq P\varepsilon'\Anorm{y,s}s\quad\text{for all }z\in B^n(y,s),\end{equation} since $(f,B^n(x,3r),\mathcal{B})$ is $P\varepsilon'$-almost affine. Also $f(y)=A_{y,s}(y)$, because $\mathcal{A}$ is adapted to $f$ at small scales. Hence \begin{equation*} \begin{split}
\Anorm{y,s}s&=\sup_{|z-y|= s}|A_{y,s}(z)-A_{y,s}(y)| \leq P\varepsilon' \Anorm{y,s}s + \sup_{|z-y|= s} |f(z)-f(y)| \\
&\leq P\varepsilon'\Anorm{y,s}{s} + (1+\varepsilon')\inf_{|z-y|= s}|f(z)-f(y)| \\ &\leq P\varepsilon'(2+\varepsilon')\Anorm{y,s}s + (1+\varepsilon')\lambda_1(A'_{y,s})s.
\end{split}\end{equation*} Thus, stipulating $(1+\varepsilon')/(1-P\varepsilon'(2+\varepsilon'))\leq 1+\varepsilon$, \begin{equation}\label{e:family3}\lambda_n(A'_{y,s})=\|A'_{y,s}\| \leq (1+\varepsilon)\lambda_1(A'_{y,s})\end{equation} for all $y\in E$ and $0<s\leq \diam E$.
Recall that if $y\in E$ and $s>\diam E$, then $A_{y,s}=A_{y_*,\diam E}$. Therefore, (\ref{e:family3}) holds for all $y\in E$ and $s>0$.

Examining the constraints put in place at various stages above, the lemma holds provided that $\varepsilon'>0$ is sufficiently small such that $P\varepsilon'\leq \min\{1,\varepsilon\}$ and $(1+\varepsilon')/(1-P\varepsilon'(2+\varepsilon'))\leq 1+\varepsilon$. \end{proof}

At last, we are ready to prove Theorem \ref{t:beta} / Theorem \ref{t:beta2}.

\begin{proof}[Proof of Theorem \ref{t:beta} / Theorem \ref{t:beta2}]
It suffices to prove the theorem when $\varepsilon>0$ is small. Thus, let $\varepsilon\in(0,\sqrt{2}-1)$ small enough such that \begin{equation}\label{e:e-up}\left(1+\frac{3}{4}\varepsilon\right)\left(\frac{1}{2}+2\varepsilon\right)\leq 1\end{equation} be given. Choose $\varepsilon_*\in(0,\varepsilon)$ to be specified below.  Fix $x\in\RR^n$, $r>0$, and a closed set $E\subset B^n(x,r)$ satisfying (\ref{e:beta-d2}) for some $\gamma_E>0$.  Suppose a map $f:\RR^N\rightarrow\RR^N$ is $\varepsilon_*$-almost affine over $B^n(x,9r)$, $f|_{B^N(x,3r)}$ is a topological embedding and $\tH_f(B^N(x,3r))\leq \varepsilon_*$, and there exists $C_E>0$ such that (\ref{e:beta-h2}) holds. Let $C_1=C_1(n)$ and $\varepsilon_1=\varepsilon_1(n,H,\rho)$ be the constants from Theorem \ref{t:extend2} corresponding to \begin{equation*} H=1+\tfrac12\varepsilon\quad\text{and}\quad \rho=\min\left\{2,\frac{1+\varepsilon}{1+\tfrac12\varepsilon}\right\}.\end{equation*} Let $\varepsilon'=\varepsilon'(\min\{\varepsilon/2,\varepsilon/C_1\},n)$ be the constant from Lemma \ref{l:family}. Assume $\varepsilon\leq C_1\varepsilon_1$ and $\varepsilon_*\leq \varepsilon'$.

By Lemma \ref{l:family}, we can find a family $\mathcal{A}$ of affine maps over $E$ such that $\mathcal{A}$ is adapted to $f$ at small scales, $\mathcal{A}$ is stable at large scales, $\lambda_n(A_{y,s}') \leq H \lambda_1(A_{y,s}')$ for all $y\in E$ and $s>0$, and $(f,E,\mathcal{A})$ is $\varepsilon/C_1$-almost affine. Moreover, we can choose $\mathcal{A}$ such that for all $y\in E$ and $0<s\leq \diam E$, \begin{equation}\label{e:family2e} |f(z)-A_{y,s}(z)|\leq \frac{\varepsilon}{2}\Anorm{y,s}s\quad\text{for all }z\in B^n(y,s).\end{equation} Note that the map $f|_E$ is nonconstant, since $f|_{B^N(x,3r)}$ is an embedding and $\diam E>0$. Thus $f|_E$ satisfies the hypotheses of Theorem \ref{t:extend2}.
Using the proof of Theorem \ref{t:extend2}, extend $f|_E$ to a map $F:\RR^n\rightarrow\RR^N$ and extend $\mathcal{A}=\{A_{y,s}:y\in E,s>0\}$ to a family of affine maps $\mathcal{A}^+=\{A_{y,s}:y\in\RR^n,s>0\}$ over $\RR^n$ such that \begin{equation} \label{e:Faa} (F,\RR^n,\mathcal{A}^+) \text{ is $\varepsilon$-almost affine,}\end{equation} \begin{equation}\label{e:Fwqs} \text{$F$ is weakly $(1+\varepsilon)$-quasisymmetric,}\end{equation} and \begin{equation}\label{e:Auqs}\lambda_n(A'_{y,s}) \leq \left(1+\tfrac{3}{4}\varepsilon\right)\lambda_1(A'_{y,s})\quad\text{for all }y\in\RR^n, s>0.\end{equation} Then $F$ is quasisymmetric by (\ref{e:Fwqs}) and Corollary \ref{c:wqs2qs}.  In fact, since $H_F(\RR^n)\leq 1+\varepsilon\leq 2$ and $H_f(B^n(x,3r))\leq 1+\varepsilon_*\leq 2$, Corollary \ref{c:wqs2qs} implies that the maps $F$ and $f|_{B^n(x,3r)}$ are uniformly quasisymmetric with some control function determined by $n$ and $N$. Hence \begin{equation*} \frac{\diam f(E)}{\diam f(B^n(x,r))} \sim_{n,N,\gamma_E} \frac{\diam F(E)}{\diam F(B^n(x,r))}\end{equation*} by (\ref{e:beta-d2}) and Lemma \ref{l:QS-comp}. Because $f(E)=F(E)$, we conclude that \begin{equation*} \diam f(B^n(x,r))\sim_{n,N,\gamma_E} \diam F(B^n(x,r)).\end{equation*} To complete the proof, we must convert the Dini conditions (\ref{e:beta-h2}) on $\tH_{f}(B^n(y,\cdot))^2$ at all $y\in E$ into Dini conditions (\ref{e:beta-c2}) on $\beta_{F(\RR^n)}(F(y),\cdot)^2$ at all $y\in\RR^n$.

Before moving on to the main argument, we stop and record a few estimates for $F$ and $f$. First, by (\ref{e:e-up}), (\ref{e:Faa}), (\ref{e:Auqs}), and Lemma \ref{l:inradius}, we have that for all $y\in\RR^n$ and $s>0$, \begin{equation}\label{e:Fir} \frac{1}{2}\Anorm{y,s}s \leq |F(z)-F(y)|\leq 2\Anorm{y,s}s\quad\text{for all }z\in\partial B^n(y,s).\end{equation} Since $H_F(\RR^n)\leq 1+\varepsilon\leq 2$, it follows that for all $y\in\RR^n$ and $s>0$, \begin{equation*} \inf_{|w-y|\geq s}|F(w)-F(y)| \geq \frac{1}{2}\sup_{|z-y|=s}|F(z)-F(y)| \geq \frac{1}{4} \Anorm{y,s}s.\end{equation*} Thus, for all $y\in \RR^n$ and $s>0$, \begin{equation}\label{e:Firb} F(\RR^n)\cap B^N\left(F(x),\frac{1}{4}\Anorm{y,s}s\right)\subset F(B^n(y,s)).\end{equation}
We can obtain similar estimates for $f$ in place of $F$ by using (\ref{e:family2e}) and repeating the proof of Lemma \ref{l:inradius}. Indeed,  by (\ref{e:e-up}), (\ref{e:family2e}), and by the fact that $\mathcal{A}$ is adapted to $f$ at small scales, we see that for all $y\in E$ and $0<s\leq \diam E$, \begin{equation}\begin{split}\label{e:for} |f(z)-f(y)|&\leq |f(z)-A_{y,s}(z)|+|A_{y,s}(z)-A_{y,s}(y)|
\\ &\leq \left(1+\frac{\varepsilon}{2}\right)\Anorm{y,s}s\leq 2\Anorm{y,s}s\quad\text{for all }z\in\partial B^n(y,s).\end{split}\end{equation} Also, for all $y\in E$ and $0<s\leq \diam E$, \begin{equation}\begin{split}\label{e:fir} |f(z)-f(y)|&\geq |A_{y,s}(z)-A_{y,s}(y)| - |f(z)-A_{y,s}(z)|\\ &\geq \left(H^{-1}-\frac{\varepsilon}{2}\right)\Anorm{y,s}\geq \frac{1}{2}\Anorm{y,s}\quad\text{for all }z\in\partial B^n(y,s).\end{split}\end{equation} Similar considerations give that for all $y\in E$ and $0<s\leq \diam E$, \begin{equation}\label{e:fdiam} \Anorm{y,s}s\leq \diam f(B^n(y,s))\leq 3\Anorm{y,s}s.\end{equation}

Now, because $\mathcal{A}$ is stable at large scales, there exists $y_*\in E$ such that $A_*:=A_{y_*,\diam E}=A_{y,s}$ for all $y\in E$ and $s>\diam E$. Assign $s_*:=\Anorm{*}\diam E$. We note that for all $y\in E$, \begin{equation*}\tau(y,\frac12\diam E,y_*,\diam E)\leq 4.\end{equation*} Hence, $\Anorm{*}\sim \Anorm{y,\frac12\diam E}$ for all $y\in E$ by Lemma \ref{l:main-est} (\ref{e:pre-d}). Therefore, in view of (\ref{e:fdiam}), \begin{equation}\label{e:ssim} s_* \sim \diam f\left(B^n\left(y,\frac{1}{2}\diam E\right)\right)\quad\text{for all }y\in E.\end{equation}
The argument now breaks up into three major steps.

\setcounter{step}{0}
\begin{step} For all $y\in \RR^n$ and for all $\tau>0$, $\int_{\tau s_*}^\infty \beta_{F(\RR^n)}(F(y),s)^2 s^{-1}ds\lesssim (\varepsilon/\tau)^2$. \end{step}

The underlying reason is simple: Since $\mathcal{A}$ is stable at large scales, $F(\RR^n)$ can be approximated by a fixed $n$-dimensional plane at all locations and large scales. We now supply some details. Recall that $d(y)=\dist(y, E)$ for all $y\in\RR^n$. On one hand, if $y\in\RR^n$ and $d(y)> (9/2)\diam E$, then $\diam Q\geq (2/9)d(y)>\diam E$ for all cubes $Q\in\mathcal{W}$ such that $y\in 2Q$ by Definition \ref{d:extend}.1(c). Hence, for all $y\in\RR^n$ such that $d(y)>(9/2)\diam E$, \begin{equation*} F(y)= \sum_{Q\in\mathcal{W}} \phi_Q(y) A_Q(y) = \sum_{Q\in\mathcal{W}} \phi_Q(y) A_*(y)=A_*(y).\end{equation*}
On the other hand, if $y\in \RR^n$ and $d(y)\leq (9/2)\diam E$, then $y\in B^n(y_*,(11/2)\diam E)$. Thus, \begin{equation*} |F(y)-A_{*}(y)| \leq \varepsilon \Anorm{*} \frac{11}{2}\diam E\lesssim\varepsilon s_*\quad\text{whenever }d(y)\leq (9/2)\diam E,\end{equation*} because $(F,\RR^n,\mathcal{A}^+)$ is $\varepsilon$-almost affine. Comparing $F(\RR^n)$ with the plane $A_*(\RR^n)$, we obtain \begin{equation*}\beta_{F(\RR^n)}(F(y),s) \lesssim \frac{1}{s} \left(\varepsilon s_*\right)\quad\text{for all }y\in\RR^n,s>0.\end{equation*} Therefore, for all $\tau>0$, \begin{equation*}\begin{split} \int_{\tau s_*}^\infty \beta_{F(\RR^n)}(F(y),s)^2\frac{ds}{s}
&\lesssim \left(\varepsilon s_{*}\right)^2 \int_{\tau s_*}^\infty \frac{ds}{s^3}\lesssim (\varepsilon/\tau)^2.\end{split}\end{equation*} This completes Step 1.

\begin{step} For all $y\in E$, $\int_{0}^{\infty} \beta_{F(\RR^n)}(F(y),s)^2 s^{-1}ds\lesssim_{n,N} C_E+\varepsilon^2$. \end{step}

To establish this step, we use the assumption that $\mathcal{A}$ is adapted to $f$ on small scales.  Fix $y\in E$ and $0<s \leq \diam E$, set  \begin{equation*}\beta_s:=\beta_{f(\RR^n)}\left(f(y),\frac{1}{4}\Anorm{y,s} s\right),\end{equation*} and choose an $n$-dimensional plane $L$ in $\RR^N$ such that \begin{equation*} \dist(p,L) \leq \frac{1}{2}\beta_s\Anorm{y,s} s\quad\text{for all }p\in f(\RR^n)\cap B^N\left(f(y),\frac{1}{4}\Anorm{y,s}s\right).\end{equation*} (The reason that we work with the scale $\frac{1}{4}\Anorm{y,s}s$ will become apparent below.) Fix $\mu\in(0,1)$ to be chosen momentarily. By (\ref{e:for}) and Lemma \ref{l:main-est} (\ref{e:pre-b}), \begin{equation*} |f(z)-f(y)| \leq 2\Anorm{y,\mu s}\mu s \leq 2(1+T_\varepsilon(1/\mu)\varepsilon)\Anorm{y,s}\mu s\quad\text{for all }z\in \partial B^n(y,\mu s).\end{equation*} Since $\varepsilon<\sqrt{2}-1$, \begin{equation*}2(1+T_\varepsilon(1/\nu)\varepsilon)\nu \lesssim \nu^{1-2\log_2(1+\varepsilon)}\log(1/\nu)\rightarrow 0\quad\text{as $\nu\rightarrow 0$}. \end{equation*} Hence we may choose $\mu$ to be sufficiently small so that $f(B^n(y,\mu s))\subset B^N(f(y),\frac{1}{4}\Anorm{y,s}s)$, which guarantees \begin{equation}\label{e:s2a} \dist(f(z),L) \leq \frac{1}{2}\beta_s \Anorm{y,s} s\quad\text{for all }z\in B^n(y,\mu s).\end{equation} Fix $\lambda\in(0,1)$ to be specified later (look after \eqref{e:lambda}), depending only on $n$. We shall use the $n$-dimensional plane $L$ to estimate $\beta_{F(\RR^n)}(F(y),\frac{1}{4}\Anorm{y,\lambda\mu s}\lambda\mu s)$.

Let $q\in F(\RR^n)\cap B^N(F(y),\frac{1}{4}\Anorm{y,\lambda\mu s}\lambda\mu s)$. By \eqref{e:Firb} applied at scale $\lambda\mu s$, we can write $q=F(z)$ for some $z\in B^n(y,\lambda\mu s)$. If $z\in E\cap B(y,\lambda\mu s)$, then $F(z)=f(z)$, and  by (\ref{e:s2a}), \begin{equation}\label{e:s2b}\dist(F(z),L) \leq \frac{1}{2}\beta_s \Anorm{y,s}s.\end{equation} On the other hand, suppose that $z\in B^n(y,\lambda\mu s)\setminus E$. Then \begin{equation}\label{e:Fdef} F(z) = \sum_{Q}\phi_Q(z)A_Q(z)=\sum_{Q} \phi_Q(z)A_{z_Q,\diam Q}(z),\end{equation} where the sum is over all cubes $Q\in\mathcal{W}$ such that $z\in 2Q$. For any such cube $Q$, we have \begin{equation}\label{e:zzq} |z-z_Q|\leq |z-w_Q|+|w_Q-z_Q|\leq \diam 2Q+4\diam Q= 6\diam Q\leq 12\sqrt{n}\lambda\mu s,\end{equation} where the last inequality holds by Definition 7.3.1(c) and the bound $d(z)\leq |z-y|\leq \lambda\mu s$. Hence \begin{equation}\label{e:lambda}|z_Q-y| \leq |z_Q-z|+ |z-y| \leq 12\sqrt{n}\lambda\mu s + \lambda\mu s \leq 13\sqrt{n} \lambda\mu s.\end{equation} Setting $\lambda:=1/26\sqrt{n}$ ensures that $z_Q\in E\cap B^n(y,\frac12\mu s)$ and $\diam Q\leq 2\sqrt{n}\lambda\mu s<\frac{1}{2}\lambda\mu s$. Since $\mathcal{A}$ is adapted to $f$ at small scales and $\diam Q<\frac{1}{2}\lambda\mu s<\diam E$, we have \begin{equation*}A_{Q}(z_Q+(\diam Q)e_i) = f(z_Q+(\diam Q)e_i)\quad\text{for all }i=0,\dots n,\end{equation*} where $e_0=0$ and $e_1,\dots,e_n$ is a standard basis for $\RR^n$. Thus, by (\ref{e:s2a}), \begin{equation}\label{e:s2c} \dist (A_{Q}(z_Q+(\diam Q)e_i),L) \leq \frac{1}{2}\beta_s\Anorm{y,s} s\quad\text{for all }i=0,\dots n.\end{equation} We now introduce an auxiliary affine map $B_Q:\RR^n\rightarrow\RR^N$, with an aim of invoking Lemma \ref{l:A-B}, as follows. For each $i=0,\dots, n$, choose $u_i\in L$ such that \begin{equation}\label{e:s2d} |A_{Q}(z_Q+(\diam Q)e_i)-u_i|=\dist (A_{Q}(z_Q+(\diam Q)e_i),L).\end{equation} Then let $B_Q:\RR^n\rightarrow\RR^N$ be the unique affine map such that $B_Q(z_Q+(\diam Q)e_i)=u_i$ for all $i=0,\dots,n$. Note that $B_Q(\RR^n)\subset L$. Set $V=\{v_0,\dots,v_n\}$, where each $v_i:=z_Q+(\diam Q)e_i$. Then $\diam V=\sqrt{2}\diam Q$ and $\Psi(V)=2^{n/2}n!$ (see \ref{e:Psi}). Thus, combining (\ref{e:s2c}) and (\ref{e:s2d}), we observe that \begin{equation*} |A_Q(v)-B_Q(v)|\leq \left(\frac{\frac{1}{2}\beta_s\Anorm{y,s}s}{\diam V}\right)\diam V\quad\text{for all }v\in V.\end{equation*} Therefore, by Lemma \ref{l:A-B}, \begin{equation*} |A_Q(z)-B_Q(z)|\leq \left(\frac{\frac12\beta_s\Anorm{y,s}s}{\diam V}\right)\left(\diam V+2^{3/2}(2n)^{(n+1)/2}\dist(z,V)\right).\end{equation*} Since $\dist (z,V)\leq |z-z_Q| \leq 6\diam Q=(6/\sqrt{2})\diam V$ by (\ref{e:zzq}), it follows that \begin{equation}\label{e:s2e}\dist(A_Q(z),L)\leq |A_Q(z)-B_Q(z)| \lesssim_n \beta_s\Anorm{y,s} s\end{equation} for all $Q\in\mathcal{W}$ such that $z\in 2Q$. Together (\ref{e:Fdef}) and (\ref{e:s2e}) yield \begin{equation}\label{e:s2f} \dist(F(z),L) \lesssim_n \beta_s\Anorm{y,s}s\quad\text{for all }z\in B^n(y,\lambda\mu s)\setminus E.\end{equation} Combining (\ref{e:s2b}) and (\ref{e:s2f}), we conclude that \begin{equation*} \dist(q,L)\lesssim_n \beta_s\|A'_{y,s}\|s\quad\text{for all }q\in F(\RR^n)\cap B^N\left(F(y),\frac{1}{4}\Anorm{y,\lambda\mu s}\lambda\mu s\right).\end{equation*} Using the fact that $\Anorm{y,\lambda \mu s}\sim_n\Anorm{y,s}$ (by Lemma \ref{l:main-est} (\ref{e:pre-d})), it follows that \begin{equation} \label{e:s2g} \beta_{F(\RR^n)}\left(F(y),\frac{1}{4}\Anorm{y,\lambda \mu s}\lambda \mu s\right) \lesssim_n \frac{\beta_s\Anorm{y,s} s}{\frac{1}{4}\Anorm{y,\lambda \mu s}\lambda \mu s}\lesssim_n \beta_{f(\RR^n)}\left(f(y),\frac{1}{4}\Anorm{y,s}s\right)
\end{equation} for all $y\in E$ and $0<s\leq \diam E$ where $\lambda\mu\in(0,1)$ depends only on $n$.

We now adjust the scales in (\ref{e:s2g}) so that they are compatible with Lemma \ref{l:bflat}. First, by (\ref{e:for}) and (\ref{e:fir}), \begin{equation*} \frac{1}{8}|f(y+se_1)-f(y)| \leq \frac{1}{4}\Anorm{y,s}s \leq \frac{1}{2}|f(y+se_1)-f(y)|.\end{equation*} Hence, by (\ref{e:mono}), \begin{equation} \label{e:s2h} \beta_{f(\RR^n)}\left(f(y),\frac{1}{4}\Anorm{y,s}s\right)\lesssim \beta_{f(\RR^n)}\left(f(y),\frac{1}{2}|f(y+se_1)-f(y)|\right). \end{equation} Second, note that $\Anorm{y,\lambda\mu s}\lambda \mu s \sim_n \Anorm{y,s} s$. By \eqref{e:for} and \eqref{e:fir}, it follows that \begin{equation*} \frac{1}{4}\Anorm{y,\lambda\mu s}\lambda\mu s \sim_n \Anorm{y,s}s\sim |f(y+se_1)-f(y)|.
\end{equation*} Hence, there exists a constant $\omega>0$ depending only on $n$ such that \begin{equation} \label{e:s2i} \beta_{F(\RR^n)}(F(y),\omega |f(y+se_1)-f(y)|) \lesssim_n \beta_{F(\RR^n)}\left(F(y),\frac{1}{4}\Anorm{y,\lambda \mu s}\lambda \mu s\right)\end{equation} by (\ref{e:mono}). Combining (\ref{e:s2g}), (\ref{e:s2h}), and (\ref{e:s2i}), we obtain \begin{equation*} \beta_{F(\RR^n)}\left(F(y),\omega|f(y+se_1)-f(y)|\right) \lesssim_n \beta_{f(\RR^n)}\left(f(y),\frac{1}{2}|f(y+se_1)-f(y)|\right).\end{equation*} Therefore, by Lemma \ref{l:bflat}, \begin{equation}\label{e:s2j} \beta_{F(\RR^n)}\left(F(y),\omega|f(y+se_1)-f(y)|\right) \lesssim_{n,N} H_f(B^N(y,2s))\end{equation} for all $y\in E$ and $0<s\leq \diam E$.

Observe that $f$ is a quasiconformal map on the open ball $\oB^N(x,3r)$ with maximal dilatation $K_f(\oB^N(x,3r))\leq H_f(B^N(x,3r))^{N-1}\leq \varepsilon_*^{N-1}\lesssim_N 1$. By replicating the proof of Corollary \ref{c:bflat} with \eqref{e:s2j} instead of \eqref{l:bflat}, one can show that there exists a constant $C=C(n,N)>1$ so that \begin{equation} \label{e:s2k} \int_0^{\diam f(B^N(y,\frac12\diam E))/C} \beta_{F(\RR^n)}(F(y),s)^2\frac{ds}{s}\leq C \int_0^{\frac12\diam E} \tH_f(B^N(y,s))^2\frac{ds}{s}.\end{equation} Note that $\diam f(B^n(y,\frac12\diam E)) \leq \diam f(B^N(y,\frac12\diam E))$ and $\frac12\diam E\leq r$. Hence \begin{equation} \label{e:s2l} \int_0^{\diam f(B^n(y,\frac12\diam E))/C} \beta_{F(\RR^n)}(F(y),s)^2\frac{ds}{s}\leq C \int_0^{r} \tH_f(B^N(y,s))^2\frac{ds}{s}\leq CC_E.\end{equation} by (\ref{e:beta-h2}). Therefore, in view of (\ref{e:ssim}), there is $\tau\sim 1/C>0$ depending only on $n$ and $N$ so that \begin{equation*}\int_0^{\tau s_*} \beta_{F(\RR^n)}(F(y),s)^2\frac{ds}{s} \lesssim_{n,N} C_E.\end{equation*} Incorporating the estimate from Step 1, we conclude that for all $y\in E$, \begin{equation*}\begin{split} \int_0^\infty \beta_{F(\RR^n)}(F(y),s)^2\frac{ds}{s} &\leq  \int_0^{\tau s_*} \beta_{F(\RR^n)}(F(y),s)^2\frac{ds}{s}+\int_{\tau s_*}^\infty \beta_{F(\RR^n)}(F(y),s)^2\frac{ds}{s}\\
 &\lesssim_{n,N} C_E + \left(\frac{\varepsilon}{\tau}\right)^2 \lesssim_{n,N} C_E+\varepsilon^2,\end{split}\end{equation*} as desired. This completes Step 2.

\begin{step} For all $y\in\RR^n\setminus E$, $\int_{0}^{\infty} \beta_{F(\RR^n)}(F(y),s)^2 s^{-1}ds \lesssim_{n,N} C_E+\varepsilon^2$. \end{step}

We exploit the fact that $F$ is smooth far away from $E$. Let $y\in \RR^n\setminus E$ and let $0<s<d(y)/2$. Then $A'_{y,s}=DF(y)=A'_{y,t}$ for all $t<d(y)/2$. Let $a(y):=\|DF(y)\|$. By Lemma \ref{l:Omega}, \begin{equation*}\sup_{z\in B(y,s)} |F(z)-A_{y,s}(z)| \lesssim_n \varepsilon \frac{s}{d(y)}\Anorm{y,s}s= \varepsilon \frac{s}{d(y)}a(y)s.\end{equation*} Hence, since $F(\RR^n)\cap B^N(F(y),\frac14 a(y)s)\subset F(B^n(y,s))$ by (\ref{e:Firb}), we obtain \begin{equation*} \beta_{F(\RR^n)}\left(F(y),\frac{1}{4}a(y)s\right) \lesssim_n \frac{\varepsilon}{d(y)}s\quad\text{for all }s<d(y)/2.\end{equation*}
Thus, \begin{equation}\begin{split}\label{e:s3a}
\int_0^{\frac{1}{8} a(y)d(y)}\beta_{F(\RR^n)}(F(y),s)^2\frac{ds}{s} &=\int_{0}^{\frac12d(y)} \beta_{F(\RR^n)}\left(F(y),\frac{1}{4}a(y)s\right)^2\frac{ds}{s}\\
&\lesssim_n \left(\frac{\varepsilon}{d(y)}\right)^2\int_0^{\frac{1}{2}d(y)} s\,ds\lesssim \varepsilon^2.\end{split}\end{equation}
On the other hand, writing $\delta=|F(y)-F(y')|$ where $y'\in E$ is a point satisfying $|y-y'|=d(y)$, we have $B^N(F(y),t)\subset B^N(F(y'),t+\delta)$ for all $t>0$. Fix $\sigma>0$ to be chosen later. Then \begin{equation*}B^N(F(y),t)\subset B^N(F(y'),t+\delta)\subset B^N\left(F(y'), \left(1+\frac{1}{\sigma}\right)t\right)\quad\text{for all }t\geq \sigma\delta.\end{equation*} Hence, by (\ref{e:monob}), \begin{equation*} \beta_{F(\RR^n)}(F(y),t)\leq \left(1+\frac{1}{\sigma}\right)\beta_{F(\RR^n)}\left(F(y'),\left(1+\frac{1}{\sigma}\right)t\right)\quad\text{for all }t\geq \sigma\delta.\end{equation*} Therefore,  \begin{equation}\begin{split}\label{e:s3b} \int_{\sigma\delta}^\infty \beta_{F(\RR^n)}(F(y),t)^2\frac{dt}{t} &\leq \left(1+\frac{1}{\sigma}\right)^2 \int_{\sigma\delta}^\infty \beta_{F(\RR^n)}\left(F(y'),\left(1+\frac{1}{\sigma}\right)t\right)^2\frac{dt}{t} \\ &\lesssim_{n,N} \left(1+\frac{1}{\sigma}\right)^2 \left(C_E+\varepsilon^2\right),\end{split}\end{equation} where the last inequality holds by Step 2. Next, note that $\delta\sim \Anorm{y,d(y)}d(y) \sim \frac{1}{8}a(y)d(y)$ where the first comparison holds by (\ref{e:Fir}) and the second comparison holds by Lemma \ref{l:main-est} (\ref{e:pre-d}) since $\mathcal{A}$ is $\varepsilon$-compatible for some $\varepsilon<1$. Choose $\sigma>0$ sufficiently small so that $\sigma\delta \leq \frac{1}{8}d(y)a(y)$. Then, combining (\ref{e:s3a}) and (\ref{e:s3b}), we obtain \begin{equation*}\begin{split} \int_0^\infty \beta_{F(\RR^n)}(F(y),s)^2\frac{ds}{s} &\leq  \int_0^{\frac{1}{8}a(y)d(y)} \beta_{F(\RR^n)}(F(y),s)^2\frac{ds}{s}+\int_{\sigma\delta }^\infty \beta_{F(\RR^n)}(F(y),s)^2\frac{ds}{s}\\
 &\lesssim_{n,N} \varepsilon^2 + \left(1+\frac{1}{\sigma}\right)^2 \left(C_E+\varepsilon^2\right) \lesssim_{n,N} C_E+\varepsilon^2\end{split}\end{equation*}
for all $y\in \RR^n\setminus E$. This completes Step 3 and the proof of Theorem \ref{t:beta} / Theorem \ref{t:beta2}. \end{proof}


\bibliography{2014-abt-ref}{}
\bibliographystyle{amsalpha}

\end{document}